\documentclass[a4paper,12pt,openany]{amsbook}
\usepackage[ngerman,english]{babel} 

\let\tempsection=\section
\let\subsection=\tempsection
\let\section=\chapter

\usepackage[a4paper,top=2.5cm,bottom=3cm,left=3cm,right=3cm]{geometry}

\usepackage{amssymb}
\usepackage[shortalphabetic,abbrev]{amsrefs}

\usepackage{upgreek} 
\usepackage{bm}      

\usepackage[all]{xy}
\usepackage[mathcal]{euscript} 

\usepackage[no-mario]{contents/mynotations}

\newcommand{\ctrct}{\lrcorner\,}
\newcommand{\fdir}{[u]}
\newcommand{\sfdir}{{\scriptscriptstyle\fdir}}
\newcommand{\ethf}{\vartheta_{\sfdir}}
\newcommand{\difff}{\diff_{\sfdir}}
\newcommand{\dbarf}{\dbar_{\sfdir}}
\newcommand{\bdir}{[v]}
\newcommand{\sbdir}{{\scriptscriptstyle\bdir}}
\newcommand{\ethb}{\vartheta_{\sbdir}}
\newcommand{\diffb}{\diff_{\sbdir}}
\newcommand{\dbarb}{\dbar_{\sbdir}}
\newcommand{\ddbarb}{\diffb\dbarb}
\newcommand{\iddbarb}{\cplxi\ddbarb}

\newcommand{\dksh}[1]{\dbar_{\kashf{}_{#1}}}
\newcommand{\dperp}[1]{\eth^*_{\bot#1}}
\newcommand{\nablaol}{\nabla^{(0,1)}}
\newcommand{\nablalo}{\nabla^{(1,0)}}

\newcommand{\Tgt}[1]{\mathbf T^{#1}}
\newcommand{\Tgtf}[1]{\Tgt{#1}_u}
\newcommand{\Tgtb}[1]{\Tgt{#1}_v}
\newcommand{\cTgt}[1]{\Tgt{*\,#1}}
\newcommand{\cTgtf}[1]{\cTgt{#1}_u}
\newcommand{\cTgtb}[1]{\cTgt{#1}_v}
\newcommand{\kashf}[1]{\sheaf H^{#1}}
\newcommand{\kashfL}[1]{\kashf{#1}_{L^2}}
\newcommand{\smcform}[1]{\sheaf A_{#1}}
\newcommand{\smform}[1]{\sheaf A^{#1}}

\renewcommand{\smooth}{\smform{}}
\newcommand{\smbform}[1]{\smform{0,(#1)}}
\newcommand{\smcfb}[3][0]{\smcform{(#2,#3),#1}}
\newcommand{\smFfb}[4]{\smform{(#1,#2),(#3,#4)}}
\newcommand{\smcFfb}[4]{\smcform{(#1,#2),(#3,#4)}}
\newcommand{\smfb}[3][0]{\smform{#1,(#2,#3)}}

\newcommand{\smFB}[3]{\smform{0,#1}_{#2}(#3;L)}
\newcommand{\smFBtwo}[1][]{\smFB{q}{#1<2>}{K_c}}
\newcommand{\smfbtwo}[1][]{\smform{0,q}_{#1<2>}}
\newcommand{\smFBXtwo}[1][]{\smFB{q}{#1<2>}{X}}

\newcommand{\smFBCtwo}[1][q]{\smFB{#1}{<2>}{\cl K_c}}
\newcommand{\smFBCXtwo}{\smFBXtwo[0\,]}
\newcommand{\smFBthree}[1][]{\smFB{q+1}{#1<3>}{K_c}}

\newcommand{\Ltwo}[1]{{L_{\!2}\,}^{\!\!\!#1}}
\newcommand{\hilbsp}[1]{\mathfrak H^{#1}}
\newcommand{\hilbbsp}[1]{\Ltwo{0,(#1)}}

\newcommand{\hilbFB}[2]{\Ltwo{0,(#1,#2)}}
\newcommand{\hilbFBtwo}[1][]{\Ltwo{0,q}_{c#1\,<2>}}
\newcommand{\hFBtwo}[1][]{\Ltwo{0,q}_{#1\,<2>}}
\newcommand{\hilbFBtwoLong}[1][]{\Ltwo{0,q}_{<2>}(K_{c#1};L)}
\newcommand{\hilbFBthree}[1][]{\Ltwo{0,q+1}_{c#1\,<3>}}
\newcommand{\tamesub}{{\mathsf t}}
\newcommand{\wildsub}{{\mathsf w}}
\newcommand{\etat}{\eta_\tamesub}
\newcommand{\etaw}{\eta_\wildsub}

\newcommand{\tCurv}{\mathfrak T}
\newcommand{\wCurv}{\mathfrak W}
\newcommand{\wtCurv}{\mathfrak{wt}}
\newcommand{\chiR}{\widetilde \chi}
\newcommand{\chiX}{\chi}
\newcommand{\oTheta}{\mathcal R}
\newcommand{\pT}{\tilde p}
\newcommand{\Hform}{\mathcal H}
\newcommand{\Hbot}{H_E}
\newcommand{\altH}{\widetilde \Hform}
\newcommand{\stsum}{\sideset{}{'}\sum}
\newcommand{\dual}{\vee}
\DeclareMathOperator{\Bd}{Bd}

\newtheorem{prop}{Proposition}[section]
\newtheorem{lemma}[prop]{Lemma}
\newtheorem{thm}[prop]{Theorem}
\newtheorem{cor}[prop]{Corollary}

\theoremstyle{remark}
\newtheorem{remark}[prop]{Remark}

\theoremstyle{definition}
\newtheorem{definition}[prop]{Definition}

\numberwithin{equation}{chapter}

\usepackage{multicol} 

\usepackage[noprefix,refpage]{nomencl} 
\newcommand{\showsym}[3][]{\nomenclature[#1]{#2}{\dotfill}}

\makenomenclature

\allowdisplaybreaks

\begin{document}

\title{The Index Theorem for quasi-Tori}


\author[Mario Chan]{Tsz On Mario Chan}
\address{Lehrstuhl Mathematik VIII \\ 
  Mathematisches Institut der Universit\"at Bayreuth \\ 
  NW II, Universit\"atstr.~30 \\ 
  95447 Bayreuth}
\email{mariocto@gmail.com}
\thanks{The author would like to thank Prof.~F.~Catanese for
  suggesting this problem and for his guidance. 
  He would also like to thank Prof.~P.~Eyssidieux for pointing out the
  mistakes in the previous dissertation version of this paper.
  The author is supported by DFG Forschergruppe 790
  ``Classification of algebraic surfaces and compact complex
  manifolds''.
}






\pagenumbering{roman}

%
%

\makeatletter

\newcommand{\@degree}{Doktorgrades (Dr. rer. nat.)}
\newcommand{\@university}{Universit\"at Bayreuth}
\newcommand{\@faculty}{Fakult\"at f\"ur Mathematik, Physik und Informatik}
\newcommand{\@address}{Bayreuth}
\newcommand{\@submitdate}{27.~November, 2012}
\newcommand{\@defencedate}{15.~Februar, 2013}
\newcommand{\authorOrigin}{Hong Kong}

\ifx\authors\@empty
\author{\hfill}
\fi

\begin{titlepage}
  \renewcommand{\baselinestretch}{1.37}
  \thispagestyle{empty}

  \null\vskip0.5in
  \begin{center}
    \hyphenpenalty=10000\Large\uppercase\expandafter{\@title}
  \end{center}
  
  \vfill
  
  \begin{center}
    \footnotesize
    DISSERTATION 

    zur Erlangung 

    des \uppercase\expandafter{\@degree}  

    der \uppercase\expandafter{\@faculty}
    
    der \uppercase\expandafter{\@university}
  \end{center}
  
  \vfill
  \begin{center}
    vorgelegt von

    \vspace{1.25cm}

    \uppercase\expandafter{\authors} 
    
    aus \authorOrigin
  \end{center} 
  
  \vfill
  \begin{center}
    {
      \renewcommand{\labelenumi}{\arabic{enumi}.}
      \parbox{0.65\textwidth}{
        \begin{enumerate}
        \item 
          Gutachter: Prof.~Dr.~Fabrizio Catanese

        \item 
          Gutachter: Prof.~Dr.~Philippe Eyssidieux

        \item 
          Gutachter: Prof.~Dr.~Ngaiming Mok

        \end{enumerate}
      }
    }
  \end{center}
  
  \vfill
  \begin{center}
    \uppercase\expandafter{\@address} 
    
    Tag der Einreichung: \@submitdate 

    \ifx\@defencedate\undefined
    \else
    Tag der Kolloquiums: \@defencedate
    \fi
  \end{center}
\end{titlepage}


\makeatother


\begin{otherlanguage}{ngerman}
  \chapter*{Erkl\"arung}
%
%

Ich versichere eidesstattlich, dass ich diese Arbeit selbst\"andig verfasst habe, und ich keine anderen als die
von mir angegebenen Quellen und Hilfsmittel benutzt habe. 

Ich best\"atige, dass Hilfe von gewerblichen Promotionsberatern bzw.~-vermittlern oder
\"ahnlichen Dienstleistern weder bisher in Anspruch genommen wurde noch k\"unftig in
Anspruch genommen wird.

Ich best\"atige, dass ich keine
fr\"uhere Promotionsversuche gemacht habe.

\vspace{3.25cm}

\noindent \hfill
\parbox[t]{0.5\textwidth}{\centering \hrulefill \\
  Unterschrift des Autors}


\end{otherlanguage}

\chapter*{Acknowledgements}
%
%

It is my pleasure to express here my gratitude to my supervisor
Prof.~Fabrizio Catanese for suggesting me this research problem and
for his continual guidance, as well as sharing his point of view
about Mathematics and a lot of his personal experience in life. 

My gratitude also goes to Prof.~Ingrid Bauer for encouraging me to
explore different fields of Mathematics. Moreover, her care to me
during my sickness made me feel like home while I was staying in a
country distant from mine.

Many thanks to all current and former colleagues in the Lehrstuhl Mathematik VIII
of Universit\"at Bayreuth, in particular to Michael L\"onne, Fabio
Perroni, Masaaki Murakami, Stephen Coughlan, Matteo Penegini, Wenfei
Liu and Yifan Chen, for their help on my thesis, inspiring discussions
on Mathematical ideas, sharing about the cultures and lifestyles of their
own countries, and, most importantly, their encouragements which
helped me to get through the most depressing period of my Ph.D.
study. Thanks also to our secretary Leni Rostock who helped to sort out all the
troubles during my stay in Bayreuth, from getting the residence permit
to finding a medical doctor. Thanks to her, we have never missed the
birthday of anybody in Lehrstuhl VIII. Wish that she would enjoy her
life after retirement.

Special thanks to my M.Phil.~supervisor Prof.~Ngaiming Mok, who taught
me the basics about the Bochner--Kodaira formulas; and to Michael
L\"onne, Florian Schrack, Sascha Weigl and Christian Glei\ss ner who
helped me to translate the abstract and summary into German.

I would also like to thank DAAD for their support under the
\foreignlanguage{ngerman}{Forschungsstipendien f\"ur Doktoranden}. 

Lastly, I would like to declare that I owe my friends outside the
Mathematics community in both Hong Kong and Germany a
lot. Without their comforts and encouragements, this thesis could
never be finished. My debts to them can never be fully redeemed. I am
also badly indebted to my parents, who have given me freedom to do whatever
I wish.


\chapter*{Abstract}
\label{abstract}


The Index theorem for holomorphic line bundles on complex tori asserts that
some cohomology groups of a line bundle vanish according to the
signature of the associated hermitian form.
In this article, this theorem is generalized to quasi-tori,
i.e.~connected complex abelian Lie groups which are not necessarily
compact. In view of the Remmert--Morimoto decomposition of quasi-tori
as well as the K\"unneth formula, it suffices to consider only 
Cousin-quasi-tori, i.e.~quasi-tori which have no non-constant
holomorphic functions.
The Index theorem is generalized to holomorphic line bundles, both
linearizable and non-linearizable, on Cousin-quasi-tori using
$L^2$-methods coupled with the Kazama--Dolbeault isomorphism and
Bochner--Kodaira formulas. 



\begin{otherlanguage}{ngerman}
%
%

\chapter*{Zusammenfassung}

Ein \emph{Quasi-Torus} ist eine zusammenh\"angende komplexe abelsche Lie-Gruppe $X = \fieldC^n / \Gamma$,
wobei $\Gamma$ eine diskrete Untergruppe von $\fieldC^n$ ist. $X$ hei\ss t
\emph{Cou\-sin-Qua\-si-To\-rus}, wenn alle holomorphen Funktionen auf $X$ konstant sind. Ist $X$ kompakt, so ist $X$ ein \emph{komplexer Torus}. 

Nach einem Satz von Remmert und Morimoto (vgl.~\cite{Morimoto} oder
\cite{Cap&Cat}*{Prop.~1.1})
gibt es f\"ur jeden Quasi-Torus $X$ eine Zerlegung $X \cong \fieldC^a \times
(\fieldC^*)^b \times X'$, wobei $X'$ ein Cousin-Quasi-Torus ist. Das Ziel des vorliegenden Artikels ist, das Verschwinden von Kohomologiegruppen von Geradenb\"undeln auf $X$ zu untersuchen. Die K\"unnethformel (vgl.~\cite{Kaup}) besagt, dass sich die Kohomologiegruppen von $X$ in direkte Summen von topologischen Tensorprodukten von Kohomologiegruppen von $\fieldC^a \times (\fieldC^*)^b$ und des Cousin-Quasi-Torus $X'$ zerlegen lassen. Man wird dadurch auf den Fall gef\"uhrt, dass $X$ ein Cousin-Quasi-Torus ist, da $\fieldC^a \times (\fieldC^*)^b$ Steinsch ist und somit alle h\"oheren Kohomologiegruppen (mit Grad $\geq 1$) von koh\"arenten Garben verschwinden. Es wird also im vorliegenden Artikel angenommen, dass $X$ ein Cousin-Quasi-Torus ist.

Sei $F$ der maximale komplexe Unterraum von $\fieldR\Gamma$ und $m := \dim_\fieldC F$. Wie im kompakten Fall kann jedem holomorphen Geradenb\"undel $L$ eine hermitesche Form $\Hform$ auf $\fieldC^n$ zugeordnet werden, deren Imagin\"arteil $\Im \Hform$ mit der ersten Chernklasse $c_1(L)$ von $L$ assoziiert ist und ganzzahlige Werte in $\Gamma \times \Gamma$ annimmt. Im Unterschied zum kompakten Fall ist $\Hform$ nicht eindeutig.
Lediglich die Einschr\"ankung von $\Im \Hform$ auf $\fieldR\Gamma
\times \fieldR\Gamma$, und somit $\Hform|_{F\times F}$, ist eindeutig bestimmt. Dies macht zumindest plausibel, dass nur $\Hform|_{F\times
  F}$ anstelle von $\Hform$ f\"ur die Eigenschaften von $L$ verantwortlich ist. Die vorliegende Dissertation widmet sich dem Beweis des folgenden Satzes:
{
\theoremstyle{plain}
\newtheorem*{idxthm-ger}{Index-Satz f\"ur Cousin-Quasi-Tori}
\begin{idxthm-ger}
  Sei $X = \fieldC^n / \Gamma$ ein Cousin-Quasi-Torus, $F$ der
  maximale komplexe Unterraum von $\fieldR\Gamma$, $L$ ein holomorphes
  Geradenb\"undel auf $X$ und $\Hform$ eine mit $L$ assoziierte
  hermitesche Form auf $\fieldC^n \times \fieldC^n$. Sei $m := \dim_\fieldC F$. Die Einschr\"ankung
  $\Hform|_{F\times F}$ habe $s_F^-$ negative und $s_F^+$
  positive Eigenwerte. Dann gilt
  \begin{equation*}
    H^q(X,L) = 0 \quad\text{ f\"ur }\; q < s_F^- \;\text{ oder }\; q > m-s_F^+ \; . 
  \end{equation*}
\end{idxthm-ger}
}
Dieser Satz wird zur\"uckgef\"uhrt auf den Index-Satz f\"ur komplexe Tori, wie er von Mumford \cite{Mumford}, Kempf \cite{Kempf}, Umemura \cite{Umemura},
Matsushima \cite{Matsushima} und Murakami \cite{Murakami} f\"ur kompakte $X$ bewiesen wurde. Da $X$ stark $(m+1)$-vollst\"andig ist
(vgl.~\cite{Kazama_pseudoconvex}; siehe auch \S
\ref{sec:pseudoconvexity}), enth\"alt der Satz auch einen Spezialfall des Resultats von Andreotti und Grauert, das besagt, dass $H^q(X, \sheaf
F) = 0$ ist f\"ur alle $q \geq m+1$ und f\"ur jede koh\"arente analytische Garbe $\sheaf F$ auf $X$ (vgl.~\cite{Andre&Grauert}).

Das Verschwinden von $H^q(X,L)$ kann unter Verwendung der
Dol\-beault-I\-so\-mor\-phis\-men auf gewisse $\dbar$-Glei\-chun\-gen
f\"ur $L$-wertige $(0,q)$-Formen zurückgef\"uhrt werden. Diese
k\"onnen mit $L^2$-Methoden gel\"ost werden. Man zeigt zun\"achst die
Existenz einer formalen L\"osung einer $\dbar$-Gleichung in einem
Hilbertraum, indem man die ben\"otigte $L^2$-Absch\"atzung nachweist,
und beweist dann die Glattheit der L\"osung. Letzteres kann mit Hilfe
der Regularit\"atstheorie von $\dbar$-Operatoren erledigt werden, also
ist der entscheidende Schritt der Nachweis der be\-n\"o\-tig\-ten
$L^2$-Absch\"atzungen. Diese kann man durch Anwendung der
Boch\-ner--Ko\-dai\-ra-Un\-glei\-chun\-gen bekommen.

Jeder Cousin Quasi-Torus $X$ hat eine Faserb\"undelstruktur \"uber einem komplexen Torus $T$ 
mit steinschen Fasern (siehe \S \ref{sec:fibre-bundle} und (\ref{eq:fibration})).
Mit Hilfe der Lerayschen Spektralsequenz folgt 
\begin{equation*}
  H^q(X,L) \isom H^q(T,p_*\holo_X(L)) \quad\text{f\"ur alle }q \geq 0 \; ,
\end{equation*}
wobei $p \colon X \to T$ die Projektion aus (\ref{eq:fibration}) ist.
Die Idee ist jetzt zu zeigen, dass der Dolbeault Komplex der Garben 
$\paren{\smform{0,\bullet}_T \otimes_{\holo_T} p_*\holo_X(L), \dbar}$,
eine azyklische Aufl\"osung von  $p_*\holo_X(L)$ auf  $T$ ist
und das Verschwinden der Kohomologie durch L\"osen der $\dbar$-Gleichungen zu zeigen.
Kazama \cite{Kazama} und
Kazama--Umeno \cite{Kazama_Dolbeault-isom} geben eine leicht
ver\"anderte Formulierung, sie betrachten die Aufl\"osung 
von  $\holo_X(L)$ durch einen Unterkomplex
$\paren{\kashf{0,\bullet}(L), \dbar}$ von 
$\paren{\smform{0,\bullet}_X(L), \dbar}$
(siehe \S \ref{sec:kazama_resolution} f\"ur die Definition von $\kashf{0,q}(L)$).
Der Teilkomplex ist ebenfalls eine azyklische Aufl\"osung von $\holo_X(L)$ auf $X$
und liefert damit den Kazama--Dolbeault Isomorphismus
(vgl.~\cite{Kazama_Dolbeault-isom}, siehe auch Theorem 
\ref{rem:Kazama-Dolbeault}).
Letzterer Ansatz wird hier aufgegriffen.
Das Ziel der Darstellung ist dann die L\"osung der $\dbar$-Gleichung $\dbar\xi = \psi$ f\"ur ein gegebenes $\psi \in \Gamma(X, \kashf{0,q}(L))$ mit $\dbar\psi = 0$.

Jedes Geradenb\"undel $L$ auf $X$ kann durch ein System von Automorphiefaktoren definiert werden, die in eine zur Appell--Humbert-Normalform analoge Normalform \"ubergef\"uhrt werden k\"onnen, die gegeben ist durch (vgl.~\cite{Cap&Cat}*{\S 2.2} und \cite{Vogt}*{\S 2})
\begin{equation*}
  \varrho(\gamma)e^{\pi \Hform(z,\gamma) + \frac{\pi}{2}\Hform(\gamma,\gamma)+
    f_\gamma(z)} \qquad \forall \: \gamma\in\Gamma \; ,
\end{equation*}
wobei $\varrho$ ein Halbcharakter auf $\Gamma$ und
$\set{f_\gamma(z)}_{\gamma\in\Gamma}$ ein additiver Kozykel ist
(vgl. \cite{Cap&Cat}*{\S 2.2} und \cite{Vogt}*{\S 2}, siehe auch
(\ref{eq:f-gamma-properties})). Wenn
$\set{f_\gamma(z)}_{\gamma\in\Gamma}$ ein Korand ist, so wird $L$ als \emph{linearisierbar} bezeichnet; andernfalls als
\emph{nicht linearisierbar}. Indem man den Trick verwendet, den Murakami in \cite{Murakami} f\"ur den kompakten Fall benutzt hat (siehe \S
\ref{sec:curv-terms-murakami-trick}), n\"amlich die Metrik $g$ so abzu\"andern, dass der vom linearen Teil (dem zahmen Teil) von $L$ in den Basisrichtungen kommende Kr\"ummungsterm von unten beschr\"ankt ist, wenn $q$ im gegebenen Bereich liegt, kann man die ben\"otigten $L^2$-Absch\"atzungen erhalten, wenn $L$ linearisierbar ist (siehe \S \ref{sec:proof-linearizable}). Dies beweist den Index-Satz f\"ur
linearisierbare $L$ (siehe Theorem \ref{thm:proof-linearizable}).

Beim Nachweis der ben\"otigten $L^2$-Absch\"atzungen f\"ur nicht linearisierbare $L$ auf $X$ gibt eine zus\"atzliche technische Schwierigkeit, die von dem vom nichtlinearen Teil (dem wilden Teil) von $L$ kommenden Kr\"ummungsterm herr\"uhrt.
F\"ur diesen wird Takayama's schwaches $\diff\dbar$-Lemma (\cite{Takayama}*{Lemma
  3.14}; siehe auch \S \ref{sec:bound-wCurv}) 
angewandt, um den Term auf relativ kompakten Teilmengen von $X$ zu beschr\"anken.
Dadurch erh\"alt man die ben\"otigten $L^2$-Absch\"atzungen nicht auf $X$, sondern lediglich auf der aussch\"opfenden Familie $\seq{K_c}_{c \in \fieldR_{> 0}}$ von pseudokonvexen relativ kompakten Teilmengen. Man erh\"alt dann eine Folge $\seq{\xi_\nu}_{\nu \geq 1}$ von lokalen L\"osungen, so dass
$\dbar \xi_\nu = \psi|_{\cl K_\nu}$ ist f\"ur ein gegebenes $\psi \in \Gamma(X,\kashf{0,q}(L)) \cap \ker \dbar$ und f\"ur alle
ganzen Zahlen $\nu \geq 1$. Indem man ein Argument im Beweis von Theorem B f\"ur Steinsche R\"aume in \cite{Grauert&Remmert}*{Ch.~IV, \S 5} nachvollzieht,
speziell indem man eine Approximation vom Runge-Typ verwendet, kann
man die lokalen L\"osungen $\xi_\nu$ so korrigieren, dass sie auf
jedem $K_c$ konvergieren, was dann eine globale L\"osung f\"ur alle
$q$ im gegebenen Bereich liefert (siehe \S
\ref{sec:proof-special-q}). Der Beweis des Index-Satzes ist damit
vollst\"andig. 


\end{otherlanguage}


\tableofcontents


 \pagenumbering{arabic}
 \setcounter{page}{1}

\section{Introduction and the main theorem}
\label{sec:intro}

%
%

A \emph{quasi-torus} is a complex abelian Lie group $X = \fieldC^n / \Gamma$,
where $\Gamma$ is a discrete subgroup of $\fieldC^n$. 
$X$ is said to be a \emph{Cousin-quasi-torus} if all holomorphic functions on $X$ are
constant functions.\footnote{A Cousin-quasi-torus is also called a
\emph{toroidal group} or \emph{$(H,C)$-group} in literature, where the latter means
that all \underline{h}olomorphic functions are \underline{c}onstant
(ref.~\cite{Abe&Kopfermann}*{Def.~1.1.1}).}
$X$ is the familiar
\emph{complex torus} when it is compact, i.e.~when $\rk \Gamma = 2n$. 

The study of quasi-tori dates back to the early 20th century when
Cousin studied the triply periodic functions of two complex
variables (\cite{Cousin}). There he showed the existence of
2-dimensional quasi-tori without non-constant holomorphic
functions. He also gave, among other things, a complete description
of holomorphic line bundles on quasi-tori of
dimension 2 and their sections using a method of asymptotic counting of zeros of the sections.
In the 60's, Kopfermann (\cite{Kopfermann}) studied systematically
toroidal groups of arbitrary dimensions with a view to generalize the
theory of abelian functions on complex tori. He also gave an example
of a non-compact toroidal group with no non-constant meromorphic
functions. Morimoto (\cite{Morimoto1} and \cite{Morimoto}) studied
Cousin-quasi-torus as the maximal toroidal subgroup of a complex (not
necessarily abelian) Lie group, aiming to classify non-compact complex
Lie groups. He classified all 3-dimensional abelian complex Lie
groups. In the early 70's, Andreotti and Gherardelli gave seminars on
quasi-abelian varieties, i.e.~Cousin-quasi-tori which possess
structures of quasi-projective algebraic varieties
(\cite{Andre&Ghera}).
They showed that, among other things, a Cousin-quasi-torus is a
quasi-abelian variety if and only if the Generalized Riemann Relations
are satisfied on it.
Later on, among other contributors, Kazama (\cite{Kazama_pseudoconvex}
and \cite{Kazama}), Pothering (\cite{Pothering}), Hefez
(\cite{Hefez}), Vogt (\cite{Vogt}), Huckleberry and Margulis
(\cite{Huckleberry&Margulis}), Abe (\cite{Abe_manuscripta} and
\cite{Abe_toroidal}), Capocasa and Catanese (\cite{Cap&Cat} and
\cite{Cap&Cat2}), and Takayama (\cite{Takayama}) made some direct
contributions to the theory of quasi-tori and Cousin-quasi-tori. A
brief exposition of the historical development of the Generalized
Riemann Relations can be found in \cite{Cap&Cat}*{p.~29}, and the
Introduction of \cite{Abe&Kopfermann} describes a brief chronology of
the study of toroidal groups in general.

The current research stems from the study of Capocasa and Catanese
(ref.~\cite{Cap&Cat} and \cite{Cap&Cat2}). 
In \cite{Cap&Cat}, they gave an affirmative answer to a long standing
problem of whether the existence of a non-degenerate meromorphic
function on a quasi-torus is equivalent to the Generalized Riemann
Relations. 
In \cite{Cap&Cat2}, they moved on to prove the Lefschetz
type theorems on quasi-tori in the best form, based on a statement of Abe
with an erroneous proof in
\cite{Abe_nagoya}*{Thm.~6.4} (see \cite{Cap&Cat2}*{Corollary
  1.2}).\footnote{Th\'eor\`eme 6.4 in \cite{Abe_nagoya} asserts that,
  on a non-compact toroidal group $X$, there exists a constant $c > 0$
  such that, for any holomorphic line bundle $L$ with an associated
  hermitian form $\Hform$ on $\fieldC^n$ such that $\Hform|_{F\times F} > c I_m$
  (where $I_m$ is the $m\times m$-identity matrix and $F$ is the
  maximal complex subspace of $\fieldR\Gamma$; see \S
  \ref{sec:notations}), $H^0(X, L)$ is non-trivial, and in fact
  infinite-dimensional.}
Abe's statement is then substituted by a result proven by Takayama
(\cite{Takayama1}*{Thm.~1.3 and Thm.~6.1}).\footnote{Theorem 1.3 and
  6.1 in \cite{Takayama1} together asserts that, for any positive line
  bundle $L$ on a non-compact toroidal group $X$, there exists an
  explicitly given integer $\mu_0>0$ such that $H^0(X, L^{\otimes
    \mu})$ is non-trivial for all $\mu \geq \mu_0$. Corollary 1.2 in
  \cite{Cap&Cat2} holds true by applying Takayama's result and
  Proposition 1.1 in \cite{Cap&Cat2}. Takayama also gives a different
  proof of a weaker form of Lefschetz type theorems in
  \cite{Takayama}.} 
These results clarify some basic properties of 
meromorphic functions and global sections of holomorphic line bundles on
quasi-tori. 
This article goes a step further into the
investigation of the higher cohomology groups of holomorphic line
bundles on quasi-tori. The aim is to generalize the
Index theorem on tori to quasi-tori.

\subsection{The main theorem}
\label{sec:the-main-thm}

Denote the $\fieldC$-span and $\fieldR$-span
of $\Gamma$ by $\fieldC\Gamma$ and $\fieldR\Gamma$ respectively. 
Let $\pi \colon \fieldC^n \to X$ be the natural projection. 
Then $K := \pi(\fieldR\Gamma) = \fieldR\Gamma / \Gamma$\showsym[k]{$K$}{$:=
  \fieldR\Gamma / \Gamma$, the maximal compact subgroup of $X$} is the
maximal compact subgroup of $X$, and $F 
:= \fieldR\Gamma\cap\cplxi\fieldR\Gamma$\showsym[f]{$F$}{$:=
  \fieldR\Gamma\cap\cplxi\fieldR\Gamma$, the maximal complex subspace
  in $\fieldR\Gamma$} is the maximal complex subspace in
$\fieldR\Gamma$.

By a theorem of Remmert and Morimoto (ref.~\cite{Morimoto}, see also
\cite{Cap&Cat}*{Prop.~1.1}),
if $X$ is a quasi-torus, there is a decomposition $X \cong \fieldC^a
\times (\fieldC^*)^b \times X'$, where $X'$ is Cousin.
The aim of this article is to investigate the vanishing of cohomology
groups of holomorphic line bundles on $X$.
The K\"unneth formula (ref.~\cite{Kaup})
asserts that the cohomology groups on $X$ decompose into direct sum of
topological tensor products of cohomology groups on $\fieldC^a \times
(\fieldC^*)^b$ and the Cousin-quasi-torus $X'$. 
In view of this, since
$\fieldC^a \times (\fieldC^*)^b$ is Stein and thus all higher
cohomology groups (with degree $\geq 1$) of coherent sheaves vanish,
one is reduced to the case where $X$ is Cousin. 
In what follows,
\emph{$X$ is assumed to be a non-compact Cousin-quasi-torus} unless
otherwise stated.
In this case, $\fieldC\Gamma = \fieldC^n$, and $\rk \Gamma = \dim_\fieldR
\fieldR\Gamma = n + m$ for some integer $m$ such that $0 < m
< n$. 
Note that $m$ is the complex dimension of $F$.

Given a holomorphic line bundle $L$ on $X$, it is analogous to the
compact case that there is a hermitian form $\Hform$ on $\fieldC^n
\times \fieldC^n$ associated to $L$, whose imaginary part $\Im \Hform$
takes integral values on $\Gamma\times\Gamma$ and corresponds to the
first Chern class $c_1(L)$ of $L$
(ref.~\cite{Cap&Cat}).
$\Im\Hform$ is uniquely determined only on $\fieldR\Gamma \times
\fieldR\Gamma$, so $\Hform$ is uniquely determined only on $F \times F$.
The following theorem is a generalization of the Index theorem on
complex tori (ref.~\cite{Mumford}*{p.~150}, \cite{Murakami} or \cite{Bir&Lange}*{\S
  3.4})\footnote{The Index theorem on complex tori was first proven by
  Mumford \cite{Mumford} and Kempf \cite{Kempf} in the algebraic case,
  and later by Umemura \cite{Umemura}, Matsushima \cite{Matsushima}
  and Murakami \cite{Murakami} in the analytic case.} to
Cousin-quasi-tori, which is the main result of this article.
\begin{thm} \label{thm:main_thm}
  Let $X = \fieldC^n / \Gamma$ be a Cousin-quasi-torus, $F$ the
  maximal complex subspace of $\fieldR\Gamma$, $L$ a holomorphic line
  bundle on $X$, and $\Hform$ a hermitian form on $\fieldC^n \times
  \fieldC^n$ associated to $L$.
  Let $m := \dim_\fieldC F$. 
  Suppose $\Hform|_{F\times F}$ has respectively $s_F^-$ negative and
  $s_F^+$ positive eigenvalues\showsym[s]{$s_F^+$ (resp.~$s_F^-$)}{number of positive
    (resp.~negative) eigenvalues \\ of $\Hform"|"_{F\times F}$}.
  Then one has
  \begin{equation*}
    H^q(X,L) = 0 \quad\text{ for }\; q < s_F^- \;\text{ or }\; q > m-s_F^+ \; . 
  \end{equation*}
\end{thm}

Let $\holoform_X^p$ be the sheaf of germs of holomorphic $p$-forms on
$X$, and set $\holoform_X^p(L) := \holoform_X^p \otimes_{\holo_X} \holo_X(L)$.
Since the cotangent bundle of $X$ is trivial, one has
$\holoform_X^p(L) \isom \bigoplus^{\binom np} \holo_X(L)$, and thus
$H^q(X,\holoform_X^p(L)) \isom \bigoplus^{\binom np} H^q(X,L)$.
Therefore, one has the following
\begin{cor}
  With the same assumptions as in Theorem \ref{thm:main_thm}, one has,
  for any $p \geq 0$,
  \begin{equation*}
    H^q(X,\holoform_X^p(L)) = 0 \quad\text{ for }\; q < s_F^- \;\text{
      or }\; q > m-s_F^+ \; . 
  \end{equation*}
\end{cor}

Note that the statement is reduced to the original Index theorem when
$X$ is a compact complex torus, in which case $m = n$. Moreover, it
can be shown that $X$ is strongly $(m+1)$-complete
(ref.~\cite{Kazama_pseudoconvex} and \cite{Takeuchi}; convention of
the numbering here following \cite{Demailly}*{pp.~512}; see also \S
\ref{sec:pseudoconvexity}), so Theorem \ref{thm:main_thm} includes a
special case of the result of Andreotti and Grauert, which asserts
that $H^q(X, \sheaf F) = 0$ for all $q \geq m+1$ and for any coherent
analytic sheaf $\sheaf F$ on $X$ (ref.~\cite{Andre&Grauert}). The
remaining part of this article is devoted to proving Theorem
\ref{thm:main_thm}.

\subsection{Methodology}
\label{sec:methodology}

Let $L$ be a holomorphic line bundle on $X$.
Since every Cousin-quasi-torus $X$ has a fibre bundle structure over a
complex torus $T$ with Stein fibres (see \S \ref{sec:fibre-bundle} and
(\ref{eq:fibration})),
it follows from a Leray spectral sequence argument that
\begin{equation*}
  H^q(X,L) \isom H^q(T,p_*\holo_X(L)) \quad\text{for all }q \geq 0 \; ,
\end{equation*}
where $p \colon X \to T$ is the projection in (\ref{eq:fibration}).
Let $\smform{0,q}_T$ (resp.~$\smform{0,q}_X$) be the sheaf of germs of smooth differential
$(0,q)$-forms on $T$ (resp.~on $X$).
The idea is then to show that the Dolbeault complex of sheaves
$\paren{\smform{0,\bullet}_T \otimes_{\holo_T} p_*\holo_X(L), \dbar}$,
is an acyclic resolution of $p_*\holo_X(L)$ on $T$,
and to prove vanishing by solving $\dbar$-equations.
A slightly different formulation is given by Kazama \cite{Kazama} and
Kazama--Umeno \cite{Kazama_Dolbeault-isom}, who consider the resolution
of $\holo_X(L)$ by a subcomplex
$\paren{\kashf{0,\bullet}(L), \dbar}$ of
$\paren{\smform{0,\bullet}_X(L), \dbar}$
(see \S \ref{sec:kazama_resolution} for the definition of $\kashf{0,q}(L)$).
The subcomplex is also an acyclic resolution of $\holo_X(L)$ on $X$,
thus yielding the Kazama--Dolbeault isomorphism
(ref.~\cite{Kazama_Dolbeault-isom}, see also Theorem
\ref{rem:Kazama-Dolbeault}).
This latter formulation is adopted in this article, so, to
prove Theorem \ref{thm:main_thm} is to solve the $\dbar$-equations
$\dbar \xi = \psi$ for any $\psi \in \Gamma(X,\kashf{0,q}(L))$ such
that $\dbar\psi = 0$ and for all $q$'s in the range given in the Theorem.

The required $\dbar$-equations are solved by exhibiting $L^2$
estimates (\ref{eq:a_priori_estimate}) for certain $L$-valued forms on
$X$.
When $L$ is linearizable (see Definition \ref{def:linearizable}),
these estimates can be obtained from Bochner--Kodaira formulas
together with a trick employed by Murakami for the case of tori
(ref.~\cite{Murakami}) (see \S \ref{sec:curv-terms-murakami-trick} and
\S \ref{sec:proof-linearizable}).

For non-linearizable $L$, the required $L^2$ estimates can only be
obtained on compact subsets of $X$ via Takayama's Weak $\ddbar$-Lemma
(ref.~\cite{Takayama}, see also \S \ref{sec:bound-wCurv}).
Then, given $\psi \in \Gamma(X,\kashf{0,q}(L))$ such that $\dbar\psi =
0$ and an exhaustive sequence $\seq{K_\nu}_{\nu \in \Nnum_{>0}}$ of
pseudoconvex relatively compact open subsets of $X$, a sequence
$\seq{\xi_\nu}_{\nu \in \Nnum_{>0}}$ of weak solutions of
$\dbar\xi_\nu = \psi|_{K_\nu}$ is obtained.
Using a Runge-type approximation (see \S \ref{sec:runge-type-approx})
and following an argument in \cite{Grauert&Remmert}*{Ch.~IV, \S 1, Thm.~7},
the solutions $\xi_\nu$'s can be adjusted so that they converge to a
weak global solution of $\dbar\xi = \psi$.
A strong solution in $\Gamma(X,\kashf{0,q-1}(L))$ then exists by the
regularity theory for $\dbar$ or elliptic operators
(ref.~\cite{Hoemander}*{Thm.~4.2.5 and Cor.~4.2.6} or
\cite{Hoemander-PDE}*{Thm.~4.1.5 and Cor.~4.1.2}) and the
Kazama--Dolbeault isomorphism (ref.~\cite{Kazama_Dolbeault-isom}, see
also Theorem \ref{rem:Kazama-Dolbeault}).


\section{Preliminaries}
\label{sec:notations}

\subsection{A $(\fieldC^*)^{n-m}$-principal bundle structure on $X$}
\label{sec:fibre-bundle}

Let $X = \fieldC^n / \Gamma$ be a Cousin-quasi-torus. 
Then one has $\fieldC\Gamma = \fieldC^n$ and $\rk\Gamma = n+m$ with
$m>0$.
Define $K := \pi(\fieldR\Gamma) = \fieldR\Gamma / \Gamma$ and $F :=
\fieldR\Gamma\cap\cplxi\fieldR\Gamma$ as before.
Fix a basis of $\fieldC^n$ such that the period matrix of $X$
is given by
\begin{equation} \label{eq:original-period-matrix}
  \begin{bmatrix}
    I_{n-m} & & A_1 + \cplxi B_1 \\
    & I_m & A_2 + \cplxi B_2
  \end{bmatrix} \; ,
\end{equation}
where an empty entry means a zero entry, $I_r$ denotes the
identity matrix of rank $r$, $A_i$ and $B_i$
denotes real matrices 
such that $A_1$ and $B_1$ are of size $(n-m) \times m$, and $A_2$ and
$B_2$ are square matrices of size $m\times m$.
By re-ordering the basis of
$\fieldC^n$ and respectively the basis of $\Gamma$, $B_2$ can be
assumed to be invertible (since $\rk\Gamma = n+m$). 
Take a change of coordinates given by the matrix
\begin{equation*}
  \begin{bmatrix}
    I_{n-m} & -B_1B_2^{-1} \\
    & B_2^{-1}
  \end{bmatrix} \; ,
\end{equation*}
the period matrix under the new coordinates is then given by
\begin{equation} 
  \label{eq:period-matrix}
  \begin{bmatrix}
    I_{n-m} & \beta_1& \alpha_1 \\
    & \beta_2 & \alpha_2 
  \end{bmatrix} \; ,
\end{equation}
where
\begin{equation*}
  \begin{aligned}
    & \beta_1 = -B_1 B_2^{-1} \; , && \alpha_1 = A_1 - B_1 B_2^{-1}A_2 \; , \\
    & \beta_2 = B_2^{-1} \; , && \alpha_2 = B_2^{-1} A_2 + \cplxi I_m
    \; ,
  \end{aligned}
\end{equation*}
which are all real matrices except for $\alpha_2$.
Let the new coordinates of $\fieldC^n$
be denoted by $(u,v) := (u^1, \dots, u^{n-m}, v^1, \dots, v^m)$,
\showsym[uv]{$u^i$ (resp.~$v^j$)}{apt coordinates on fibres,
  $i=1,\dots,n-m$ \\ (resp.~on the base, $j=1,\dots,m$)}
or simply by $z := (z^1,\dots, z^n)$\showsym[uv1]{$z = (u,v)$}{}.
This new coordinate system is called an \emph{apt coordinate system
  (with respect to $\Gamma$)} (see \cite{Cap&Cat}*{Def.~2.3}; also
called an \emph{toroidal coordinate system}, see \cite{Abe&Kopfermann}*{\S
  1.1.12}), which is characterized by the properties
\begin{enumerate}
\item \label{item:eq-of-F} $F = \set{(u,v)\in \fieldC^n \: : \: u = 0}$;
\item each coordinate of the imaginary part $\Im u$ of $u$ is
  a global function on $X$ and $K = \set{(u,v)\!\!\!\mod\Gamma \in X \: :
    \: \Im u = 0}$;
\item \label{item:def:H-apt-coord_fibre-std-basis-in-lattice} the
  standard basic vectors $e_1,\dots,e_{n-m}$ in $\fieldC^n$
  can be completed to a basis of $\Gamma$.
\end{enumerate}

The choice of an apt coordinate system fixes a decomposition
$\fieldC^n = E \oplus F$\showsym[EtF]{$E \oplus F$}{decomposition of
  $\fieldC^n$ induced from the chosen apt coordinate system}, where
$E$\label{page:def-E} is the complex vector subspace
of $\fieldC^n$ spanned by $e_1, \dots, e_{n-m}$ with $u$ as the
coordinate vector.
Set $\Gamma' := \Gamma \cap E = \Znum\genby{e_1,\dots,e_{n-m}} =
\Znum^{n-m}$.
Let $\pT \colon \fieldC^n \to F$ be the projection $(u,v) \mapsto v$.
It can be seen from 
(\ref{eq:period-matrix}) that $\pT(\Gamma)$ is a lattice in $F$,
i.e.~a discrete subgroup of $F$ of rank $2m$.
Let $T^m := F / \pT(\Gamma)$\showsym[T]{$T^m$}{the base complex
  torus of the fibration (\ref{eq:fibration})}, which is a complex
torus of dimension $m$.
Then $\pT$ induces a holomorphic epimorphism $p \colon X \to T^m$
with kernel $E / \Gamma' \isom (\fieldC^*)^{n-m}$.
Therefore,
$X$ has a $(\fieldC^*)^{n-m}$-principal bundle structure given by the
exact sequence of groups
\begin{equation}\label{eq:fibration}
  \xymatrix{0 \ar[r] & (\fieldC^*)^{n-m} \ar[r]^-\iota & X \ar[r]^-p & T^m
    \ar[r] & 0} 
\end{equation}
(ref.~\cite{Steenrod}*{\S 7.4} and \cite{Hirzebruch}*{Thm.~3.4.3}).
In local coordinates, $\iota$ is given by $u \!\!\mod \Gamma' \mapsto
(u,0) \!\!\mod \Gamma$ and $p$ by $(u,v) \!\!\mod \Gamma
\mapsto v \!\!\mod \pT(\Gamma)$.
In view of the fibre bundle structure, the tangential directions
with respect to the $u$-coordinates are called the \emph{fibre
  directions}, while those of the $v$-coordinates are called the
\emph{base directions}.
These terminologies are used throughout this article to simplify
description. 

Since the cotangent bundle 
of $X$ is trivial, the decomposition $\fieldC^n = E \oplus F$ induces a
decomposition of the holomorphic cotangent bundle $\cTgt{1,0} := \cTgt{1,0}_X$ of
$X$ with respect to the fibre and base directions, i.e.
\begin{equation} \label{eq:cotangent-split}
  \cTgt{1,0} = \cTgtf{1,0} \oplus \cTgtb{1,0} \; ,
\end{equation}%
\showsym[tcoTgtf]{$\cTgtf{1,0}$ (resp.~$\cTgtb{1,0}$)}{holomorphic cotangent bundle along
  the fibre \\ (resp.~base) directions}%
where $\cTgtf{1,0}$ and $\cTgtb{1,0}$ are holomorphic subbundles
generated at every point of $X$ respectively by $d u^i$ for $i = 1,\dots
n-m$ and  $d v^j$ for $j = 1,\dots, m$. 
For later use, define as usual
$\cTgtb{p,q} := \bigwedge^p \cTgtb{1,0} \wedge \bigwedge^q
\conj{\cTgtb{1,0}}$ for any integers $p,q \geq 0$, where $\cTgtb{0,0}
= \bigwedge^0 \cTgtb{1,0} = \bigwedge^0 \conj{\cTgtb{1,0}}$ denotes
the trivial line bundle on $X$.
Define $\cTgtf{p,q}$ similarly with $\cTgtf{}$ in place of $\cTgtb{}$.

\subsection{An exhaustive family of pseudoconvex subsets}
\label{sec:pseudoconvexity}

Every Cousin-quasi-torus is pseudoconvex and strongly $(m+1)$-complete
(cf. \cite{Kazama_pseudoconvex} and \cite{Takeuchi}; convention of
the numbering here following \cite{Demailly}*{pp.~512}).
Indeed, define $\varphi(z) := \varphi(\Im u) :=
\norm{\Im u}^2$\showsym[phi]{$\varphi$}{$:= \norm{\Im u}^2$ (Euclidean
  $2$-norm)} ($\norm\cdot$ is the Euclidean $2$-norm here). 
Then $\varphi$ is an exhaustion function on $X$ whose Levi form is given by
\begin{equation*}
  \cplxi\diff\dbar\varphi = \frac{\cplxi}{2} \sum_{i = 1}^{n-m} du^i\wedge d\conj{u^i} \; ,
\end{equation*}
which is semi-positive definite with exactly $n-m$ positive
eigenvalues everywhere on $X$.
Therefore, $X$ is pseudoconvex and strongly $(m+1)$-complete.

For any $c > 0$, set $K_c := \set{z\in X \: : \: \varphi(z) <
  c}$\showsym[Kc]{$K_c$}{$:= \set{z\in X \: : \: \varphi(z) < c}$ for
  $0 < c < \infty$}.
Then $\seq{K_c}_{c > 0}$ forms an exhaustive family of open relatively compact subsets
of $X$. 
Set also $K_\infty := X$\showsym[Kinf]{$K_\infty$ (resp.~$K_0$)}{$:=
  X$ (resp.~$K$)}, and $K_0 := K$, the maximal compact subgroup of
$X$.
For every $c > 0$, $K_c$ is of course itself pseudoconvex. 

\subsection{Kazama sheaves and Kazama--Dolbeault isomorphism}
\label{sec:kazama_resolution}

Let $\smform{} := \smform{}_X$\showsym[A1]{$\smform{}$}{sheaf of germs of smooth ($C^\infty$)
  functions on $X$} be the sheaf of germs of smooth functions on $X$.
Fix a choice of an apt coordinate system.
Let $V$ be any holomorphic vector bundle on $X$.
Define on $X$ the \emph{Kazama sheaves} as in \cite{Kazama_Dolbeault-isom} to be
\begin{gather*} 
  \kashf{} := \set{ f \in \smform{}  \: : \:  \dfrac{\diff
      f}{\diff\conj{u^i}} \equiv 0 \;\textrm{ for } 1\leq i\leq n-m} \quad
  \text{and} \\
  \kashf{0,q} := \kashf{} \otimes_{p^{-1}\smform{}_{T^m}}
  p^{-1}\smform{0,q}_{T^m}
  \; , \quad
  \kashf{0,q}(V) := \kashf{0,q}\otimes_{\holo_X} \holo_X(V)
  \quad\text{for $1\leq q \leq m$,}
\end{gather*}%
\showsym[Hwm]{$\kashf{}$}{Kazama sheaf of smooth ($C^\infty$)
  germs of functions \\ holomorphic along the fibre directions}%
\showsym[Hwmq]{$\kashf{0,q}$}{}%
where $p$ is the projection given in (\ref{eq:fibration}) and
$\smform{0,q}_{T^m}$ is the sheaf of germs of $(0,q)$-forms on the base
torus $T^m$.
In words, Kazama sheaf $\kashf{}$ consists of germs of sections of
$\smform{}$ which are \emph{holomorphic in the fibre directions}, and
$\kashf{0,q}$ consists of $\kashf{}$-valued \emph{$(0,q)$-forms in the
  base directions}.
Note that the definitions of the sheaves depend on the choice of the
decomposition (\ref{eq:cotangent-split}).
Set also $\kashf{0,0}(V) := \kashf{}(V)$. 
  For notational convenience, the space of sections $\Gamma(U,\kashf{0,q}(V))$ over any subset
  $U$ of $X$ is also denoted by
  $\kashf{0,q}(U;V)$\showsym[HwUL]{$\kashf{0,q}(U;V)$}{$:=
    \Gamma(U,\kashf{0,q}(V))$}, and similarly for spaces of
  sections of other sheaves.
The following \emph{Kazama--Dolbeault isomorphism} is proven in
\cite{Ka&U_Dolb-isom_earlier} and \cite{Kazama_Dolbeault-isom} (see
also \cite{Kazama}).
\begin{thm} \label{rem:Kazama-Dolbeault}
  The complex
  \begin{equation}\label{eq:kazama_resolution_L}
    \xymatrix@C-5pt{0 \ar[r] & \holo_X(V) \ar[r] & \kashf{0,0}(V) \ar[r]^-\dbar &
      \kashf{0,1}(V) \ar[r]^-\dbar &  \dots \ar[r]^-\dbar & \kashf{0,m}(V)
      \ar[r] & 0}
  \end{equation}
  is an acyclic resolution of $\holo_X(V)$ over $X$, i.e.~$H^p(X,
  \kashf{0,q}(V)) = 0$ for any $p \geq 1$ and $0\leq q \leq m$.
  Consequently, the natural injection of complexes
  \begin{equation*}
    \xymatrix@C-5pt{
      {\paren{\kashf{0,\bullet}(X;V), \dbar}} \ar@{^(->}[r] &
      {\paren{\smform{0,\bullet}(X;V), \dbar}}} 
  \end{equation*}
  induces the isomorphisms
  \begin{equation*}
    H^q_\dbar (\kashf{0,\bullet}(X;V)) \isom
    H^q_\dbar (\smform{0,\bullet}(X;V)) \isom H^q(X, V)
  \end{equation*}
  for all $q \geq 0$.
\end{thm}

In view of the Kazama--Dolbeault isomorphism,
to show the vanishing of $H^q(X,V)$ it
suffices to show that for any $\dbar$-closed $\psi \in
\kashf{0,q}(X;V)$ there exists $\xi \in \smform{0,q-1}(X;V)$ such that
\begin{equation}\label{eq:dbar_eq}
  \dbar \xi = \psi \; .
\end{equation}
In fact, (\ref{eq:dbar_eq}) means that the class $\psi
\!\!\mod{\dbar\smform{0,q-1}(X;V)}$ is the zero class in
$H^q_\dbar(\smform{0,\bullet}(X;V))$, so, by the isomorphism, the class
$\psi \!\!\mod{\dbar\kashf{0,q-1}(X;V)}$ is also the zero class
in $H^q_\dbar(\kashf{0,\bullet}(X;V))$. 
Therefore, $\xi$ in (\ref{eq:dbar_eq}) can be chosen in $\kashf{0,q-1}(X;V)$. 

\subsection{Holomorphic line bundles on $X$}


Every holomorphic line bundle $L$\showsym[L]{$L$}{a holomorphic line bundle on $X$} on $X$ can
be defined by a system of factors of
automorphy, which can be taken into a normal form analogous to the
Appell--Humbert normal form, given by
(ref.~\cite{Cap&Cat}*{Remark 1.11 and \S 2.2} and \cite{Vogt}*{\S 2})
\begin{equation} \label{eq:factor-of-automorphy-normal-form}
  \varrho(\gamma)e^{\pi \Hform(z,\gamma) + \frac{\pi}{2} \Hform(\gamma,\gamma)+
    f_\gamma(z)} \qquad \forall \: \gamma\in\Gamma \; ,
\end{equation}%
\showsym[H]{$\Hform$}{a hermitian form on $\fieldC^n \times \fieldC^n$ associated to $L$}%
\showsym[fgam]{$f_\gamma$}{a non-linear exponent in the normal form of
  the factors \\ of automorphy defining $L_\wildsub$}%
where
\begin{itemize}
\item $\Hform$ is a hermitian form on $\fieldC^n \times \fieldC^n$, whose imaginary part
  $\Im \Hform$ takes integral values on $\Gamma\times\Gamma$ and
  corresponds to the first Chern class $c_1(L)$ of $L$;

\item $\varrho$ is a semi-character for $\Im \Hform$ on $\Gamma$,
  i.e.
  \begin{equation*}
    \varrho(\gamma+\gamma') =
    \varrho(\gamma)\varrho(\gamma')e^{\pi\cplxi\Im
      \Hform(\gamma,\gamma')} 
  \end{equation*}
  for all $\gamma, \gamma' \in \Gamma$, and $\abs{\varrho(\gamma)} =
  1$ for all $\gamma \in \Gamma$; and

\item $\set{f_\gamma}_{\gamma\in\Gamma}$ is an additive 1-cocycle
  with values in $\holo_{\fieldC^n}(\fieldC^n)$, i.e.~$f_\gamma \in
  \holo_{\fieldC^n}(\fieldC^n)$ for all $\gamma \in \Gamma$ and
  \begin{equation*}
    f_{\gamma+\gamma'}(z) = f_{\gamma'}(z+\gamma) + f_\gamma(z)
  \end{equation*}
  for all $\gamma, \gamma' \in \Gamma$.
\end{itemize}
According to \cite{Vogt}*{Prop.~8}, under a fixed apt coordinate
system, $f_\gamma(z)$ can be taken to be independent of the variable
$v$ for every $\gamma\in\Gamma$.
Denote by $\gamma_u$ the image of $\gamma\in\Gamma$ under the
projection $\fieldC^n \ni (u,v) \mapsto u \in E$ (see page
\pageref{page:def-E} for the definition of $E$).
Also according to \cite{Vogt}*{Prop.~8} (cf.~also
\cite{Cap&Cat}*{\S 1.2}), for any  
$u\in E$, one has
\begin{equation} \label{eq:f-gamma-properties}
  \begin{cases}
    f_{\gamma'}(u) = 0 \\
    f_{\gamma}(u + \gamma'_u) = f_{\gamma}(u)
  \end{cases}
  \quad\text{for all }\gamma'\in \Gamma' \text{ and }\gamma\in\Gamma \; ,
\end{equation}
where $\Gamma' := \Gamma \cap E = \Znum\genby{e_1,\dots,e_{n-m}}$ as
in \S \ref{sec:fibre-bundle}. 

It is apparent that $L$ can be decomposed into $L_\tamesub \otimes
L_\wildsub$\showsym[L2]{$L_\tamesub$ (resp.~$L_\wildsub$)}{tame
  (resp.~wild) part of the line bundle $L$}, where $L_\tamesub$ is
defined by the linear part
\begin{equation*}
  \varrho(\gamma)e^{\pi \Hform(z,\gamma) + \frac{\pi}{2}\Hform(\gamma,\gamma)}
\end{equation*}
of the factor of automorphy in
(\ref{eq:factor-of-automorphy-normal-form}), while $L_\wildsub$ is
defined by the non-linear part
\begin{equation*}
    e^{f_\gamma(z)} \; .
\end{equation*}
Call $L_\tamesub$ and $L_\wildsub$ the \emph{tame part} and \emph{wild part} of $L$
respectively. 
\begin{definition} \label{def:linearizable}
  $L$ is said to be \emph{linearizable} if $L_\wildsub$ is trivial,
  i.e.~there exists a holomorphic function $g$ on $\fieldC^n$ such
  that $g(z + \gamma) - g(z) = f_\gamma(z)$ for all $\gamma \in
  \Gamma$ and $z \in \fieldC^n$. $L$ is said to be
  \emph{non-linearizable} otherwise.
\end{definition}






$\Im \Hform$ is uniquely determined only on
$\fieldR\Gamma\times\fieldR\Gamma$ (see \cite{Cap&Cat}*{Remark 1.11}
and also \cite{Andre&Ghera}). 
Then one has the following proposition.
\begin{prop} \label{prop:H_bot_R-arbitrary}
  Let $\Hform$ be a hermitian form associated to $L$. 
  Suppose
  in a chosen apt coordinate system the matrix associated
  to $\Hform|_{E\times E}$ is given by $\Hbot$. 
  Then, $\Re \Hbot$ can be chosen arbitrarily by multiplying the
  cocycle defining $L$ by a suitable coboundary.
\end{prop}

{
\newcommand{\Bform}{\mathcal B}

\begin{proof}
  Fix $\Hform$ and an apt coordinate system.
  Let $\Bform(u,u)$ be any symmetric $\fieldC$-bilinear form with
  \emph{real} coefficients on $E \times E$ and denote the
  corresponding $(n-m) \times (n-m)$-matrix under the chosen apt
  coordinates by $B$.
  Note that $\gamma_u$
  is a real vector by the choice of coordinates (see
  (\ref{eq:period-matrix})).
  Then multiplying $e^{\frac{\pi}{2} \Bform(u + \gamma_u,u +
    \gamma_u)-\frac{\pi}{2} \Bform(u,u)}$ (which is a component of a
  1-coboundary) to (\ref{eq:factor-of-automorphy-normal-form}) gives
  rise to a system of factors of automorphy defining a line bundle
  isomorphic to $L$. 
  The new system of factors of automorphy is of the
  same form as in (\ref{eq:factor-of-automorphy-normal-form}) with
  $\Hform$ replaced by $\Hform'$, where $\Hform'$ is a hermitian form such that
  $\Hform'(z,\gamma) = \Hform(z,\gamma) + \Bform(u,\gamma_u)$ (note that such
  hermitian $\Hform'$ exists since all $\gamma_u$'s as well as $B$ are real).
  Then $\Re \Hbot' = \Re \Hbot + B$, while the other entries of the
  matrix of $\Im \Hform'$ are the same as the respective entries of $\Im
  \Hform$. Therefore, since $B$ is arbitrary, $\Re \Hbot$ can be chosen
  arbitrarily.
\end{proof}

}

This shows that one cannot, in general, replace $s_F^+$ and $s_F^-$ in
Theorem \ref{thm:main_thm} by $s^+$ and $s^-$, the numbers of positive
and negative eigenvalues of $\Hform$ (instead of $\Hform|_{F\times F}$)
respectively. In fact, if $L$ is the trivial line bundle, $\Hform$ can be chosen
such that $\Re \Hbot$ is negative definite and the other entries of
the matrix associated to $\Im \Hform$ are zero. Such $\Hform$ has at least $1$
negative eigenvalue. However, $\dim H^0(X,\holo_X)$ cannot be $0$
since there exist constant functions on
$X$ (which is true even for any complex manifold). In fact, Kazama
has shown in \cite{Kazama}*{Thm.~4.3} that $H^q(X,\holo_X)$ are
non-trivial for all $1\leq q\leq m$ for any Cousin-quasi-torus $X$.

\subsection{A hermitian metric on $L$}
\label{sec:metric-on-L}

Given a holomorphic line bundle $L$,
hermitian metrics $\etat$ on the tame part $L_\tamesub$ and $\etaw$ on
the wild part $L_\wildsub$ of $L$ are defined below.
The product $\eta := \etat \etaw$ then defines a hermitian metric on $L$.

Define a hermitian metric on $L_\tamesub$ by $\etat(z) :=
e^{-\pi \Hform(z,z)}$\showsym[etat]{$\etat$}{$:= e^{-\pi \Hform(z,z)}$, a
  hermitian metric on $L_\tamesub$} as in the compact case. 
Then the corresponding curvature form on $X$, called
the 
\emph{tame part of the curvature form of $L$},
is given by
\begin{equation*}
  \begin{split}
    \Theta_\tCurv &:= -\cplxi\diff\dbar\log\etat = \pi\cplxi\diff\dbar \Hform(z,z)
    \; .
  \end{split}
\end{equation*}
%


Next is to define a hermitian metric $\etaw$ on $L_\wildsub$.
An apt coordinate system is fixed in what follows.
First notice the following
\begin{prop} \label{prop:existence-of-hbar}
  There exists a smooth function $\hbar$ on $\fieldC^n$ which
  is holomorphic along the fibre directions under the chosen apt
  coordinate system and satisfies 
  \begin{equation} \label{eq:f-gamma-coboundary}
    \hbar(z+\gamma) - \hbar(z) = f_\gamma(u) \quad\text{ for all }
    \gamma\in\Gamma \; .
  \end{equation}
\end{prop}%
\showsym[h]{$\hbar$}{a smooth function on $\fieldC^n$ whose coboundary
  is $\set{f_\gamma}_{\gamma \in \Gamma}$; \\ see Proposition
  \ref{prop:existence-of-hbar}}%

\begin{proof}

  This follows from the fact that $H^1(X,\kashf{}) = 0$
  (ref.~\cite{Kazama_Dolbeault-isom}).
  A direct proof is given as follows.

  Let $\Gamma''$ be the subgroup of
  $\Gamma$ generated by the last $2m$ column vectors of the period
  matrix (\ref{eq:period-matrix}) of $\Gamma$.
  Then $\Gamma = \Gamma' \oplus \Gamma''$ ($\Gamma'$ defined as in \S
  \ref{sec:fibre-bundle}).
  Write $\gamma_v := \pT(\gamma)$ for all $\gamma \in \Gamma$ ($\pT$
  defined as in \S \ref{sec:fibre-bundle}).
  Note that $\gamma'_v = 0$ for all $\gamma'\in \Gamma'$.
  Recall that $\pT(\Gamma) = \pT(\Gamma'')$ is the lattice defining
  $T^m$ in (\ref{eq:fibration}), therefore discrete in $F$.
  Take a suitable smooth function $\rho$ with compact support on $F$
  with variable $v$ such that $\sum_{\gamma''\in\Gamma''}\rho(v +
  \gamma''_v) \equiv 1$.
  Note that the sum is a sum of finitely many non-zero terms at
  each $v\in F$ due to the discreteness of $\Gamma''$.
  Define
  \begin{equation*}
    \hbar(z) := - \sum_{\gamma''\in\Gamma''} \rho(v + \gamma''_v)
    f_{\gamma''}(u) \; .
  \end{equation*}
  Then $\hbar$ is holomorphic along the fibre directions. 
  To see that it satisfies (\ref{eq:f-gamma-coboundary}), note that, for
  any $\gamma_0 = \gamma'_0 + \gamma''_0 \in \Gamma$ 
  where $\gamma'_0 \in \Gamma'$ and $\gamma''_0 \in \Gamma''$,
  \begin{equation*}
    \begin{split}
      \hbar(z + \gamma_0) &= - \sum_{\gamma''\in\Gamma''} \rho(v + \gamma''_v
      + (\gamma_0)_v)~f_{\gamma''}(u + (\gamma_0)_u) \\ &= -
      \sum_{\gamma''\in\Gamma''} \rho(v + \gamma''_v +
      (\gamma''_0)_v)~f_{\gamma''}(u + (\gamma''_0)_u) \\ &= -
      \sum_{\gamma''\in\Gamma''} \rho(v + \gamma''_v + (\gamma''_0)_v)
      \paren{f_{\gamma'' + \gamma''_0}(u) - f_{\gamma''_0}(u)} \\ &= \hbar(z)
      + f_{\gamma''_0}(u) \; ,
    \end{split}
  \end{equation*}
  using the fact that $f_{\gamma''}(u + \gamma'_u) = f_{\gamma''}(u)$
  for all $\gamma'\in \Gamma'$ (see (\ref{eq:f-gamma-properties}))
  and $\Gamma'' = \Gamma'' + \gamma''_0$. Applying
  (\ref{eq:f-gamma-properties}) again, one obtains
  \begin{equation*}
    f_{\gamma_0}(u) = f_{\gamma''_0}(u + (\gamma'_0)_u) +
    f_{\gamma'_0}(u) = f_{\gamma''_0}(u) \; .
  \end{equation*}
  This $\hbar$ therefore satisfies (\ref{eq:f-gamma-coboundary}).
\end{proof}

It follows from (\ref{eq:f-gamma-coboundary}) that $\frac{\diff}{\diff
  \conj{v^j}}\hbar$ and $\frac{\diff}{\diff \conj{u^i}}\hbar$ define
smooth functions on $X$ (note that $f_\gamma(u)$ are
holomorphic). Therefore, $\dbar\hbar$ is a (smooth) $1$-form on $X$,
so is $\diff\conj\hbar$.

Take any $\delta \in \kashf{}(X)$\showsym[delta]{$\delta$}{an
  auxiliary function in $\kashf{}(X)$ appears in the \\ definition of
  $\etaw$},
and let $\hbar_\delta := \hbar -
\delta$\showsym[h]{$\hbar_\delta$}{$:= \hbar - \delta$} for notational
convenience.
Define a hermitian metric on the wild part $L_\wildsub$ of $L$ by $\etaw(z) :=
e^{-2\Re\hbar_\delta(z)}$\showsym[etaw]{$\etaw$}{$:=
e^{-2\Re\hbar_\delta}$, a hermitian metric on
  $L_\wildsub$}.
The corresponding curvature form, called the
\emph{wild part of the curvature form of $L$},
is given by
\begin{equation*}
  \begin{split}
    \Theta_\wCurv &:= -\cplxi\diff\dbar\log\etaw = 2\cplxi\diff\dbar
    \Re\hbar_\delta \\
    &= \cplxi d\paren{\dbar\hbar_\delta - \diff\conj\hbar_\delta} \; .
  \end{split}
\end{equation*}
Note that, since $\dbar\hbar$ is a smooth $(0,1)$-form on $X$, $\cplxi
d\paren{\dbar\hbar_\delta - \diff\conj\hbar_\delta}$
is a $d$-exact smooth real $(1,1)$-form on $X$. 

The function $\delta$ is an auxiliary function which will be chosen
suitably according to the Weak $\diff\dbar$-Lemma of Takayama
\cite{Takayama}*{Lemma 3.14} (see also Lemma \ref{lem:weak-ddbarlemma})
in order to obtain the required $L^2$ estimates.
Details are given in \S 
\ref{sec:bound-wCurv}.

With the chosen $\etat$ and $\etaw$, a hermitian metric on $L$ is defined
by
\begin{equation} \label{eq:metric-on-L}
  \eta(z) := \etat(z)\etaw(z) = e^{-\pi \Hform(z,z) - 2\Re\hbar_\delta(z)} \; .
\end{equation}\showsym[etawt]{$\eta$}{$:= \etat \etaw$, a hermitian
  metric on $L$}%
The \emph{curvature form of $L$ with respect to $\eta$} is then given by
\begin{equation*}
  \Theta_\tCurv + \Theta_\wCurv \; ,
\end{equation*}
which represents the class $2\pi c_1(L)$ in $H^2(X,\fieldR)$ (while
$\Theta_\tCurv$ represents $2\pi c_1(L)$ in $2\pi H^2(X,\Znum)$).

\subsection{An $L^2$-norm, the $L^2$-spaces $\hilbFB {q'}{q''}_{c,\chi}$ and
  differential operators}
\label{sec:choice_of_metric}

Let $g$\showsym[G1]{$g$}{a translational invariant hermitian metric on $X$ such that
  $\Tgtf{1,0} \perp \Tgtb{1,0}$} be a hermitian metric on $X$.
Fix an apt coordinate system.
For the purpose of this article, \emph{$g$ is chosen to be a translational invariant
  metric such that the decomposition $\Tgt{1,0} = \Tgtf{1,0} \oplus
  \Tgtb{1,0}$ is orthogonal.}
Denote by $\omega := - \Im g$ the associated $(1,1)$-form as usual. 

Fix any holomorphic line bundle $L$. 
Consider any $0 < c \leq \infty$ and $0 \leq q \leq n$.
Denote the pointwise $2$-norm on $\smform{0,q}(K_c;L)$ induced from
the hermitian metrics $g$ and $\eta$ by
$\abs\cdot_{g,\eta}$\showsym[z]{$\abs\cdot_{g,\eta}$}{a pointwise norm on
  $\smform{0,q}(K_c;L)$ with respect to \\ hermitian metrics $g$ and $\eta$}.
Let also $\chiR \colon \fieldR_{\geq 0} \to \fieldR$ be a smooth function
and set $\chiX := \chiR \circ \varphi$. 
For the purpose of this article, \emph{$\chiR$ is always assumed to be
  a non-negative convex increasing function.}
In this case, $\chiX$ is plurisubharmonic.
Set $\abs{\zeta}_{g,\eta, \chi}^2 := \abs{\zeta}_{g,\eta}^2
e^{-\chiX}$.
Let $\mu$ be the measure induced from the volume form
$\frac{\omega^{\wedge n}}{n!}$.
Define
\begin{equation*}
  \norm\zeta_{K_c,\chi} :=  \sqrt{\int_{K_c} \abs{\zeta}_{g,\eta,\chi}^2
  d\mu}
  \qquad \text{for any }\zeta\in\smform{0,q}(K_c;L) \; .
\end{equation*}%
\showsym[z]{$\norm\cdot_{K_c,\chi}$
  (resp.~$\inner\cdot\cdot_{K_c,\chi}$)}{an $L^2$-norm (resp.~inner
  product) \\ on $\Ltwo{0,q}_{c,\chi}$ with respect to $g$ and $\eta$
  with weight $\chi$}%
Then 
$\norm\cdot_{K_c,\chi}$ defines an $L^2$-norm with weight $e^{-\chiX}$ (or
simply $\chi$) on $\smform{0,q}_0(K_c;L)$, the space of sections in
$\smform{0,q}(K_c;L)$ with compact support. 
To simplify notation, $d\mu$ in the integral is made implicit in what follows. 
The inner product corresponding to $\norm\cdot_{K_c,\chi}$ is denoted by
$\inner\cdot\cdot_{K_c,\chi}$.
The norm is written as $\norm\cdot_{K_c,g,\eta,\chi}$ to emphasize its
dependence on $g$ and $\eta$ when necessary.

Denote by $\Ltwo{0,q}_{c,\chi} := \Ltwo{0,q}_\chi(K_c;L)$ the Hilbert
space of ($\mu$-)measurable $L$-valued $(0,q)$-forms $\zeta$ on $K_c$ such
that $\norm\zeta_{K_c,\chi} < \infty$\showsym[L4]{$\Ltwo{0,q}_{c,\chi}
  := \Ltwo{0,q}_\chi (K_c;L)$}{the Hilbert
space of measurable $L$-valued \\ $(0,q)$-forms $\zeta$ on $K_c$ such
that $\norm\zeta_{K_c,\chi} < \infty$}.
It is well known that $\smform{0,q}_0(K_c;L) \subset \Ltwo{0,q}_{c,\chi}$ is a
dense subspace under the norm $\norm\cdot_{K_c,\chi}$.

For any $0\leq p',q' \leq n-m$ and $0 \leq p'',q'' \leq m$, define
\begin{gather*}
  \smFfb {p'}{p''}{q'}{q''} := \smooth\paren{\cTgtf{p',q'} \wedge \cTgtb{p'',q''}} \; ,
\end{gather*}%
\showsym[A5]{$\smFfb {p'}{p''}{q'}{q''}$}{$:=
  \smooth\paren{\cTgtf{p',q'} \wedge \cTgtb{p'',q''}}$}%
i.e.~a sheaf of germs of smooth sections of 
$\cTgtf{p',q'} \wedge \cTgtb{p'',q''}$ (defined in \S
\ref{sec:fibre-bundle}).
For other values of $p', p'', q'$ and $q''$, set $\smFfb {p'}{p''}{q'}{q''} := 0$.
Note that, for $0 \leq p,q \leq n$,
there is a decomposition
\begin{equation} \label{eq:form-decomposition}
  \smform{p,q} = \bigoplus_{\substack{p'+p'' = p\\ q'+q'' = q}}
  \smFfb{p'}{p''}{q'}{q''} \; .
\end{equation}
This decomposition depends on the choice of the decomposition
(\ref{eq:cotangent-split}).
Since the fibre and base directions are orthogonal to each other with
respect to $g$, the decomposition is also orthogonal with respect to $g$.
As only those sheaves with $p'+p'' = 0$ are considered in what
follows, set
\begin{equation*}
  \smfb {q'}{q''} := \smFfb 00{q'}{q''}
\end{equation*}%
\showsym[A5a]{$\smfb {q'}{q''}$}{$:= \smFfb 00{q'}{q''}$}%
for notational convenience.
Notice that $\kashf{0,q''}(L)$ is a subsheaf of $\smbform{0,q''}(L)$ for
$0 \leq q'' \leq m$. 
For any $c > 0$, denote also the space of sections in $\smfb {q'}{q''}(K_c;L)$
with compact support by $\smfb {q'}{q''}_0(K_c;L)$\showsym[A6]{$\smfb
  {q'}{q''}_0(K_c;L)$}{space of sections in $\smfb {q'}{q''}(K_c;L)$ \\ with compact
  support}.
Define
\begin{equation*}
  \hilbFB {q'}{q''}_{c,\chi} := \hilbFB {q'}{q''}_\chi (K_c; L) := \cl{\smfb {q'}{q''}_0
    (K_c;L)} \; ,
\end{equation*}%
\showsym[L5]{$\hilbFB {q'}{q''}_{c,\chi} := \hilbFB {q'}{q''}_\chi (K_c;
  L)$}{closure of $\smfb {q'}{q''}_0(K_c;L)$ in $\Ltwo{0,q'+q''}_{c,\chi}$ \\ with
  respect to $\norm\cdot_{K_c,\chi}$}%
i.e.~the closure of $\smfb {q'}{q''}_0(K_c;L)$ in $\paren{\Ltwo{0,q'+q''}_{c,\chi},
\norm\cdot_{K_c,\chi}}$.
Note that the decomposition
\begin{equation} \label{eq:hilbsp-decomposition}
  \Ltwo{0,q}_{c,\chi} = \bigoplus_{q'+q'' = q} \hilbFB{q'}{q''}_{c,\chi}
\end{equation}
induced from (\ref{eq:form-decomposition}) is also an orthogonal decomposition.

The operator $\dbar$ is decomposed into $\dbarf + \dbarb$%
\showsym[diffu]{$\dbarf$ (resp.~$\dbarb$)}{composition of $\dbar$
  followed by \\ the projection $\cTgt{0,1} \to \cTgtf{0,1}$
  (resp.~$\cTgt{0,1} \to \cTgtb{0,1}$)} according to
the decomposition (\ref{eq:cotangent-split}),  
where $\dbarf$ and $\dbarb$ are operators such that
\begin{gather*}
  \dbarf \colon \smfb {q'}{q''}(K_c;L) \to \smfb{q'+1}{q''} (K_c;L) \quad\text{and} \\ 
  \dbarb \colon \smfb {q'}{q''}(K_c;L) \to \smfb {q'}{q''+1} (K_c;L) \; .
\end{gather*}
Denote the \emph{formal adjoints} of $\dbarf$ and $\dbarb$ above
respectively by
\begin{gather*}
  \ethf \colon \smfb{q'+1}{q''} (K_c;L) \to \smfb {q'}{q''} (K_c;L) \quad\text{and} \\
  \ethb \colon \smfb {q'}{q''+1} (K_c;L) \to \smfb {q'}{q''} (K_c;L) 
\end{gather*}%
\showsym[diffzadjf]{$\ethf$ (resp.~$\ethb$)}{formal adjoint of
  $\dbarf$ (resp.~$\dbarb$)}%
(see, for example, \cite{Demailly}*{Ch.~VI, 1.5} for the definition).

Some basic facts about differential operators on Hilbert spaces are
recalled here.
Extend the action of these operators to $\hilbFB
{q'}{q''}_{c,\chi}$ in the sense of distributions (or currents).
Then, they define \emph{closed} (i.e.~having closed graph)
and \emph{densely defined} linear operators on $\hilbFB {q'}{q''}_{c,\chi}$ (see, for example,
\cite{Hoemander-PDE}*{Ch.~1} and \cite{Demailly-L2}*{Prop.~4.9}) with \emph{domain} given by
\begin{equation} \label{eq:domain-dbar}
  \Dom_{K_c,\chi}^{(q',q'')} T  \;\text{ (or $\Dom T$)}\; 
  := \set{\zeta\in\hilbFB {q'}{q''}_{c,\chi} \: : \: \norm{T \zeta}_{K_c,\chi} < \infty} \; ,
\end{equation}
where $T$ denotes any of the above operators.
Note that $T$ is densely defined since $\smfb {q'}{q''}_0 (K_c;L) \subset
\Dom_{K_c,\chi}^{(q',q'')} T$. 
An operator will be written as $(T,\Dom T)$ when the domain is 
emphasized.

Given $\dbarf \colon \hilbFB {q'}{q''}_{c,\chi} \to
\hilbFB {q'+1}{q''}_{c,\chi}$ and $\dbarb \colon \hilbFB {q'}{q''}_{c,\chi} \to \hilbFB
{q'}{q''+1}_{c,\chi}$ with domains given as in (\ref{eq:domain-dbar}),
their Hilbert space adjoints (also called Von Neumann's adjoints, see
for example \cite{Demailly}*{Ch.~VIII, \S 1} for a discussion on them) are
denoted respectively by
\begin{equation*}
  \dbarf^* \colon \hilbFB {q'+1}{q''}_{c,\chi} \to \hilbFB {q'}{q''}_{c,\chi} \; \quad\text{and}\quad
  \dbarb^* \colon \hilbFB {q'}{q''+1}_{c,\chi} \to \hilbFB {q'}{q''}_{c,\chi} \; ,
\end{equation*}%
\showsym[diffzHadjf]{$\dbarf^*$ (resp.~$\dbarb^*$)}{Hilbert space
  adjoint of $\dbarf$ (resp.~$\dbarb$)}%
which are closed and densely defined operators on $\hilbFB
{q'+1}{q''}_{c,\chi}$ and $\hilbFB {q'}{q''+1}_{c,\chi}$ respectively.
Denote also their domains of definition
respectively by $\Dom_{K_c,\chi}^{(q'+1,q'')}\dbarf^*$ and
$\Dom_{K_c,\chi}^{(q',q''+1)}\dbarb^*$.

In general, one has $\Dom_{K_c,\chi}^{(q'+1,q'')}\dbarf^* \subset
\Dom_{K_c,\chi}^{(q'+1,q'')}\ethf$ and $\dbarf^*\zeta = \ethf \zeta$ for
all $\zeta \in \Dom_{K_c,\chi}^{(q'+1,q'')}\dbarf^*$ (see, for
example, \cite{Demailly}*{Ch.~VIII, \S 3}).
The same holds true for $\dbarb^*$ and $\ethb$.

\section{$L^2$ estimates}
\label{sec:L2-estimates}

%
%

{ 
\newcommand{\hsp}{\mathfrak H}
\newcommand{\Nform}{\mathcal N}
\newcommand{\siu}[1]{\left[#1\right]_{\text{Siu}}}

\subsection{Existence of a solution of $\dbar\xi = \psi$}

The aim of this section is to show that, for $0 \leq q \leq m$, given
$\psi \in \kashf{0,q}(K_c;L) \cap \hilbFB 0q_{c,\chi}$
such that $\dbar\psi = 0$ on $K_c$,
there exists a weak solution $\xi \in \hilbFB 0{q-1}_{c,\chi}$ of the
$\dbar$-equation $\dbar \xi = \psi$ provided that an $L^2$ estimate is
satisfied.
When $c=\infty$, there exists a strong solution which lies in
$\kashf{0,q-1}(X;L)$.

First recall the following classical theorems for $L^2$ estimates
(see, for example, \cite{Hoemander}*{Lemmas 4.1.1 and 4.1.2} or
\cite{Demailly}*{Ch.~VIII, Thm.~1.2}). 
Let $\paren{\hsp_1,
  \inner\cdot\cdot_1}$, $\paren{\hsp_2, \inner\cdot\cdot_2}$ and
$\paren{\hsp_3, \inner\cdot\cdot_3}$ be some Hilbert spaces, and let
$(S,\Dom S)$ and $(T, \Dom T)$ be two closed (i.e.~closed graph) and
densely defined linear operators with domains $\Dom S \subset \hsp_2$
and $\Dom T \subset \hsp_1$ respectively such that 
\begin{equation*}
  \xymatrix{{\hsp_1} \ar[r]^-T & {\hsp_2} \ar[r]^-S & {\hsp_3}}
\end{equation*}
and $S\circ T = 0$, i.e.~$T(\Dom T) \subset \ker S := \set{\zeta \in
  \Dom S \: : \: S \zeta = 0}$.
Let $S^*$ and $T^*$ denote the Hilbert space adjoints of $S$ and $T$
respectively, which are also closed, densely defined and satisfies 
$T^* \circ S^* = 0$ (see, for example, \cite{Demailly}*{Ch.~VIII,
  Thm.~1.1}).

\begin{thm}[see \cite{Hoemander}*{Lemmas 4.1.1 and 4.1.2}]
  \label{thm:existence-weak-solution} 
  If there exists a constant $C >0$ such that
  \begin{equation} \label{eq:a-priori-estimate_general}
    \norm{S\zeta}_3^2 + \norm{T^*\zeta}_1^2 \geq C \norm{\zeta}_2^2
    \quad\text{for all }\zeta \in \Dom S \cap \Dom T^* \; ,
  \end{equation}
  then
  \begin{enumerate}
  \item \label{item:solve-T-eq} for every $\psi \in \ker S$, there exists $\xi \in \cl{\im
      T^*} \cap \Dom T$ such that $T\xi = \psi$ and $\norm\xi_1^2 \leq \frac{1}{C}
    \norm\psi_2^2$. In other words, $\ker S = \im T$ (and thus $\im T$
    is closed as $\ker S$ is so);
  \item \label{item:solve-T-star-eq} for every $\Psi \in \paren{\ker T}^\bot = \cl{\im T^*}$, there
    exists $\Xi \in \cl{\im T} \cap \Dom T^*$ such that $T^*\Xi =
    \Psi$ and $\norm\Xi_2^2 \leq \frac{1}{C} \norm\Psi_1^2$. In other
    words, $\cl{\im T^*} = \im T^*$.
  \end{enumerate}
\end{thm}

\begin{remark}
  By exchanging the roles of $S$ and $T^*$, one also gets $\ker T^* =
  \im S^*$ and $\cl{\im S} = \im S$ if the $L^2$ estimate
  (\ref{eq:a-priori-estimate_general}) is satisfied.
\end{remark}

When $X$ is compact, consider the complex
\begin{equation*}
  \xymatrix@1{{\Ltwo{0,q-1}(X;L)} \ar[r]^-\dbar & {\Ltwo{0,q}(X;L)}
    \ar[r]^-\dbar & {\Ltwo{0,q+1}(X;L)}} \; .
\end{equation*}
Murakami \cite{Murakami} shows that the $L^2$ estimates
(\ref{eq:a-priori-estimate_general}) hold for $q < s^-$ or $q > n -
s^+$ by choosing the hermitian metric $g$ suitably.
The $L^2$ estimate on $\Ltwo{0,q}(X;L)$ implies that the harmonic
$L$-valued $(0,q)$-forms must vanish.
Elements in $H^q(X,L)$ are represented by harmonic forms
when $X$ is compact, so this proves the vanishing of $H^q(X,L)$ in
the compact case.

In the current situation, although elements in $H^q(X,L)$ are not represented
by harmonic forms in general, the $L^2$ estimate
(\ref{eq:a-priori-estimate_general}) is still useful in solving
$\dbar$-equations which leads to the vanishing of $H^q(X,L)$ for
suitable $q$'s according to Theorem
\ref{thm:existence-weak-solution}~(\ref{item:solve-T-eq}).

Due to the existence of non-linearizable line bundles, it turns out it
is necessary to solve $\dbar$-equation on $K_c$ for any $0 < c
< \infty$ (see \S \ref{sec:bound-wCurv}).
Therefore, the aim now is to solve the $\dbar$-equation $\dbar\xi =
\psi|_{K_c}$ for a given $\psi \in \kashf{0,q}(X;L)$ with $\dbar\psi =
0$.
In view of the fibre bundle structure (\ref{eq:fibration}), instead of considering the complex
$\xymatrix@1{{\Ltwo{0,q-1}_{c,\chi}} \ar[r]^-\dbar & {\Ltwo{0,q}_{c,\chi}} \ar[r]^-\dbar
  & {\Ltwo{0,q+1}_{c,\chi}}}$, it is natural (see the discussion in \S
\ref{sec:methodology}) to consider the subcomplex
\begin{equation} \label{eq:short-dbar-complex}
  \xymatrix{{\hilbbsp{0, q-1}_{c,\chi}} \ar@<0.5ex>[r]^-{T_{q-1}} & {\hilbFBtwo[,\chi]} 
    \ar@<0.5ex>[l]^-{T_{q-1}^*} \ar@<0.5ex>[r]^-{S_q} &
    {\hilbFBthree[,\chi]} \ar@<0.5ex>[l]^-{S_q^*} } \; ,
\end{equation}%
\showsym[Tgq]{$T_{q-1}$ (resp.~$S_q$)}{a closed and densely
defined linear \\ operator $\hilbFB 0{q-1}_{c,\chi} \xrightarrow{T_{q-1}}
\hilbFBtwo[,\chi]$ such that $T_{q-1}\zeta = \dbar\zeta$ \\
(resp.~$\hilbFBtwo[,\chi] \xrightarrow{S_q}
\hilbFBthree[,\chi]$ such that $S_q\zeta = \dbar\zeta$)}%
where 
$T_{q-1}$ and $S_q$ act as $\dbar$ on $\hilbbsp{0,q-1}_{c,\chi}$ and
$\hilbFBtwo[,\chi]$ respectively, and $T_{q-1}^*$ and $S_q^*$ are
their Hilbert space adjoints.\footnote{The symbol $T_{q-1}$ (resp.~$S_q$)
is used instead of $\dbar$ so that the domains and codomains of the
two operators can be distinguished.
More precisely, if $\iota \colon \hilbFB 0{q-1}_{c,\chi}
\hookrightarrow \Ltwo{0,q-1}_{c,\chi}$ and $\pr \colon
\Ltwo{0,q}_{c,\chi} \to \hilbFBtwo[,\chi]$ are respectively the
inclusion and projection, then $T_{q-1} = \pr \circ \dbar \circ \iota$.
Therefore, $T_{q-1}^*$ and $\dbar^*$ are different operators.}
The Hilbert spaces in the complex are defined as
\begin{gather*}
  \begin{aligned}
    \smFBtwo &:= \smfb 1{q-1} \oplus \smbform{0,q}(K_c;L) \; ,
    \\
    \smFBthree &:= \smfb 2{q-1} \oplus \smfb 1{q} \oplus
    \smbform{0,q+1}(K_c;L) \; ;
  \end{aligned}
  \\
  \begin{aligned}
    \hilbFBtwo[,\chi] &:= \cl{\smFBtwo[0\,]} = \hilbFB 1{q-1}_{c,\chi} \oplus
    \hilbbsp{0,q}_{c,\chi} \; , \\
    \hilbFBthree[,\chi] &:= \cl{\smFBthree[0\,]} = \hilbFB 2{q-1}_{c,\chi} \oplus
    \hilbFB 1q_{c,\chi} \oplus \hilbbsp{0,q+1}_{c,\chi} \; .
  \end{aligned}
\end{gather*}%
\showsym[A8]{$\smFBtwo$}{$:= \smfb 1{q-1} \oplus \smbform{0,q}(K_c;L)$}%
\showsym[A9]{$\smFBthree$}{$:= \smfb 2{q-1} \oplus \smfb 1{q} \oplus
  \smbform{0,q+1}(K_c;L)$}%
\showsym[L6]{$\hilbFBtwo[,\chi]$}{$:= \hilbFB 1{q-1}_{c,\chi} \oplus \hilbbsp{0,q}_{c,\chi}$}%
\showsym[L7]{$\hilbFBthree[,\chi]$}{$:= \hilbFB 2{q-1}_{c,\chi} \oplus 
  \hilbFB 1q_{c,\chi} \oplus \hilbbsp{0,q+1}_{c,\chi}$}%
Recall from (\ref{eq:form-decomposition}) and
(\ref{eq:hilbsp-decomposition}) that all the direct sums on the right hand
sides above are orthogonal decompositions.
Denote the norms on $\hilbbsp{0,q-1}_{c,\chi}$, $\hilbFBtwo[,\chi]$ and
$\hilbFBthree[,\chi]$ respectively by $\norm\cdot_1$, $\norm\cdot_2$ and
$\norm\cdot_3$\showsym[z123]{$\norm\cdot_1$, $\norm\cdot_2$,
$\norm\cdot_3$}{the norms on resp.~$\hilbbsp{0,q-1}_{c,\chi}$, $\hilbFBtwo[,\chi]$,
$\hilbFBthree[,\chi]$}, and their inner products by $\inner\cdot\cdot$ with
the corresponding subscripts. 

Write the Hilbert space adjoint of $\dbar \colon \Ltwo{0,q-1}_{c,\chi}
\to \Ltwo{0,q}_{c,\chi}$ as $\dbar^*$.
Let $\pr \colon \Ltwo{0,q-1}_{c,\chi} \to \hilbFB 0{q-1}_{c,\chi}$ be
the orthogonal projection.
For later use, $(T_{q-1}^*, \Dom T_{q-1}^*)$ is described more explicitly.
\begin{prop} \label{prop:T-star}
  With the notation described above, one has
  \begin{equation*}
    \begin{aligned}
      \Dom T_{q-1}^* &= \Dom_{K_c,\chi} \dbar^* \cap \hilbFBtwo[,\chi]
      \\
      &= \Dom_{K_c,\chi}^{(1,q-1)} \dbarf^* \oplus
      \Dom_{K_c,\chi}^{(0,q)} \dbarb^* \; .
    \end{aligned}    
  \end{equation*}
  Moreover, for any $\zeta = \zeta' + \zeta'' \in \Dom T_{q-1}^*$ where
  $\zeta' \in \Dom_{K_c,\chi}^{(1,q-1)} \dbarf^*$ and $\zeta'' \in
  \Dom_{K_c,\chi}^{(0,q)} \dbarb^*$, one has $T_{q-1}^* \zeta = \pr
  \dbar^*\zeta = \dbarf^*\zeta' + \dbarb^*\zeta''$.
\end{prop}

{
\newcommand{\adjCnddt}{W}

\begin{proof}
  Define operators $(\adjCnddt_1,\Dom\adjCnddt_1)$ and
  $(\adjCnddt_2,\Dom\adjCnddt_2)$ from $\hilbFBtwo[,\chi]$ into $\hilbFB
  0{q-1}_{c,\chi}$ such that
  \begin{equation*}
    \begin{aligned}
      \Dom\adjCnddt_1 &:= \Dom_{K_c,\chi}\dbar^* \cap \hilbFBtwo[,\chi] \; , \\
      \Dom\adjCnddt_2 &:= \Dom_{K_c,\chi}^{(1,q-1)} \dbarf^* \oplus
      \Dom_{K_c,\chi}^{(0,q)} \dbarb^* \; ,
    \end{aligned}
  \end{equation*}
  and 
  \begin{equation*}
    \begin{aligned}
      \adjCnddt_1\zeta &:= \pr\dbar^*\zeta && \text{for } \zeta \in
      \Dom\adjCnddt_1 \; , \\
      \adjCnddt_2\zeta &:= \dbarf^*\zeta' + \dbarb^*\zeta'' && \text{for
      } \zeta = \zeta' + \zeta'' \in \Dom\adjCnddt_2 \; .
    \end{aligned}
  \end{equation*}
  These are closed and densely defined linear operators on $\hilbFBtwo[,\chi]$.
  Since $\norm{T_{q-1}\zeta}_{2}^2 = \norm{\dbar\zeta}_{2}^2 =
  \norm{\dbarf\zeta}_{2}^2 + \norm{\dbarb\zeta}_{2}^2$ for all $\zeta
  \in \hilbFB 0{q-1}_{c,\chi}$, it follows that
  \begin{equation*}
    \begin{aligned}
      \Dom T_{q-1} &= \Dom \dbar \cap \hilbFB 0{q-1}_{c,\chi} \\ 
      &= \Dom_{K_c,\chi}^{(0,q-1)} \dbarf \cap
      \Dom_{K_c,\chi}^{(0,q-1)} \dbarb \; .
    \end{aligned}
  \end{equation*}

  First is to show that $(T_{q-1}^*, \Dom T_{q-1}^*) =
  (\adjCnddt_1,\Dom\adjCnddt_1)$.
  Note that, for any $f \in \hilbFB 0{q-1}_{c,\chi}$ and any $\zeta
  \in \Dom\adjCnddt_1$, one has
  \begin{equation*}
    {\inner f{\adjCnddt_1\zeta}_1} = {\inner
      f{\pr\dbar^*\zeta}_1} = {\inner f{\dbar^*\zeta}_{K_c,\chi}} \; .
  \end{equation*}
  For any $\tilde \zeta \in \Ltwo{0,q}_{c,\chi} = \hilbFBtwo[,\chi]
  \oplus \paren{\hilbFBtwo[,\chi]}^\perp$, write $\tilde \zeta = \zeta
  + \zeta^\perp$ where $\zeta \in \hilbFBtwo[,\chi]$ and $\zeta^\perp
  \in \paren{\hilbFBtwo[,\chi]}^\perp = \bigoplus_{q'=2}^q \hilbFB
  {q'}{q-q'}_{c,\chi}$.
  Note that $\dbar^*\zeta^\perp \in \bigoplus_{q'=1}^{q-1}
  \hilbFB {q'}{q-1-q'}_{c,\chi} = \paren{\hilbFB
    0{q-1}_{c,\chi}}^\perp$, thus ${\inner
    f{\dbar^*\zeta^\perp}_{K_c,\chi}} = 0$ for any $f \in \hilbFB
  0{q-1}_{c,\chi}$.
  Therefore, for any $f \in \hilbFB 0{q-1}_{c,\chi}$, one has
  \begin{align*}
      &f \in \Dom \adjCnddt_1^* \\
      :\iff &\exists~C > 0 \colon
      \begin{aligned}[t]
        &\forall~\zeta \in \Dom\adjCnddt_1\: , \; \abs{\inner f{\adjCnddt_1\zeta}_1} =
        \abs{\inner f{\dbar^*\zeta}_{K_c,\chi}} \leq C \norm\zeta_2 
      \end{aligned}
      \\
      \iff &\exists~C > 0 \colon
      \begin{aligned}[t]
        &\forall~\tilde\zeta \in
        \Dom_{K_c,\chi}\dbar^* \: , \\ &\abs{\inner
          f{\dbar^*\tilde\zeta}_{K_c,\chi}} = \abs{\inner
          f{\dbar^*\zeta}_{K_c,\chi}} \leq C
        \norm{\tilde\zeta}_{K_c,\chi} 
      \end{aligned}
      \\
      \iff & f \in \Dom \dbar \cap \hilbFB 0{q-1}_{c,\chi} = \Dom
      T_{q-1}  \qquad\qquad \text{as }\paren{\dbar^*}^* = \dbar
  \end{align*}
  (ref.~\cite{Demailly}*{Ch.~VIII, \S 1} for the definition of the
  domain of Hilbert space adjoints), and thus $\Dom \adjCnddt_1^* = \Dom
  T_{q-1}$.
  It follows that $\inner f{\adjCnddt_1 \zeta}_1 = \inner
  f{\dbar^*\zeta}_{K_c,\chi} = \inner {\dbar f}\zeta_2 = \inner
  {T_{q-1}f} \zeta_2$ for any $f \in \Dom T_{q-1}$ and $\zeta \in
  \Dom\adjCnddt_1$.
  As a result, $(T_{q-1}, \Dom T_{q-1}) =
  (\adjCnddt_1^*,\Dom\adjCnddt_1^*)$, and hence $(T_{q-1}^*, \Dom T_{q-1}^*) =
  (\adjCnddt_1,\Dom\adjCnddt_1)$ (ref.~\cite{Demailly}*{Ch.~VIII, Thm.~1.1}).

  The proof of $(T_{q-1}^*, \Dom T_{q-1}^*) =
  (\adjCnddt_2,\Dom\adjCnddt_2)$ is similar.
  Notice that $\norm{\zeta}_2^2 = \norm{\zeta'}_2^2 +
  \norm{\zeta''}_2^2$ and thus $\norm{\zeta'}_2 + \norm{\zeta''}_2
  \leq \sqrt 2 \norm\zeta_2$ for all $\zeta = \zeta' + \zeta'' \in \hilbFBtwo[,\chi]$.
  Then, for any $f \in \hilbFB 0{q-1}_{c,\chi}$, one has
  \begin{equation*}
    \begin{aligned}
      &f \in \Dom \adjCnddt_2^* \\
      :\iff &\exists~C > 0 \colon
      \begin{aligned}[t]
        &\forall~\zeta = \zeta' + \zeta''
        \in \Dom\adjCnddt_2 \: , \\ &\abs{\inner f{\adjCnddt_2\zeta}_1} =
        \abs{\inner f{\dbarf^*\zeta' + \dbarb^*\zeta''}_{1}}
        \leq C \norm\zeta_2 
      \end{aligned}
      \\
      \iff &\exists~C > 0 \colon 
      \begin{aligned}[t]
        &\forall~\zeta' \in \Dom_{K_c,\chi}^{(1,q-1)}\dbarf^*
        \text{ and } \forall~\zeta'' \in
        \Dom_{K_c,\chi}^{(0,q)}\dbarb^* \: , \\ &
          \abs{\inner f{\dbarf^*\zeta'}_{1}} \leq C
          \norm{\zeta'}_{2}   \text{ and }
          \abs{\inner f{\dbarb^*\zeta''}_{1}} \leq C
          \norm{\zeta''}_{2} 
      \end{aligned}
      \\
      \iff & f \in \Dom_{K_c,\chi}^{(0,q-1)} \dbarf \cap
      \Dom_{K_c,\chi}^{(0,q-1)} \dbarb = \Dom T_{q-1} \; ,
    \end{aligned}
  \end{equation*}
  and thus $\Dom \adjCnddt_2^* = \Dom T_{q-1}$.
  Note that 
  $\inner f{\adjCnddt_2\zeta}_1 = \inner {\dbarf f}{\zeta'}_2
  + \inner{\dbarb f}{\zeta''}_2 = \inner{\dbarf f + \dbarb f}{\zeta' +
    \zeta''}_2 = \inner{T_{q-1}f}\zeta_2$ for $f \in \Dom T_{q-1}$ and
  $\zeta \in \Dom \adjCnddt_2$, since $\hilbFB 1{q-1}_{c,\chi}
  \perp \hilbFB 0q_{c,\chi}$.
  Therefore, one has $(T_{q-1}, \Dom T_{q-1}) =
  (\adjCnddt_2^*,\Dom\adjCnddt_2^*)$, and thus $(T_{q-1}^*, \Dom T_{q-1}^*) =
  (\adjCnddt_2,\Dom\adjCnddt_2)$ (ref.~\cite{Demailly}*{Ch.~VIII, Thm.~1.1}).
\end{proof}

}

Suppose now given $0 < c \leq \infty$ and $\psi \in \kashf{0,q}(K_c;L) \cap
\hilbbsp{0,q}_{c,\chi} \subset \hilbFBtwo[,\chi]$ such that $S_q \psi
= \dbar\psi = 0$. 
Theorem \ref{thm:existence-weak-solution}~(\ref{item:solve-T-eq}) asserts that, if the $L^2$ estimate
(\ref{eq:a-priori-estimate_general}) is satisfied, then there exists
$\xi \in \cl{\im T_{q-1}^*} \subset \hilbbsp{0,q-1}_{c,\chi}$ such
that 
\begin{equation} \label{eq:dbarb_eq}
  T_{q-1}\xi = \dbar\xi = \psi \quad\text{in }\hilbbsp{0,q}_{c,\chi} \; .
\end{equation}

One can have a further reduction. 
When $c = \infty$, 
since $(X,g)$ is complete in the sense of Riemannian
geometry, $\smFBCXtwo$ is dense in $\Dom_X T_{q-1}^* \cap \Dom_X S_q$
under the above graph norm (see, for example, \cite{Demailly}*{Ch.~VIII,
  Thm.~3.2}). 
Therefore, it suffices to establish the required $L^2$ estimates
(\ref{eq:a-priori-estimate_general}) for $\zeta \in \smFBCXtwo$.

Suppose $c < \infty$. 
Note that $\smFBCtwo \subset \Dom S_q$. 
Since $\bdry K_c$ is smooth and $\chiX$ is smooth on a
neighborhood of $\cl K_c$, using \cite{Hoermander-L2}*{Prop.~2.1.1}
together with an argument of partition
of unity, it yields the following
{
\newcommand{\wtd}{\widetilde}

\begin{prop}
  $\smFBCtwo \cap \Dom T_{q-1}^*$ is dense in $\Dom T_{q-1}^* \cap
  \Dom S_q$ under the graph norm
  $\sqrt{\norm{T_{q-1}^*\zeta}_1^2 + \norm{S_q\zeta}_3^2 +
    \norm{\zeta}_2^2}$.
\end{prop}
\begin{proof}
  Note that the statement follows from
  \cite{Hoermander-L2}*{Prop.~2.1.1} when $\cTgt{0,q}_X$ and $L$ are
  both trivial by using a partition of unity.
  The aim now is to handle the case when $L$ is non-trivial.

  Take a locally finite open cover $\set{U_\alpha}_{\alpha \in A}$ of
  $X$ such that every $U_\alpha$ is a coordinate chart of $X$ and $L$ is
  trivialized on each $U_\alpha$ with transition functions
  $\sigma_{\alpha\beta} \in \holo_X^*(U_\alpha \cap U_\beta)$ for all
  $\alpha ,\beta \in A$.
  Then, for any $\zeta \in \Ltwo{0,q}_{\chi}(X;L)$ with $\zeta_\alpha$
  representing $\zeta$ over $U_\alpha$ under the trivialization, one
  has $\zeta_\alpha = \sigma_{\alpha\beta} \zeta_\beta$ on $U_\alpha
  \cap U_\beta$.

  Fix any $\zeta \in \Dom T_{q-1}^* \cap \Dom S_q$.
  It suffices to show that $\zeta$ can be approximated by a sequence
  $\seq{\zeta^{(\nu)}}_{\nu \in \Nnum} \subset \smFBCtwo
  \cap \Dom T_{q-1}^*$ under the given graph norm.

  Extend $\zeta$ by zero to a section on $X$.
  Using a partition of unity which decomposes $\zeta$ into a sum of
  finitely many compactly supported sections, one can assume that $\zeta$ is compactly
  supported in a coordinate chart $U := U_0 \in \set{U_\alpha}_{\alpha
    \in A}$.
  Then the hermitian metric $\eta$ on $L$ can be viewed as a function
  $\wtd\eta := \eta_0$ on $U = U_0$ (under the given trivialization),
  and any $L$-valued form $f \in \Ltwo{0,q}_{g,\eta,\chi}(U;L)$ can be
  viewed as a $\holo_X$-valued form $\wtd f := f_0 \in
  \Ltwo{0,q}_{g,\wtd\eta,\chi}(U)$.
  Let $W := U \cap K_c$.
  Note that one has $\norm{\wtd f}_{W,g,\wtd\eta,\chi} =
  \norm{f}_{W,g,\eta,\chi}$, $\norm{\dbar\wtd f}_{W,g,\wtd\eta,\chi} =
  \norm{\dbar f}_{W,g,\eta,\chi}$ and $\norm{\dbar^*\wtd
    f}_{W,g,\wtd\eta,\chi} =\norm{\dbar^*f}_{W,g,\eta,\chi}$
  for all $f \in \Ltwo{0,q}_{g,\eta,\chi}(W;L)$.
  Then $\zeta \in \Dom T_{q-1}^* \cap \Dom S_q$ implies $\wtd\zeta \in
  \Dom_{W,g,\wtd\eta,\chi} \dbar^* \cap \Dom_{W,g,\wtd\eta,\chi} \dbar \cap
  \hFBtwo[g,\wtd\eta,\chi](W)$.
  Since $g$ and $\chi$ are fixed in what follows, subscripts of
  them are omitted from the notations below.


  By \cite{Hoermander-L2}*{Prop.~2.1.1} (or applying
  \cite{Hoermander-L2}*{Prop.~1.2.4} directly), there exists a sequence
  $\seq{\wtd\zeta^{(\nu)}}_{\nu\in\Nnum} \subset \smform{0,q}(\cl W) \cap
  \Dom_{W,\wtd\eta} \dbar^*$ such that 
  \begin{equation*}
    \norm{\dbar^*\paren{\wtd\zeta^{(\nu)} - \wtd\zeta}}_{W,\wtd\eta}^2 +
    \norm{\dbar\paren{\wtd\zeta^{(\nu)} - \wtd\zeta}}_{W,\wtd\eta}^2 +
    \norm{\wtd\zeta^{(\nu)} - \wtd\zeta}_{W,\wtd\eta}^2 \tendsto 0
  \end{equation*}
  as $\nu \tendsto \infty$ and $\supp\wtd\zeta^{(\nu)} \Subset U$ for
  all $\nu \in \Nnum$.
  As $\wtd\zeta^{(\nu)}$'s are obtained from convolutions between
  smoothing kernels and $\wtd\zeta$ which do not change the type of
  forms, it follows that $\wtd\zeta^{(\nu)} \in \smfbtwo(\cl W)$.
  The sections $\zeta^{(\nu)} \in \smfbtwo(\cl W;L)$ defined by
  $\zeta^{(\nu)}_\alpha := \frac{1}{\sigma_{0\alpha}} \wtd\zeta^{(\nu)}$ on
  $U_\alpha \cap U \neq \emptyset$ are compactly supported in $U$
  (hence $\zeta^{(\nu)} \in \smFBCtwo$) and
  satisfy $\wtd{\zeta^{(\nu)}} = \wtd\zeta^{(\nu)}$. 
  Therefore, one obtains a sequence $\seq{\zeta^{(\nu)}}_{\nu\in\Nnum} \subset
  \Dom_{K_c,\eta} \dbar^* \cap \smFBCtwo = \Dom T_{q-1}^* \cap
  \smFBCtwo$ (see Proposition \ref{prop:T-star}) such that
  \begin{align*}
      &~\norm{T_{q-1}^*\paren{\zeta^{(\nu)} - \zeta}}_{1}^2 +
      \norm{S_q\paren{\zeta^{(\nu)} - \zeta}}_{3}^2 +
      \norm{\zeta^{(\nu)} - \zeta}_{2}^2 \\
      \leq
      &~\norm{{\dbar^*}\paren{\zeta^{(\nu)} - \zeta}}_{W,\eta}^2 +
      \norm{\dbar\paren{\zeta^{(\nu)} - \zeta}}_{W,\eta}^2 +
      \norm{\zeta^{(\nu)} - \zeta}_{W,\eta}^2 && \text{\parbox{3cm}{as $T_{q-1}^* =
      \pr\dbar^*$  by Prop.~\ref{prop:T-star}}} \\
      =
      &~\norm{\dbar^*\paren{\wtd\zeta^{(\nu)} - \wtd\zeta}}_{W,\wtd\eta}^2 +
      \norm{\dbar\paren{\wtd\zeta^{(\nu)} - \wtd\zeta}}_{W,\wtd\eta}^2 +
      \norm{\wtd\zeta^{(\nu)} - \wtd\zeta}_{W,\wtd\eta}^2 \\
      \tendsto &~ 0 && \text{as }\nu \tendsto \infty
  \end{align*}
  as required.
\end{proof}

}  

As a result, it suffices to establish the required $L^2$ estimates
(\ref{eq:a-priori-estimate_general}) for $\zeta \in \smFBCtwo \cap
\Dom T_{q-1}^*$.

The above discussion is summarized in the following
\begin{prop} \label{prop:estimate-imply-strong-sol}
  Suppose $0 < c \leq \infty$. 
  If there exists a constant $C >0$ such that
  \begin{multline} \label{eq:a_priori_estimate}
    \norm{S_q\zeta}_3^2 + \norm{T_{q-1}^*\zeta}_1^2 \geq C \norm{\zeta}_2^2 \\
    \text{for all }\zeta \in
    \begin{cases}
      \smFBCtwo \cap \Dom T_{q-1}^* & \text{when }c <
      \infty \; , \\
      \smFBCXtwo & \text{when }c = \infty \; ,
    \end{cases} 
  \end{multline}
  then, for every $\psi \in \kashf{0,q}(K_c; L)\cap
  \hilbbsp{0,q}_{\chi}(K_c;L)$ such that $\dbar\psi = 0$, there
  exists $\xi \in \hilbbsp{0,q-1}_\chi (K_c;L)$ such that $\dbar\xi = \psi$ in
  $\hilbFB 0q_{\chi}(K_c;L)$.
\end{prop}

\begin{remark} \label{rem:regularity}
  Let $\Ltwo{0,q-1}(K_c;L;\text{loc})$ denote the space of locally $L^2$
  $L$-valued $(0,q-1)$-forms on $K_c$, which contains $\hilbFB
  0{q-1}_\chi(K_c;L)$ as a subspace.
  It follows from the classical regularity theory for $\dbar$-operator
  or elliptic operators (ref.~\cite{Hoemander}*{Thm.~4.2.5 and Cor.~4.2.6} or
  \cite{Hoemander-PDE}*{Thm.~4.1.5 and Cor.~4.1.2}) that the existence
  of $\xi \in \Ltwo{0,q-1}(K_c;L;\text{loc})$ satisfying the equation
  (\ref{eq:dbarb_eq}) in $\Ltwo{0,q}(K_c;L;\text{loc})$ implies that there exists $\xi \in
  \smform{0,q-1}(K_c;L)$ (but not necessarily in $\smfb
  0{q-1}(K_c;L)$) satisfying the same equation in $\smform{0,q}(K_c;L)$.
  In case $c = \infty$, Theorem \ref{rem:Kazama-Dolbeault} implies that
  there even exists a solution $\xi \in \kashf{0,q-1}(X;L)$ such that
  $\dbar\xi = \psi$ on $X$.
\end{remark}

\begin{remark}
  Write $\kashfL{0,q}(K_c;L) := \kashf{0,q}(K_c;L) \cap
  \hilbbsp{0,q}_{c,\chi}$.
  Following the idea discussed in \S \ref{sec:methodology}, it would
  be more natural to consider the $L^2$ estimate on
  $\hilbsp{0,q}_{c,\chi} := \cl{\kashfL{0,q}(K_c;L)}$ rather than
  $\hilbFBtwo[,\chi]$, where the closure is taken in $\hilbbsp{0,q}_{c,\chi}$. 
  However, the author faces the difficulty in obtaining the required estimate from the
  Bochner--Kodaira inequalities when $\hilbsp{0,q}_{c,\chi}$ instead of
  $\hilbFBtwo[,\chi]$ is considered. Write $\dksh c^*$ as the Hilbert space
  adjoint of $\dbar = \dbarb \colon \hilbsp{0,q}_{c,\chi} \to
  \hilbsp{0,q+1}_{c,\chi}$. 
  It can be shown that $\dksh c^* = \pr_c \circ \dbarb^*$ on
  $\Dom_{K_c,\chi}^{(0,q)} \dksh c^*$, where $\pr_c \colon \hilbbsp{0,q}_{c,\chi} \to
  \hilbsp{0,q}_{c,\chi}$ is the orthogonal projection. Set $\dperp c :=
  \dbarb^* - \dksh c^*$, then $\dksh c^*\zeta$ and $\dperp c \zeta$
  are orthogonal to each other for all $\zeta \in \Dom_{K_c,\chi}^{(0,q)} \dksh
  c^*$ and 
  \begin{equation*} 
    \norm{\dbarb^*\zeta}_{K_c,\chi}^2 = \norm{\dksh c^*\zeta}_{K_c,\chi}^2 +
    \bignorm{\dperp c\zeta}_{K_c,\chi}^2 \; .
  \end{equation*}
  From the Bochner--Kodaira inequalities, one obtains
  \begin{equation*}
    \norm{\dbar\zeta}_{K_c,\chi}^2 + \norm{\dbarb^*\zeta}_{K_c,\chi}^2 \geq \int_{K_c}
    \operatorname{Curv}(\zeta, \zeta)
  \end{equation*}
  for all $\zeta \in \kashfL{0,q}(K_c;L) \cap \Dom_{K_c,\chi}^{(0,q)}\dbar \cap
  \Dom_{K_c,\chi}^{(0,q)} \dbarb^*$, where $\int_{K_c} \operatorname{Curv}(\zeta,\zeta)$ is the
  curvature term arising from the curvature of $L$. By choosing
  suitably the metrics $g$ and $\eta$, the curvature term can be
  bounded below by $C \norm\zeta_{K_c,\chi}^2$ for some constant
  $C>0$. Therefore, in order to obtain the desired estimate
  $\norm{\dbar\zeta}_{K_c,\chi}^2 + \norm{\dksh c^*\zeta}_{K_c,\chi}^2 \geq
  C''\norm\zeta_{K_c,\chi}^2$ for some constant $C'' > 0$, one has 
  to show that $\norm{\dperp c\zeta}_{K_c,\chi}^2 \leq C'
  \norm{\zeta}_{K_c,\chi}^2$ for some constant $C'>0$ such that $C >
  C'$. However, the constant $C'$ depends on $g$ in general and one
  may not be able to make $C'$ smaller than $C$ by altering
  $g$. That's why the $L^2$ estimate on $\hilbFBtwo[,\chi]$ instead of
  $\hilbsp{0,q}_{c,\chi}$ is considered in this article.
\end{remark}


\subsection{Bochner--Kodaira formulas}

Let
\begin{equation*}
  \nabla \colon \smform{}(\cTgt{\bullet,\bullet} \otimes L) \to
  \smform{}(\cTgt{\fieldC} \otimes \cTgt{\bullet,\bullet} \otimes
  L) \; ,
\end{equation*}%
\showsym[nabla]{$\nabla$}{the connection on $\cTgt{\bullet,\bullet}
  \otimes L$ induced from the Chern connections on the holomorphic
  hermitian vector bundles $(\Tgt{1,0},g)$ and $(L, \eta e^{-\chi})$}%
where $\cTgt{\fieldC} := \cTgt{1,0} \oplus \cTgt{0,1}$, be the
connection on $\cTgt{\bullet,\bullet} \otimes L$ induced from the Chern
connections on the holomorphic hermitian vector bundles
$(\Tgt{1,0},g)$ and $(L, \eta e^{-\chi})$.
Therefore, $\nabla$ is
compatible with the pointwise norm $\abs\cdot_{g,\eta,\chi}$.

Under a chosen apt coordinate system, set $\diff_k := \frac{\diff}{\diff z^k}$
and $\diff_{\conj k} := \frac{\diff}{\diff\conj{z^k}}$ for $1\leq k \leq
n$\showsym[diffk]{$\diff_{\conj k}$ (resp.~$\diff_k$)}{$:=
  \diff/\diff\conj{z^k}$ (resp.~$\diff/\diff z^k$) for $1\leq k \leq
  n$}.
These define global vector fields on $X$.
Set $\nabla_k := \nabla_{\diff_k}$ and $\nabla_{\conj k} :=
\nabla_{\diff_{\conj k}}$\showsym[nablak]{$\nabla_{\conj k}$
  (resp.~$\nabla_k$)}{$:= \nabla_{\diff_{\conj k}}$ (resp.~$\nabla_{\diff_k}$)}
for $1\leq k \leq n$.
Set also $\nabla_{v^j} := \nabla_{n-m+j} = \nabla_{\frac{\diff}{\diff
    v^j}}$ and $\nabla_{\conj{v^j}} := \nabla_{\conj{n-m+j}} =
\nabla_{\frac{\diff}{\diff \conj{v^j}}}$\showsym[nablakj]{$\nabla_{\conj{v^j}}$
  (resp.~$\nabla_{v^j}$)}{$:= \nabla_{\conj{n-m+j}} =
  \nabla_{\frac{\diff}{\diff \conj{v^j}}}$
  (resp.~$\nabla_{n-m+j} = \nabla_{\frac{\diff}{\diff v^j}}$) for
  $1\leq j \leq m$} (and define $\diff_{v^j}$ and
$\diff_{\conj{v^j}}$\showsym[diffkj]{$\diff_{\conj{v^j}}$
  (resp.~$\diff_{v^j}$)}{$:= \diff_{\conj{n-m + j}} =
  \frac{\diff}{\diff \conj{v^j}}$ (resp.~$\diff_{n-m + j} =
  \frac{\diff}{\diff v^j}$)} similarly) for $1 \leq j \leq m$ for
notational convenience.
Since the hermitian metric $g$ is translational invariant on $X$,
the Christoffel symbols given from $g$ vanish and thus one has locally
\begin{equation} \label{eq:connection-in-local-coordinates}
  \begin{gathered}
    \nabla_k = \diff_k + \diff_k \log\paren{\eta e^{-\chiX}} \; ,  \\
    \nabla_{\conj k} = \diff_{\conj k} \;
  \end{gathered}
\end{equation}
for $1\leq k \leq n$. 
For later use, note that the commutator of $\nabla_k$ and
$\nabla_{\conj \ell}$ is given by
\begin{equation*}
  \Theta_{k\conj \ell} := \commut{\nabla_k}{\nabla_{\conj \ell}} =
  -\diff_k\diff_{\conj \ell} \log\paren{\eta e^{- \chiX}} \; ,
\end{equation*}
and the curvature form of $L$ endowed with the metric $\eta e^{-\chiX}$ is given
by
\begin{equation} \label{eq:curvature-tensor}
  \Theta := -\cplxi \diff\dbar \log\paren{\eta e^{-\chiX}} = \cplxi
  \sum_{k,\ell=1}^n \Theta_{k\conj \ell}~dz^k \wedge d\conj{z^\ell} \;
  .
\end{equation}%
\showsym[Theta]{$\Theta$}{$:= \cplxi\sum_{k,\ell = 1}^n \Theta_{k\conj
    \ell}~dz^k \wedge d\conj{z^\ell} := -\cplxi \diff\dbar
  \log\paren{\eta e^{-\chiX}}$}%
Write the curvature tensor associated to $\Theta$ as
\begin{equation*}
  \oTheta := \sum_{k,\ell = 1}^n \Theta_{k\conj \ell}~dz^k \otimes
  d\conj{z^\ell} \; .
\end{equation*}%
\showsym[RTheta]{$\oTheta$}{$:= \sum_{k,\ell = 1}^n \Theta_{k\conj
    \ell}~dz^k \otimes d\conj{z^\ell}$}%

Since the base and fibre directions are orthogonal to each other with
respect to $g$, the identification between $\smform{p,q}$ and
$\conj{\smcform{p,q}} = \smcform{q,p} :=
\smooth\paren{\Tgt{q,p}}$\showsym[A3]{$\smcform{p,q}$}{$:=
  \smooth\paren{\Tgt{p,q}}$} induced from $g$ respects the
decomposition (\ref{eq:cotangent-split}) ($\conj{\smcform{p,q}}$ here
means the complex conjugate of $\smcform{p,q}$).
For later use, set $\smcFfb {p'}{p''}{q'}{q''} := \smooth\paren{\Tgtf{p',q'} \wedge
    \Tgtb{p'',q''}}$\showsym[A4]{$\smcFfb {p'}{p''}{q'}{q''}$}{$:=
    \smooth\paren{\Tgtf{p',q'} \wedge \Tgtb{p'',q''}}$} and $\smcfb
  {p'}{p''} := \smcFfb {p'}{p''}00$\showsym[A4a]{$\smcfb {p'}{p''}$}{$:=
    \smcFfb {p'}{p''}00$} for $0 \leq p',q' \leq n-m$ and $0 \leq
  p'',q'' \leq m$.
For any $\zeta \in \smform{p,0} \otimes \smform{0,q}$,
let $\zeta^\dual$\showsym[vdual]{${\cdot}^\dual$}{isomorphism
  $\smform{p,0} \otimes \smform{0,q} \xrightarrow{\isom} \smcform{0,p}
  \otimes \smcform{q,0}$ induced from $g$} denote the image of $\zeta$
in $\smcform{0,p} \otimes \smcform{q,0}$ via the isomorphism induced
from $g$.
Then, for example, if $\zeta \in \smfb {q'}{q''}$, one has
$\zeta^\dual \in \smcfb {q'}{q''}$.

As a bilinear form on $\smcform{1,0}\otimes \conj{\smcform{1,0}}$,
$\oTheta$ can be
decomposed according to the decomposition (\ref{eq:cotangent-split})
into the sum of
\begin{equation*}
  \begin{aligned}
    \oTheta_{u\conj u} := \oTheta|_{\smcfb 10 \otimes \conj{\smcfb 10}} \; , && 
    \oTheta_{u\conj v} := \oTheta|_{\smcfb 10 \otimes \conj{\smcfb 01}} \; , \\
    \oTheta_{v\conj u} := \oTheta|_{\smcfb 01 \otimes \conj{\smcfb 10}} \; , &&
    \oTheta_{v\conj v} := \oTheta|_{\smcfb 01 \otimes \conj{\smcfb 01}} \; .
  \end{aligned}
\end{equation*}%
\showsym[RThetauv]{$\oTheta_{u\conj u} + \oTheta_{u\conj v} +
  \oTheta_{v\conj u} + \oTheta_{v\conj v}$}{decomposition of $\oTheta$
with respect to \\ the decomposition (\ref{eq:cotangent-split})}%
Since $\oTheta$ is a hermitian form, it follows that $\oTheta_{u\conj u} =
\conj{\oTheta_{u\conj u}}$, $\oTheta_{v\conj v} =
\conj{\oTheta_{v\conj v}}$ and $\oTheta_{u\conj v} =
\conj{\oTheta_{v\conj u}}$.

Let $\Tr_{g} \colon \smform{0,q} \otimes \smform{q,0} \to
\smform{0,0}$\showsym[Tg]{$\Tr_g$}{the complete contraction
  $\smform{0,q} \otimes \smform{q,0} \to \smform{0,0}$ via $g$} be the
trace operator which is defined in such a way that $\zeta \otimes \xi
\mapsto \xi^\dual \ctrct \zeta$, where $\zeta \in \smform{0,q}$, $\xi
\in \smform{q,0}$ and $\xi^\dual \ctrct \zeta$ denotes
the complete contraction between $\zeta$ and $\xi^\dual$. 
Denote by $\Tr_{g,\eta}$ the similar contraction for
$L$-valued forms. 

Fix any $0 < c < \infty$.
Denote the Hilbert space adjoint of $\dbar \colon
\Ltwo{0,q-1}_{c,\chi} \to \Ltwo{0,q}_{c,\chi}$ by $\dbar^* \colon
\Ltwo{0,q}_{c,\chi} \to \Ltwo{0,q-1}_{c,\chi}$.
Identify $\smform{1,1}$ and $\smform{1,0} \otimes \smform{0,1}$ via
the isomorphism $dz^k \wedge d\conj{z^\ell} \mapsto dz^k \otimes
d\conj{z^\ell}$ for any $1\leq k, \ell \leq n$. 
Let $\oTheta^\dual(\zeta \otimes \conj \zeta)$
(resp.~$\paren{\ddbar\varphi}^\dual(\zeta \otimes \conj \zeta)$)
denotes the natural contraction between $\oTheta^\dual$
(resp.~$\paren{\ddbar\varphi}^\dual$) and $\zeta \otimes \conj
\zeta$.
Let $\nabla = \nablalo + \nablaol$\showsym[nablalo]{$\nabla = \nablalo
  + \nablaol$}{the decomposition of $\nabla$ into $(1,0)$- and
  $(0,1)$-types} be the decomposition of $\nabla$ into $(1,0)$- and
$(0,1)$-types.
The $\conj\nabla$-Bochner--Kodaira formula (cf.~\cite{Siu}*{(2.1.4)
  and (1.3.3)})
is then given by
\begin{equation} \label{eq:classical-01-BK}
  \begin{aligned}
    \norm{\dbar\zeta}_{K_c,\chi}^2 + \norm{\dbar^*\zeta}_{K_c,\chi}^2
    = &~\int_{\bdry K_c} \tfrac{e^{-\chi}}{\abs{d\varphi}_g} \Tr_{g,\eta}
    \paren{\ddbar\varphi}^\dual(\zeta \otimes \conj \zeta) \\
    &~+ \norm{\nablaol \zeta}_{K_c,\chi}^2 
    + \int_{K_c} e^{-\chi} \Tr_{g,\eta}
    \oTheta^\dual(\zeta \otimes \conj \zeta)
  \end{aligned}
\end{equation}
for all $\zeta \in \smform{0,q}(\cl K_c;L) \cap \Dom_{K_c,\chi}
\dbar^*$.
\begin{remark}
  Note that the measure for the boundary integral is induced from
  $\parres{\frac{(d\varphi)^\dual}{\abs{d\varphi}_g} \ctrct
  \frac{\omega^{\wedge n}}{n!}}_{\bdry K_c}$.
  In order to compare notations in \cite{Siu}*{(2.1.4)} and those in
  (\ref{eq:classical-01-BK}), write $\siu{x}$ to mean the symbol $x$
  used in \cite{Siu}.
  Then
  \begin{gather*}
    \siu{\conj\nabla} = \nablaol \; , \quad \siu{\nabla} = \nablalo \;
    , \quad \siu\rho = \frac{\varphi - c}{\abs{d\varphi}_g} \; , \quad
    \siu{R_{i\conj j k \conj l}} = 0  \; , \\
    \text{and }\quad
    \siu{-\Omega_{\alpha \conj\beta s\conj t}} = \text{components of
    } \oTheta = \Theta_{k \conj \ell} \; .
  \end{gather*}
  Note that $\siu{R_{i\conj j k \conj l}} = 0$ as the Chern connection
  on $(\Tgt{1,0},g)$ is flat.
  Also be aware of the typos of the signs preceding the curvature
  integrals involving $\siu{\Omega_{\alpha \conj \beta \hphantom{s}
      \conj t}^{\hphantom{\alpha \conj \beta} \conj s}}$ and
  $\siu{R^{\conj s}_{\hphantom{s} \conj t}}$ in \cite{Siu}*{(2.1.4)}.
  The correct signs can be found in \cite{Siu}*{(1.3.3)}.
  To see that the boundary term in (\ref{eq:classical-01-BK})
  coincides with the one in \cite{Siu}*{(2.1.4)}, note that at every
  $z \in \bdry K_c$,
  \begin{equation*}
    \ddbar\paren{\frac{\varphi - c}{\abs{d\varphi}_g}} (z)
    = \frac{\ddbar\varphi}{\abs{d\varphi}_g}(z) 
    - \frac{\diff\varphi \wedge
        \dbar\abs{d\varphi}_g}{\abs{d\varphi}_g^2}(z) 
    - \frac{\diff\abs{d\varphi}_g \wedge
        \dbar\varphi}{\abs{d\varphi}_g^2}(z) \; .
  \end{equation*}
  After taking ${}^\dual$ and contracting with $\zeta \otimes
  \conj\zeta$ where $\zeta \in \smform{0,q}(\cl K_c;L) \cap
  \Dom_{K_c,\chi} \dbar^*$, the last two terms on the right hand side
  vanish because, for $\zeta \in \smform{0,q}(\cl K_c;L)$,
  $\paren{\diff\varphi}^\dual \ctrct \zeta = 0$ on $\bdry K_c$ if and
  only if $\zeta \in \Dom_{K_c,\chi} \dbar^*$
  (ref.~\cite{Hoermander-L2}*{pg.~101} or \cite{Siu}*{(2.1.1)}).
  The boundary terms therefore coincides.
\end{remark}


When the subcomplex (\ref{eq:short-dbar-complex}) is considered,
the $\conj\nabla$-Bochner--Kodaira formula (\ref{eq:classical-01-BK})
is restricted to $\zeta \in \smFBCtwo \cap \Dom_{K_c,\chi} \dbar^* =
\smFBCtwo \cap \Dom_{K_c,\chi} T_{q-1}^*$ (see Proposition
\ref{prop:T-star}).
%
%
The $(0,1)$-connect\-ion splits into $\nablaol = \nablaol_u +
\nablaol_v$ according to the decomposition (\ref{eq:cotangent-split}).
Write $\nabla_{\conj u} := \nablaol_u$ and $\nabla_{\conj v} :=
\nablaol_v$\showsym[nablau-]{$\nablaol = \nabla_{\conj u} +
  \nabla_{\conj v}$}{decomposition according to the decomposition
  (\ref{eq:cotangent-split})} for notational convenience.
Let also $\pr_F \colon \smform{0,q} \otimes
\conj{\smform{0,s}} \to \smbform{0,q} \otimes
\conj{\smbform{0,s}}$\showsym[prF]{$\pr_F$}{the canonical projection \\
   $\smform{0,q} \otimes \conj{\smform{0,s}} \to \smbform{0,q} \otimes
  \conj{\smbform{0,s}}$} be the canonical projection (where
$\conj{\smform{0,s}}$ (resp.~$\conj{\smfb 0s}$) is the complex conjugate of
$\smform{0,s}$ (resp.~$\smfb 0s$)).
Set 
\begin{equation} \label{eq:boundary-term}
  \Bd(\zeta,\zeta) := 
  \int_{\bdry K_c} \tfrac{e^{-\chiX}}{\abs{d\varphi}_g} \Tr_{g,\eta}
  \paren{\ddbar\varphi}^\dual (\zeta \otimes \conj{\zeta})
\end{equation}
for notational convenience.
Then (\ref{eq:classical-01-BK}) gives the following
\begin{lemma} \label{lem:BK-formula-01}
  For any $\zeta = \zeta' + \zeta'' \in \smFBCtwo \cap \Dom
  T_{q-1}^*$, 
  where $\zeta' \in \smfb 1{q-1}(\cl K_c;L) \cap \Dom
  \dbarf^*$ and $\zeta'' \in \smbform{0,q}(\cl K_c; L)$,
  one has
  \begin{equation} \label{eq:Bochner-Kodaira-01}
    \begin{aligned}
      \norm{S_q\zeta}_3^2 + \norm{T_{q-1}^*\zeta}_1^2 
      = &~\Bd(\zeta,\zeta) 
      + \norm{\dbarf\zeta''}_{3}^2 + \norm{\dbarb\zeta'}_{3}^2 \\
      &+ \norm{\nabla_{\conj u}\zeta'}_{K_c,\chi}^2 
      + \norm{\nabla_{\conj v}\zeta''}_{K_c,\chi}^2 \\
      &+ \int_{K_c} e^{-\chiX} \Tr_{g,\eta}
      \pr_F\paren{\oTheta^\dual(\zeta \otimes \conj \zeta)}  \; .
    \end{aligned}
  \end{equation}
\end{lemma}

\begin{proof}
  On $\Dom_{K_c,\chi} \dbar^*$, one has $\dbar^* = \ethf + \ethb$.
  Then, for all $\zeta = \zeta' + \zeta'' \in \Dom T_{q-1}^* =
  \Dom_{K_c,\chi} \dbar^* \cap \hilbFBtwo[,\chi]$ (see Proposition
  \ref{prop:T-star}),
  one has
  \begin{equation*}
    \dbar^* \zeta = \ethf\zeta' + \ethf\zeta'' + \ethb\zeta' +
    \ethb\zeta'' = T_{q-1}^* \zeta + \ethb\zeta' \; ,
  \end{equation*}
  as $T_{q-1}^* \zeta = \dbarf^*\zeta' + \dbarb^*\zeta''$ (see
  Proposition \ref{prop:T-star}) and $\ethf \zeta'' = 0$.
%
  Note also that
  $\nablaol \zeta = \nabla_{\conj
    u}\zeta' + \nabla_{\conj u}\zeta'' + \nabla_{\conj v}\zeta' +
  \nabla_{\conj v}\zeta''$, and $\dbar\zeta = S_q\zeta$.
  Since the decomposition (\ref{eq:cotangent-split}) is orthogonal
  with respect to $g$, it follows that
  \begin{align*}
    \norm{\dbar^*\zeta}_{K_c,\chi}^2 &=
    \norm{T_{q-1}^*\zeta}_1^2 + \norm{\ethb\zeta'}_{K_c,\chi}^2 
    \quad\text{and } \\
    \norm{\nablaol\zeta}_{K_c,\chi}^2 &= 
    \norm{\nabla_{\conj u}\zeta'}_{K_c,\chi}^2 
    + \norm{\nabla_{\conj u}\zeta''}_{K_c,\chi}^2 
    + \norm{\nabla_{\conj v}\zeta'}_{K_c,\chi}^2 
    + \norm{\nabla_{\conj v}\zeta''}_{K_c,\chi}^2 \; .
  \end{align*}
  Note that $\norm{\nabla_{\conj u}\zeta''}_{K_c,\chi}^2 =
  \norm{\dbarf\zeta''}_3^2$.

  Following the argument in \cite{Hoermander-L2}*{pg.~101} with
  $\dbarb$ in place of $\dbar$, it follows that, for any $\zeta \in
  \smfb {q'}{q''}(\cl K_c;L)$,
    $\zeta \in \Dom_{K_c,\chi}^{(q',q'')} \dbarb^*$ if and only if
    $(\diffb\varphi)^\dual \ctrct \zeta = 0$ on $\bdry K_c$.
  Since $\diffb \varphi = 0$, it follows that $\smfb {q'}{q''}(\cl K_c;L)
  \subset \Dom_{K_c,\chi}^{(q',q'')} \dbarb^*$. 
  In particular, $\zeta' \in \Dom_{K_c,\chi}^{(1,q-1)} \dbarb^*$ for
  all $\zeta' \in \smfb 1{q-1}(\cl K_c;L)$.
  Then, since the decomposition (\ref{eq:cotangent-split}) is
  orthogonal with respect to $g$, by taking the analogy between the decompositions
  $\smform r = \bigoplus_{p+q =r} \smform{p,q}$ and $\smform{p,q} =
  \bigoplus_{\substack{p = p' + p'' \\ q = q' + q''}}
  \smFfb{p'}{p''}{q'}{q''}$ and putting $\dbarb$ in place of $\dbar$,
  one can follow the derivation of 
  (\ref{eq:classical-01-BK}) as in \cite{Siu}*{\S 1 and \S 2} to obtain
  \begin{equation*}
    \norm{\dbarb\zeta'}_{K_c,\chi}^2 +
    \norm{\ethb\zeta'}_{K_c,\chi}^2 
    =
    \begin{aligned}[t]
      &\int_{\bdry K_c} \tfrac{e^{-\chiX}}{\abs{d\varphi}_g}
      \Tr_{g,\eta} \paren{\diffb\dbarb\varphi}^\dual \paren{\zeta'
        \otimes \conj{\zeta'}} \\
      + &\norm{\nabla_{\conj
          v}\zeta'}_{K_c,\chi}^2 + \int_{K_c} e^{-\chiX} \Tr_{g,\eta}
      \oTheta_{v\conj v}^\dual (\zeta' \otimes \conj{\zeta'})
    \end{aligned}
  \end{equation*}
  for any $\zeta' \in \smfb 1{q-1}(\cl K_c;L)$.
  The boundary term vanishes as $\diffb\dbarb\varphi = 0$.
  Therefore, combining the above results with (\ref{eq:classical-01-BK}) yields
  \begin{equation*}
    \begin{aligned}
      \norm{S_q\zeta}_3^2 + \norm{T_{q-1}^*\zeta}_1^2 
      = &~\Bd(\zeta,\zeta) 
      + \norm{\dbarf\zeta''}_{3}^2 + \norm{\dbarb\zeta'}_{3}^2
      + \norm{\nabla_{\conj u}\zeta'}_{K_c,\chi}^2 
      + \norm{\nabla_{\conj v}\zeta''}_{K_c,\chi}^2 \\
      &+ \int_{K_c} e^{-\chiX} \Tr_{g,\eta} \oTheta^\dual(\zeta \otimes
      \conj \zeta) - \int_{K_c} e^{-\chiX} \Tr_{g,\eta} \oTheta_{v\conj
        v}^\dual (\zeta' \otimes \conj{\zeta'}) \; .
    \end{aligned}
  \end{equation*}

  For every fixed $z \in K_c$, $\Tr_{g,\eta} \oTheta^\dual (\zeta \otimes \conj\zeta)$
  is a hermitian form in $\zeta$.
  Again, since the decomposition (\ref{eq:cotangent-split}) is
  orthogonal with respect to $g$, it follows that
  \begin{align*}
    \Tr_{g,\eta} \oTheta^\dual (\zeta' \otimes \conj{\zeta'})
    &= \Tr_{g,\eta} \oTheta_{u\conj u}^\dual (\zeta' \otimes
    \conj{\zeta'}) 
    + \Tr_{g,\eta} \oTheta_{v\conj v}^\dual (\zeta' \otimes
    \conj{\zeta'}) \; , \\
    \conj{\Tr_{g,\eta} \oTheta^\dual (\zeta' \otimes \conj{\zeta''})}
    &= \Tr_{g,\eta} \oTheta^\dual (\zeta'' \otimes \conj{\zeta'}) 
    = \Tr_{g,\eta} \oTheta_{v\conj u}^\dual (\zeta'' \otimes
    \conj{\zeta'}) \; , \\
    \Tr_{g,\eta} \oTheta^\dual (\zeta'' \otimes \conj{\zeta''})
    &= \Tr_{g,\eta} \oTheta_{v\conj v}^\dual (\zeta'' \otimes
    \conj{\zeta''}) \; .
  \end{align*}
  On the other hand, one has
  \begin{equation*}
    \begin{aligned}
      \pr_F \paren{\oTheta^\dual (\zeta' \otimes \conj{\zeta'})} 
      &= \oTheta_{u\conj u}^\dual (\zeta' \otimes \conj{\zeta'}) \; , &
      \pr_F \paren{\oTheta^\dual (\zeta' \otimes \conj{\zeta''})} 
      &= \oTheta_{u\conj v}^\dual (\zeta' \otimes
      \conj{\zeta''}) \; , \\
      \pr_F \paren{\oTheta^\dual (\zeta'' \otimes \conj{\zeta'})} 
      &= \oTheta_{v\conj u}^\dual (\zeta'' \otimes \conj{\zeta'}) \; , &
      \pr_F \paren{\oTheta^\dual (\zeta'' \otimes \conj{\zeta''})} 
      &= \oTheta_{v\conj v}^\dual (\zeta'' \otimes \conj{\zeta''}) \; . 
    \end{aligned}  
  \end{equation*}
  Therefore, it follows that
  \begin{equation*}
    \begin{aligned}
      \Tr_{g,\eta} \oTheta^\dual (\zeta \otimes \conj{\zeta}) 
      - \Tr_{g,\eta} \oTheta_{v\conj v}^\dual (\zeta' \otimes
      \conj{\zeta'}) 
      &=
      \begin{aligned}[t]
        &\Tr_{g,\eta} \oTheta_{u\conj u}^\dual (\zeta' \otimes
        \conj{\zeta'}) 
         + \Tr_{g,\eta} \oTheta_{u\conj v}^\dual (\zeta'
         \otimes \conj{\zeta''}) 
        \\
        &~+ \Tr_{g,\eta} \oTheta_{v\conj u}^\dual (\zeta''
        \otimes \conj{\zeta'})
        + \Tr_{g,\eta}
        \oTheta_{v\conj v}^\dual (\zeta'' \otimes
        \conj{\zeta''})
      \end{aligned}
      \\
      &= \Tr_{g,\eta} \pr_F \paren{\oTheta^\dual (\zeta \otimes \conj{\zeta})}
    \end{aligned}
  \end{equation*}
  and hence the lemma.
\end{proof}

Let $g_F := \pr_F g$, and let $(g_F)^{\conj j j'}$'s for $1 \leq j, j'
\leq m$ be the entries of the inverse of the matrix of $g_F$ under
the chosen coordinates.
Denote by $\ptinner\cdot\cdot_{g,\eta,\chi}$ the pointwise inner
product induced from $\abs\cdot_{g,\eta,\chi}$.
Write $\nablalo = \nablalo_u +
\nablalo_v$\showsym[nablav]{$\nablalo = \nabla_u +
  \nabla_v$}{decomposition according to the decomposition
  (\ref{eq:cotangent-split})} as the splitting of
$\nablalo$ according to the decomposition (\ref{eq:cotangent-split}),
and set $\nabla_u := \nablalo_u$ and $\nabla_v := \nablalo_v$ for
convenience.
The following integration by parts argument is put into a lemma for clarity.
\begin{lemma}
  For all $\zeta'' \in \smfb 0q(\cl K_c;L)$, one has
  \begin{equation*}
    \norm{\nabla_{\conj v}\zeta''}_{K_c,\chi}^2 =
    \norm{\nabla_v\zeta''}_{K_c,\chi}^2 -  \int_{K_c} \paren{\Tr_g
      \oTheta_{v\conj v}} \abs{\zeta''}_{g,\eta,\chi}^2 \; .
  \end{equation*}
\end{lemma}
\begin{proof}
  Recall that $d\mu := \frac{\omega^{\wedge n}}{n!}$ is the volume element
  on $K_c$, while that on $\bdry K_c$ is given by $d\sigma :=
  \parres{\frac{(d\varphi)^\dual}{\abs{d\varphi}_g} \ctrct
    d\mu}_{\bdry K_c}$.
  Einstein summation convention is applied in what follows.
  Fix a $\zeta'' \in \smfb 0q(\cl K_c;L)$.
  Let
  \begin{align*}
    Y &:= \ptinner{\nabla_{\conj{v^j}}
      \zeta''}{\zeta''}_{g,\eta,\chi} (g_F)^{\conj j j'}
    \frac{\diff}{\diff v^{j'}} &
    \text{and}  &&
    W &:= \ptinner{\nabla_{v^{j'}}
      \zeta''}{\zeta''}_{g,\eta,\chi} (g_F)^{\conj j j'}
    \frac{\diff}{\diff \conj{v^j}} 
  \end{align*}
  be two vector fields in $\smcFfb 0100(\cl K_c)$ and $\smcFfb
  0001(\cl K_c)$ respectively.
  Then, using the fact that $\nabla$ is compatible with
  $\ptinner\cdot\cdot_{g,\eta,\chi}$, it follows that
  \begin{align*}
      \abs{\nabla_{\conj v}\zeta''}_{g,\eta,\chi}^2 d\mu 
      = &\paren{\diff_{v^{j'}}
      \ptinner {(g_F)^{\conj j j'}
        \nabla_{\conj{v^j}}\zeta''}{\zeta''}_{g,\eta,\chi}} d\mu 
    - \ptinner{(g_F)^{\conj jj'} \nabla_{v^{j'}} \nabla_{\conj{v^j}}
        \zeta''}{\zeta''}_{g,\eta,\chi} d\mu \\
      = &~
      \begin{aligned}[t]
        d\paren{Y \ctrct d\mu} -&~\ptinner{(g_F)^{\conj jj'}
          \nabla_{\conj{v^j}} \nabla_{v^{j'}}
          \zeta''}{\zeta''}_{g,\eta,\chi} d\mu \\
        &- \ptinner{(g_F)^{\conj
            jj'} \commut{\nabla_{v^{j'}}}{\nabla_{\conj{v^j}}}
          \zeta''}{\zeta''}_{g,\eta,\chi} d\mu
      \end{aligned} \\
      = &~
      \begin{aligned}[t]
        d\paren{Y \ctrct d\mu} 
        -&~\paren{\diff_{\conj{v^j}} \ptinner{(g_F)^{\conj j j'}
          \nabla_{v^{j'}}\zeta''}{\zeta''}_{g,\eta,\chi}} d\mu 
        + \abs{\nabla_v\zeta''}_{g,\eta,\chi}^2 d\mu \\
        &- \paren{\Tr_g \oTheta_{v\conj v}}
        \abs{\zeta''}_{g,\eta,\chi}^2 d\mu
      \end{aligned} \\
      = &~d\paren{Y \ctrct d\mu} 
      - d\paren{W \ctrct d\mu} 
      + \abs{\nabla_v\zeta''}_{g,\eta,\chi}^2 d\mu 
      - \paren{\Tr_g \oTheta_{v\conj v}}
      \abs{\zeta''}_{g,\eta,\chi}^2 d\mu \; .
  \end{align*}
  Since $\bdry K_c = \set{\varphi = c}$ and $\parres{d\varphi}_{\bdry K_c} = 0$,
  it follows that for any vector field $V$ such that $\ptinner V
  {\frac{(d\varphi)^\dual}{\abs{d\varphi}_g}}_g = 0$, one has
  $\parres{V \ctrct d\mu}_{\bdry K_c} = 0$.
  The component of $Y-W$ in the direction of
  $\frac{(d\varphi)^\dual}{\abs{d\varphi}_g}$ is $\ptinner{Y -
    W}{\frac{(d\varphi)^\dual}{\abs{d\varphi}_g}}_g$.
  Therefore, by integrating over $K_c$ and applying Stokes' theorem,
  it yields
  \begin{equation*}
    \begin{aligned}
      \norm{\nabla_{\conj v}\zeta''}_{K_c,\chi}^2 = \int_{\bdry K_c}
      \ptinner{Y - W}{\frac{(d\varphi)^\dual}{\abs{d\varphi}_g}}_g
      d\sigma 
      + \norm{\nabla_v\zeta''}_{K_c,\chi}^2 - \int_{K_c}
      \paren{\Tr_g \oTheta_{v\conj v}} \abs{\zeta''}_{g,\eta,\chi}^2
      d\mu \; .
    \end{aligned}
  \end{equation*}
  But $(d\varphi)^\dual \in \smcFfb 1000 \oplus \smcFfb 0010 (X)$, so
  $\ptinner{Y - W}{\frac{(d\varphi)^\dual}{\abs{d\varphi}_g}}_g = 0$
  and hence the lemma.
\end{proof}

Combining the result above with (\ref{eq:Bochner-Kodaira-01}) yields
\begin{equation} \label{eq:Bochner-Kodaira-10}
  \begin{aligned}
    \norm{S_q\zeta}_3^2 + \norm{T_{q-1}^*\zeta}_1^2 
    = &~\Bd(\zeta,\zeta) 
    + \norm{\dbarf\zeta''}_{3}^2 + \norm{\dbarb\zeta'}_{3}^2 \\
    &+ \norm{\nabla_{\conj u}\zeta'}_{K_c,\chi}^2 
    + \norm{\nabla_{v}\zeta''}_{K_c,\chi}^2
    -  \int_{K_c} \paren{\Tr_g\oTheta_{v\conj v}}
    \abs{\zeta''}_{g,\eta,\chi}^2 \\
    &+ \int_{K_c} e^{-\chiX} \Tr_{g,\eta} \pr_F \paren{\oTheta^\dual(\zeta \otimes
    \conj \zeta)} \; .
  \end{aligned}
\end{equation}
This formula is analogous to the usual $\nabla$-Bochner--Kodaira
formula (see \cite{Siu}*{(2.2.1)}). However, it contains term involving
$\nabla_{\conj u}$ and $\nabla_v$ but not $\nabla_u$, and the boundary
term is the same as the one in the $\conj\nabla$-Bochner--Kodaira
formula.

Consider the boundary term $\Bd(\zeta,\zeta)$ in (\ref{eq:boundary-term}). 
Since $\cplxi \diff\dbar \varphi$ is non-negative on $\bdry K_c$,
i.e.~$K_c$ is pseudoconvex, by choosing coordinates at any point in
$\bdry K_c$ such that $\iddbar\varphi$ and $g$ are simultaneously
diagonalized, one sees that $\Bd(\zeta,\zeta)$ is non-negative for all
$\zeta \in \smFBCtwo$.
Noting that all other norm-square terms are also
non-negative, one then obtains 
\begin{equation} \label{eq:Bochner-Kodaira_ineq}
  \norm{S_q\zeta}_3^2 + \norm{T_{q-1}^*\zeta}_1^2 
  \geq 
  \int_{K_c} e^{-\chiX} \Tr_{g,\eta} \pr_F \paren{\oTheta^\dual(\zeta \otimes
  \conj \zeta)}
\end{equation}
and
\begin{equation} \label{eq:Bochner-Kodaira_ineq2}
  \begin{aligned}
    \norm{S_q\zeta}_3^2 + \norm{T_{q-1}^*\zeta}_1^2 
    \geq 
    &- \int_{K_c} \paren{\Tr_g\oTheta_{v\conj v}}
    \abs{\zeta''}_{g,\eta,\chi}^2 \\
    &+ \int_{K_c} e^{-\chiX} \Tr_{g,\eta} \pr_F \paren{\oTheta^\dual(\zeta \otimes
    \conj \zeta)}
  \end{aligned}
\end{equation}
for all $\zeta \in \smFBCtwo \cap \Dom_{K_c,\chi} T_{q-1}^*$.
These are the \emph{Bochner--Kodaira inequalities for $T_{q-1}^*$ and
  $S_q$} which are used to obtain the required $L^2$ estimates.

\subsection{Murakami's trick}
\label{sec:curv-terms-murakami-trick}

From (\ref{eq:curvature-tensor}) and $\log\paren{\eta e^{-\chiX}} =
\log\etat + \log\etaw - \chiX$ (see \S \ref{sec:metric-on-L} for
notation), it follows that
\begin{equation} \label{eq:Theta}
  \begin{aligned}
    \Theta 
    &= \Theta_\tCurv + \Theta_\wCurv + \iddbar\chiX \\ 
    &= \pi \iddbar \Hform + 2\iddbar\Re
      \hbar_\delta + \iddbar\chiX
  \end{aligned}
\end{equation}
where $\Theta_\tCurv$ and $\Theta_\wCurv$ are respectively the tame and wild
curvature forms of $L$ defined in \S \ref{sec:metric-on-L},
and $\Hform$ is a hermitian form on $\fieldC^n \times \fieldC^n$ associated to $L$.
Therefore, by abusing $\Hform$ to mean the associated hermitian form
in $\smform{1,0} \otimes \smform{0,1}(X)$, the curvature integral $\int_{K_c} e^{-\chiX} \Tr_{g,\eta}
\pr_F \paren{\oTheta^\dual (\zeta \otimes \conj\zeta)}$ in
(\ref{eq:Bochner-Kodaira_ineq}) and (\ref{eq:Bochner-Kodaira_ineq2})
can be split into the sum of
\begin{equation}
  \begin{aligned}
    \tCurv(\zeta, \zeta) &:= \pi \int_{K_c} e^{-\chiX} \Tr_{g,\eta}
    \pr_F
    \paren{\Hform^\dual (\zeta \otimes \conj\zeta)} \; , \\
    \label{eq:wild-curv-integral}
    \wCurv(\zeta, \zeta) &:= \int_{K_c} e^{-\chiX} \Tr_{g,\eta} \pr_F
    \paren{\paren{2~\ddbar\Re \hbar_\delta}^\dual (\zeta \otimes
    \conj\zeta)} \; , \\
    \wtCurv(\zeta, \zeta) &:= \int_{K_c} e^{-\chiX} \Tr_{g,\eta} \pr_F
    \paren{\paren{\ddbar\chiX}^\dual (\zeta \otimes \conj\zeta)} 
  \end{aligned}
\end{equation}%
\showsym[tCurv]{$\tCurv(\zeta, \zeta)$, $\wCurv(\zeta,\zeta)$,
  $\wtCurv(\zeta, \zeta)$}{curvature integrals defined in
  (\ref{eq:wild-curv-integral})}%
(recall that $\smform{1,1}$ is identified with $\smform{1,0} \otimes
\smform{0,1}$ via $dz^k \wedge d\conj{z^\ell} \mapsto dz^k \otimes
d\conj{z^\ell}$ for all $1 \leq k, \ell \leq n$).

One of the essential ingredients for obtaining the required $L^2$
estimates for
$q < s_F^-$ or $q > m - s_F^+$ is Murakami's trick used in
\cite{Murakami}.
The trick is applied to the part of the curvature integral
$\tCurv(\zeta,\zeta)$ involving $\Hform_F := \Hform|_{F\times
  F}$\showsym[HF]{$\Hform_F$}{$:= \Hform"|"_{F\times F}$}.
For any $\zeta = \zeta' + \zeta'' \in \smFBCtwo$ where $\zeta' \in
\smfb 1{q-1}(\cl K_c;L)$ and $\zeta'' \in \smfb 0q (\cl K_c;L)$, that
part is given by
\begin{equation*}
  \begin{aligned}
    \tCurv_F(\zeta,\zeta) 
    &:= \pi \int_{K_c} e^{-\chiX} \Tr_{g,\eta} \pr_F
    \paren{\Hform_F^\dual(\zeta \otimes \conj\zeta)} \\
    &= \pi \int_{K_c} e^{-\chiX} \Tr_{g,\eta} \Hform_F^\dual(\zeta''
    \otimes \conj{\zeta''}) 
    = \tCurv_F(\zeta'',\zeta'') \; .
  \end{aligned}
\end{equation*}%
\showsym[tCurvF]{$\tCurv_F(\zeta,\zeta)$}{a summand of $\tCurv(\zeta,\zeta)$
  corresponding to $\Hform_F$}%


\begin{definition}
  An $\Hform$-apt coordinate system is an apt coordinate system such
  that the matrix of $\Hform_F$ under such coordinate system is given
  by 
 \begin{equation} \label{eq:diag-H_F}
   H_F = D := \Diag{\underbrace{1,\dots,1}_{s_F^+},
     \underbrace{-1,\dots,-1}_{s_F^-}, \underbrace{0,\dots,0}_{m -
       s_F^+ - s_F^-}} \; .
 \end{equation}
\end{definition}
Under a chosen apt coordinate system, an $\Hform$-apt coordinate
system can be obtained by a linear change of coordinates only in the
variable $v$ (which preserves the decomposition
(\ref{eq:cotangent-split})).

In what follows, write $d\conj{v^{J_q}}:= d\conj{v^{j_1}}\wedge \dots
\wedge d\conj{v^{j_q}}$\showsym[dv]{$d\conj{v^{J_q}}$}{$:=
  d\conj{v^{j_1}}\wedge \dots \wedge d\conj{v^{j_q}}$} for every
$q$-multiindex $J_q = (j_1,\dotsc,j_q)$.
Moreover, let $\zeta''_{\conj J_{q}}$ be the component of $\zeta'' \in
\smfb 0{q}(L)$ corresponding to $d\conj{v^{J_{q}}}$, and
$\zeta'_{\conj i \conj J_{q-1}}$ the component of $\zeta' \in \smfb
1{q-1}(L)$ corresponding to $d\conj{u^i} \wedge d\conj{v^{J_{q-1}}}$.

\begin{lemma}[Murakami's trick for $q > m - s_F^+$] \label{lem:murakami-trick}
  For any $q > m - s_F^+$ and given any constant $M > 0$, one can
  choose the translational invariant hermitian metric $g$ suitably such that $\tCurv_F(\zeta'', \zeta'')
  \geq \pi  M \norm{\zeta''}_2^2$ for every $\zeta'' \in \smfb 0q(\cl K_c;L)$.
\end{lemma}

\begin{proof}
  Fix an $\Hform$-apt coordinate system.
  Choose $g$ such that it is diagonal in the chosen $\Hform$-apt
  coordinates and its matrix is given by
  \begin{equation} \label{eq:choice-of-g}
    \Diag{\underbrace{1,\dots,1}_{n-m}, \underbrace{\frac{1}{g^F_+},\dots,\frac{1}{g^F_+}}_{s_F^+},
     \underbrace{\frac{1}{g^F_-},\dots,\frac{1}{g^F_-}}_{s_F^-},
     \underbrace{\frac{1}{g^F_0},\dots,\frac{1}{g^F_0}}_{m - s_F^+ -
       s_F^-}} \; ,
  \end{equation}
  where $g^F_+$, $g^F_-$ and $g^F_0$ are positive numbers.
  Given $M > 0$, $g^F_+$, $g^F_-$ and $g^F_0$ are chosen as
  \begin{equation*}
    g^F_+ := s_F^- + M \; , \quad g^F_- := 1  \quad\text{ and }\quad
    g^F_0 := 1 \; .
  \end{equation*}

  Under the chosen $\Hform$-apt coordinates, since $\Hform_F$ and $g$
  are both diagonal, the monomial forms $\zeta''_{\conj J_q}
  d\conj{v^{J_q}} \in \smfb 0q(\cl K_c;L)$ with different multiindices
  $J_q$ are orthogonal to one another with respect to $\tCurv_F$ and
  $\inner\cdot\cdot_{2}$.
  Therefore, it suffices to show that
  \begin{equation} \label{eq:tCurv_F-bbd-monomial} \tag{$*$}
    \tCurv_F \paren{\zeta''_{\conj J_q} d\conj{v^{J_q}}, \zeta''_{\conj J_q}
    d\conj{v^{J_q}}} \geq \pi  M \norm{\zeta''_{\conj J_q}
      d\conj{v^{J_q}}}_2^2
  \end{equation}
  for all monomial forms $\zeta''_{\conj J_q} d\conj{v^{J_q}} \in \smfb 0q
  (\cl K_c;L)$.
  
  In fact, for each multiindex $J_q = (j_1,\dots,j_q)$, one has
  \begin{equation*}
    \begin{aligned}
      \tCurv_F\paren{\zeta''_{\conj J_q} d\conj{v^{J_q}},
        \zeta''_{\conj J_q} d\conj{v^{J_q}}} 
      &= \pi \int_{K_c} \paren{\sum_{\nu = 1}^q (g_F)^{\conj j_\nu j_\nu}
        (H_F)_{j_\nu \conj j_\nu}} \abs{\zeta''_{\conj J_q}
        d\conj{v^{J_q}}}_{g,\eta,\chi}^2 \\
      &= \pi \paren{\sum_{\nu = 1}^q (g_F)^{\conj j_\nu j_\nu}
        (H_F)_{j_\nu \conj j_\nu}} \norm{\zeta''_{\conj J_q}
        d\conj{v^{J_q}}}_2^2  \; ,
    \end{aligned}
  \end{equation*}
  where $(g_F)^{\conj j_\nu j_\nu}$'s are the diagonal components of
  $(g_F)^{-1} := (\pr_F g)^{-1}$, and
  $(H_F)_{j_\nu \conj j_\nu}$'s are the diagonal entries of $H_F$ in
  (\ref{eq:diag-H_F}), which are either $1$, $-1$ or $0$.
  Define
  \begin{gather*}
    R^+(J_q) := \#\set{j_\nu\in J_q\: : \: 1 \leq j_\nu \leq s_F^+} \\
    R^-(J_q) := \#\set{j_\nu\in J_q\: : \: s_F^+ + 1 \leq j_\nu \leq
      s_F^+ + s_F^-} \; .
  \end{gather*}
  Then, the sum in the parenthesis becomes
  \begin{equation} \label{eq:tame-Curv-01}
    \begin{aligned}
      \sum_{\nu = 1}^q (g_F)^{\conj j_\nu j_\nu} (H_F)_{j_\nu \conj j_\nu} 
      &= g^F_+ R^+(J_q) - g^F_- R^-(J_q) \\
      &= (s_F^- + M) R^+(J_q) - R^-(J_q) \; .
    \end{aligned}
  \end{equation}
  Since $q > m - s_F^+$, it follows that $R^+(J_q) \geq 1$ for any
  multiindex $J_q$. 
  Note also that $R^-(J_q) \leq s_F^-$ for any $J_q$. 
  Therefore, by the choice of $g^F_+$ and $g^F_-$,
  one obtains $g^F_+ R^+(J_q) - g^F_- R^-(J_q) \geq M$ and thus
  (\ref{eq:tCurv_F-bbd-monomial}) follows.
  This completes the proof.
\end{proof}


In order to apply Lemma \ref{lem:murakami-trick} to
(\ref{eq:Bochner-Kodaira_ineq}), 
note
that for any $\zeta'' \in \smbform{0,q}(\cl K_c;L)$, one has
\begin{equation*}
  \norm{\zeta''}_2^2  
  = 
  \int_{K_c} e^{-\chiX} \Tr_{g,\eta}
  \zeta''\otimes \conj{\zeta''} \; .
\end{equation*}
Decompose $\Hform \in \smform{1,0} \otimes \smform{0,1}(X)$ into $\Hform_E +
\Hform_{u\conj v} + \Hform_{v\conj u} + \Hform_F$\showsym[HFE]{$\Hform_E +
\Hform_{u\conj v} + \Hform_{v\conj u} + \Hform_F$}{decomposition
of $\Hform$ with respect to \\ the decomposition (\ref{eq:cotangent-split})} according
to the decomposition (\ref{eq:cotangent-split}) as is done to $\oTheta$
(write $\Hform_E$ for $\Hform_{u\conj u}$ and $\Hform_F$ for
$\Hform_{v\conj v}$ to respect previous notations).
Now note that, for any $\zeta = \zeta' + \zeta'' \in \smFBCtwo$ such
that $\zeta' \in \smfb 1{q-1}(\cl K_c;L)$ and $\zeta'' \in \smfb 0q
(\cl K_c;L)$, one has
\begin{equation*}
  \begin{aligned}
    \tCurv(\zeta,\zeta) = \pi \int_{K_c} e^{-\chiX} \Tr_{g,\eta}
    &\left(\Hform_E^\dual (\zeta' \otimes \conj{\zeta'}) +
      \Hform_{u\conj v}^\dual(\zeta' \otimes \conj{\zeta''}) \right. \\
    &\left.~+~\Hform_{v\conj u}^\dual(\zeta'' \otimes \conj{\zeta'})
      +\Hform_F^\dual (\zeta'' \otimes \conj{\zeta''}) \right) \; .
    \end{aligned}
\end{equation*}
If $q > m - s_F^+$, Lemma \ref{lem:murakami-trick} then implies that,
given $M > 0$, $g$ can be chosen such that
\begin{equation*}
  \tCurv(\zeta,\zeta) \geq
  \begin{aligned}[t]
    \pi \int_{K_c} e^{-\chiX} &\left[\Tr_{g,\eta} \paren{\Hform_E^\dual
        (\zeta' \otimes \conj{\zeta'}) +
        \Hform_{u\conj v}^\dual(\zeta' \otimes \conj{\zeta''})
        + \Hform_{v\conj u}^\dual(\zeta'' \otimes
        \conj{\zeta'})}\right. \\ 
    + &~M \left. \Tr_{g,\eta} \zeta'' \otimes \conj{\zeta''} \right]
    \; .
  \end{aligned}
\end{equation*}
Define $\altH(M)$ to be an element in $\smform{0,0} \oplus
\paren{\smform{1,0} \otimes \smform{0,1}}(X)$ 
such that 
\begin{equation} \label{eq:def-altH}
  \begin{aligned}
    \pr_F \paren{\paren{\altH(M)}^\dual (\zeta \otimes \conj\zeta)} =
      &~\Hform_E^\dual
        (\zeta' \otimes \conj{\zeta'}) +
        \Hform_{u\conj v}^\dual(\zeta' \otimes \conj{\zeta''}) \\
      &~+ \Hform_{v\conj u}^\dual(\zeta'' \otimes \conj{\zeta'})
        + M~ \zeta'' \otimes \conj{\zeta''}
    \end{aligned}
\end{equation}
for any $\zeta = \zeta' + \zeta'' \in \smFBCtwo$ (note that $\zeta' =
0$ when $q = 0$).
Then one has
\begin{equation*}
  \tCurv(\zeta,\zeta)
  \geq
  \pi \int_{K_c} e^{-\chiX} \Tr_{g,\eta} \pr_F \paren{\paren{\altH(M)}^\dual
  (\zeta \otimes \conj\zeta)}
\end{equation*}
when $q > m - s_F^+$. 
Therefore, the consequence of Lemma \ref{lem:murakami-trick} applied to
(\ref{eq:Bochner-Kodaira_ineq}) can be stated as follows.
\begin{cor} \label{cor:BK-ineq-with-Murakami}
  Suppose $q > m - s_F^+$. Then, given any constant $M > 0$, the
  translational invariant hermitian metric $g$ can be chosen suitably such that
  (\ref{eq:Bochner-Kodaira_ineq}) yields
  \begin{equation*}
    \begin{aligned}
      \norm{S_q\zeta}_3^2 + \norm{T_{q-1}^*\zeta}_1^2 \geq 
      &~\pi \int_{K_c} e^{-\chiX} \Tr_{g,\eta} \pr_F \paren{\paren{\altH(M)}^\dual
      (\zeta \otimes \conj\zeta)}
      \\
      &+ \wCurv(\zeta,\zeta) + \wtCurv(\zeta,\zeta)
    \end{aligned}
  \end{equation*}
  for all $\zeta \in \smFBCtwo \cap \Dom T_{q-1}^*$.
\end{cor}

Now consider the integral involving $\Tr_g \oTheta_{v\conj v}$ in
(\ref{eq:Bochner-Kodaira_ineq2}). 
Note that
\begin{equation*}
  \pr_F \Theta = \pi \iddbarb \Hform + 2
  \iddbarb\paren{\Re \hbar_\delta} \; .
\end{equation*}
Here no term involving $\chiX$ appears since 
$\iddbarb\chiX = 0$. 
 Again, by abusing $\Hform_F$ to mean the associated hermitian form in
 $\smform{1,0} \otimes \smform{0,1}(X)$, 
the curvature integral $- \int_{K_c}
\paren{\Tr_g\oTheta_{v\conj v}} \abs{\zeta''}_{g,\eta,\chi}^2$ in
(\ref{eq:Bochner-Kodaira_ineq2}) can be split into the sum of
\begin{equation}
  \begin{aligned}
    \tCurv_F'(\zeta'',\zeta'') &:= - \pi \int_{K_c}
    \paren{\Tr_g\Hform_F}
    \abs{\zeta''}_{g,\eta,\chi}^2 \; , \\
    \label{eq:wild-trace-integral}
    \wCurv'_F(\zeta'',\zeta'') &:= - \int_{K_c}
    \paren{2 \Tr_g \ddbarb\Re \hbar_\delta}
    \abs{\zeta''}_{g,\eta,\chi}^2 \; .
  \end{aligned}
\end{equation}%
\showsym[tCurvF]{$\tCurv_F'(\zeta'',\zeta'')$,
  $\wCurv'_F(\zeta'',\zeta'')$}{curvature integrals defined in
  (\ref{eq:wild-trace-integral})}%
Similar argument as in the proof of Lemma \ref{lem:murakami-trick} yields
\begin{lemma}[Murakami's trick for $q < s_F^-$] \label{lem:Murakami-trick2}
  Suppose that $q < s_F^-$. Then for any given constant $M > 0$, one
  can choose the translational invariant hermitian metric $g$ suitably such that
  \begin{equation*}
    \tCurv_F'(\zeta'', \zeta'')
    + \tCurv_F(\zeta'',\zeta'') \geq \pi  M
    \norm{\zeta''}_2^2 
  \end{equation*}
  for all $\zeta'' \in \smbform{0,q}(\cl K_c;L)$. 
\end{lemma}

\begin{proof}
  Fix an $\Hform$-apt coordinate system.
  Choose $g$ as in (\ref{eq:choice-of-g}).
  Given $M > 0$, $g^F_+$, $g^F_-$ and $g^F_0$ are chosen as
  \begin{equation*}
    g^F_+ := 1 \; , \quad g^F_- := s_F^+ + M  \quad\text{ and }\quad
    g^F_0 := 1 \; .
  \end{equation*}

  Using the $\Hform$-apt coordinates, one sees that
  $\Tr_g \Hform_F = g^F_+ s_F^+ - g^F_- s_F^-$ and therefore
  \begin{equation*}
    \tCurv_F'(\zeta'',\zeta'') = \pi \paren{g^F_- s_F^- - g^F_+ s_F^+}
    \norm{\zeta''}_2^2 \; .
  \end{equation*}

  Again, since $\Hform_F$ and $g$ are both diagonal under the chosen
  $\Hform$-apt coordinates, the monomial forms $\zeta''_{\conj J_q}
  d\conj{v^{J_q}} \in \smfb 0q(\cl K_c;L)$ with different multiindices
  $J_q$ are orthogonal to one another with respect to $\tCurv_F$ and
  $\inner\cdot\cdot_{2}$.
  Therefore, it suffices to show that
  \begin{equation} \label{eq:tr-tCurv_F-bbd-monomial} \tag{$**$}
    \begin{aligned}
      \pi \paren{g^F_- s_F^- - g^F_+ s_F^+} \norm{\zeta''_{\conj J_q}
        d\conj{v^{J_q}}}_2^2 + \tCurv_F \paren{\zeta''_{\conj J_q}
        d\conj{v^{J_q}}, \zeta''_{\conj J_q} d\conj{v^{J_q}}} 
      \geq \pi M \norm{\zeta''_{\conj J_q} d\conj{v^{J_q}}}_2^2
    \end{aligned}  
  \end{equation}
  for all monomial forms $\zeta''_{\conj J_q} d\conj{v^{J_q}} \in \smfb 0q
  (\cl K_c;L)$.

  Taking into account (\ref{eq:tame-Curv-01}) and the expression of
  $\tCurv_F$ in the proof of Lemma \ref{lem:murakami-trick}, it follows that
  \begin{equation*}
    \begin{aligned}
      &\hphantom{=} \pi  \paren{g^F_- s_F^- - g^F_+ s_F^+} \norm{\zeta''_{\conj
          J_q} d\conj{v^{J_q}}}_2^2 + \tCurv_F \paren{\zeta''_{\conj J_q}
      d\conj{v^{J_q}}, \zeta''_{\conj J_q} d\conj{v^{J_q}}} \\
      &= \pi \paren{g^F_- \paren{s_F^- - R^-(J_q)}-
        g^F_+ \paren{s_F^+ - R^+(J_q)}} \norm{\zeta''_{\conj J_q}
        d\conj{v^{J_q}}}_{2}^2 \\
      &= \pi \paren{\paren{s_F^+ + M} \paren{s_F^- - R^-(J_q)}-
        \paren{s_F^+ - R^+(J_q)}} \norm{\zeta''_{\conj J_q}
        d\conj{v^{J_q}}}_{2}^2 \; .
    \end{aligned}
  \end{equation*}
  Since $q < s_F^-$, it follows that $s_F^- - R^-(J_q) \geq 1$ for any
  multiindex $J_q$. 
  Note also that $s_F^+ - R^+(J_q) \leq s_F^+$ for any $J_q$. 
  Therefore, by the choice of $g^F_+$ and $g^F_-$, one obtains $g^F_-
  \paren{s_F^- - R^-(J_q)}- g^F_+ \paren{s_F^+ - R^+(J_q)} \geq M$ and
  thus (\ref{eq:tr-tCurv_F-bbd-monomial}) follows.
  This completes the proof.
\end{proof}

Considering the definition of $\altH(M)$ in (\ref{eq:def-altH}), Lemma
\ref{lem:Murakami-trick2} then implies that, if $q < s_F^-$, then,
given $M > 0$, $g$ can be chosen such that
\begin{equation*}
  \tCurv_F'(\zeta'',\zeta'') + \tCurv(\zeta, \zeta) \geq 
  \pi \int_{K_c} e^{-\chiX} \Tr_{g,\eta} \pr_F \paren{\paren{\altH(M)}^\dual
      (\zeta \otimes \conj\zeta)}
\end{equation*}
for all $\zeta = \zeta' + \zeta'' \in \smFBCtwo$.
Combining this with (\ref{eq:Bochner-Kodaira_ineq2}) yields
\begin{cor} \label{cor:BK-ineq2-with-Murakami}
  Suppose $q < s_F^-$. Then, given any constant $M > 0$, the
  translational invariant hermitian metric $g$ can be chosen suitably such that
  (\ref{eq:Bochner-Kodaira_ineq2}) yields
  \begin{equation*}
    \begin{aligned}
      \norm{S_q\zeta}_3^2 + \norm{T_{q-1}^*\zeta}_1^2 \geq 
      &~\pi \int_{K_c} e^{-\chiX} \Tr_{g,\eta} \pr_F \paren{\paren{\altH(M)}^\dual
      (\zeta \otimes \conj\zeta)}
      \\
      &+ \wCurv_F'(\zeta'',\zeta'')
      + \wCurv(\zeta,\zeta)
      + \wtCurv(\zeta,\zeta)
    \end{aligned}
  \end{equation*}
  for all $\zeta = \zeta' + \zeta'' \in \smFBCtwo \cap \Dom T_{q-1}^*$,
  where $\zeta'' \in \smfb 0q(\cl K_c;L)$ and $\zeta' \in \smfb 1{q-1}(\cl K_c;L) \cap
  \Dom \dbarf^*$.
\end{cor}


The remaining part of this section is devoted to getting a suitable
estimate of the integral
\begin{equation*}
  \pi \int_{K_c} e^{-\chiX} \Tr_{g,\eta} \pr_F \paren{\paren{\altH(M)}^\dual
  (\zeta \otimes \conj\zeta)}
\end{equation*}
by varying $\Hform_E$ in $\altH(M)$ (see (\ref{eq:def-altH})) according to
Proposition \ref{prop:H_bot_R-arbitrary}.

\begin{lemma} \label{lem:bound-tame-part}
  Given a constant $M > 0$ and a fixed translational invariant hermitian metric $g$ on $X$ such that
  the decomposition (\ref{eq:cotangent-split}) is orthogonal, one can choose
  $\Hform_E$ sufficiently positive according to Proposition
  \ref{prop:H_bot_R-arbitrary} such that
  \begin{equation*}
    \pi \int_{K_c} e^{-\chiX} \Tr_{g,\eta} \pr_F \paren{\paren{\altH(M)}^\dual
    (\zeta \otimes \conj\zeta)}
    \geq \tfrac{\pi }{4} M \norm\zeta_2^2 
  \end{equation*}
  for all $\zeta \in \smFBCtwo$.
\end{lemma}

\begin{proof}
  For $q = 0$, it follows from (\ref{eq:def-altH}) that
  \begin{equation*}
    \pi \int_{K_c} e^{-\chiX} \Tr_{g,\eta} \pr_F \paren{\paren{\altH(M)}^\dual
    (\zeta \otimes \conj\zeta)} = \pi M \norm\zeta_2^2 \geq
    \tfrac{\pi}{4} M \norm\zeta_2^2 \; ,
  \end{equation*}
  so this case is done.

  Assume $q \neq 0$.
  Since $\Hform_{u\conj v}^\dual$ is a bounded linear operator on
  $\hilbFB 10_{c,\chi} \otimes \conj{\hilbFB 01_{c,\chi}}$ (where
  $\conj{\hilbFB 01_{c,\chi}}$ here means the complex conjugate of
  $\hilbFB 01_{c,\chi}$), it follows that
  there is a bounded linear operator $\Nform \colon \hilbFB
  0q_{c,\chi} \to \hilbFB 1{q-1}_{c,\chi}$ such that
  \begin{equation*}
    \int_{K_c} e^{-\chiX} \Tr_{g,\eta} \Hform_{u\conj v}^\dual (\zeta'
    \otimes \conj{\zeta''}) = \inner{\zeta'}{\Nform \zeta''}_2 
  \end{equation*}
  for all $\zeta' \in \hilbFB 1{q-1}_{c,\chi}$ and $\zeta'' \in
  \hilbFB 0q_{c,\chi}$. 
  In fact, after a linear change of coordinates such
  that $g$ becomes the Euclidean metric while keeping the
  decomposition (\ref{eq:cotangent-split}) orthogonal, one has
  \begin{equation*}
    \Tr_{g,\eta} \Hform_{u\conj v}^\dual (\zeta' \otimes
    \conj{\zeta''}) = \eta \stsum_{J_{q-1}} \sum_{i=1}^{n-m} \sum_{j=1}^m
    \zeta'_{\conj i \conj J_{q-1}} \conj{(\Hform_{v\conj u})_{j \conj i}~\zeta''_{\conj
        j \conj J_{q-1}}} \; ,
  \end{equation*}
  where $\sum'_{J_{q-1}}$ denotes summation over all ordered
  multiindices $J_{q-1}$ such that $1\leq j_1 < \dots < j_{q-1} \leq
  m$, and $(\Hform_{v\conj u})_{j\conj i}$'s are the components of $\Hform_{v\conj u} =
  \conj{\Hform_{u\conj v}}$.
  Therefore, under such coordinates,
  \begin{equation*}
    \paren{\Nform\zeta''}_{\conj i \conj J_{q-1}} 
    = \sum_{j=1}^m (\Hform_{v\conj u})_{j \conj i}~\zeta''_{\conj j \conj J_{q-1}} \; .
  \end{equation*}
  Moreover,
  \begin{equation*}
    \begin{aligned}
      \abs{\Nform\zeta''}_{g,\eta}^2 
      &= \eta \stsum_{J_{q-1}} \sum_{i=1}^{n-m} \abs{\sum_{j=1}^m
        (\Hform_{v\conj u})_{j \conj i}~\zeta''_{\conj j \conj J_{q-1}}}^2 \\
      &\leq \eta \stsum_{J_{q-1}} \sum_{i=1}^{n-m} \paren{\sum_{j=1}^m
      \abs{(\Hform_{v\conj u})_{j \conj i}}^2} \paren{\sum_{j=1}^m \abs{\zeta''_{\conj j
          \conj J_{q-1}}}^2}  && \text{\parbox{2.6cm}{by
        Cauchy-- \\ Schwarz ineq.,}} 
    \\
      &= \abs{\Hform_{v\conj u}}_g^2 \cdot q \abs{\zeta''}_{g,\eta}^2
      = \abs{\Hform_{u\conj v}}_g^2 \cdot q \abs{\zeta''}_{g,\eta}^2
      && \text{as }\Hform_{u\conj v} = \conj{\Hform_{v\conj u}} \; .
    \end{aligned}
  \end{equation*}
  Since both $\Hform_{u\conj v}$ and $g$ are translational invariant forms,
  $\abs{\Hform_{u\conj v}}_g^2$ is a constant.
  Set $\nu := \sqrt q \abs{\Hform_{u\conj v}}_g$. 
  Then, one has
  \begin{equation} \label{eq:nu-ineq} \tag{$*_\nu$}
    \norm{\Nform\zeta''}_2 \leq \nu \norm{\zeta''}_2
  \end{equation}
  for all $\zeta'' \in \hilbFB 0q_{c,\chi}$.
  Note that $\nu$ depends only on $q$, $\Hform_{u\conj v}$ and $g$.
  It is independent of $\Hform_E$ in particular.
  
  Since the decomposition (\ref{eq:cotangent-split}) is orthogonal
  with respect to $g$, $g$ can be decomposed into $g_E + g_F$ such
  that $g_E$ is a hermitian metric on $\Tgtf{1,0}$ and $g_F$ is that
  on $\Tgtb{1,0}$.
  Choose a real number $\lambda > 0$ such that
  \begin{equation} \label{eq:lambda-ineq} \tag{$*_\lambda$}
    \lambda \geq \max\set{\frac{M}{2}, \frac{2\nu^2}{M}, 4\nu} \; .
  \end{equation}
  Since $\nu$ is independent of $\Hform_E$, by varying the real part
  of the matrix of $\Hform_E$ under the chosen apt coordinates
  according to Proposition \ref{prop:H_bot_R-arbitrary}, $\Hform_E$
  can be chosen such that
  \begin{equation*}
    \Hform_E \geq \lambda g_E \; ,
  \end{equation*}
  and therefore,
  \begin{equation*}
    \int_{K_c} e^{-\chiX} \Tr_{g,\eta} \Hform_E^\dual (\zeta'
      \otimes \conj{\zeta'}) \geq \lambda \norm{\zeta'}_2^2 
  \end{equation*}
  for all $\zeta' \in \smfb 1{q-1}(\cl K_c;L)$. 

  It follows from (\ref{eq:def-altH}) that, for any $\zeta = \zeta' +
  \zeta'' \in \smFBCtwo$,
  \begin{align*}
      &~ \int_{K_c} e^{-\chiX} \Tr_{g,\eta} \pr_F \paren{\altH(M)}^\dual (\zeta
      \otimes \conj \zeta) && \\
      \geq &~
      \lambda \norm{\zeta'}_2^2 + 2\Re \inner{\zeta'}{\Nform\zeta''}_2
      + M \norm{\zeta''}_2^2 && \\
      = &~
      \lambda \norm{\zeta' + \frac{1}{\lambda} \Nform \zeta''}_2^2  
      - \frac{1}{\lambda} \norm{\Nform\zeta''}_2^2
      + M \norm{\zeta''}_2^2 && \text{by completing square} \; , \\
      \geq &~
      \lambda \norm{\zeta' + \frac{1}{\lambda} \Nform \zeta''}_2^2  
      - \frac{\nu^2}{\lambda} \norm{\zeta''}_2^2
      + M \norm{\zeta''}_2^2 && \text{by (\ref{eq:nu-ineq})} \; , \\
      \geq &~
      \frac{M}{2} \paren{\norm{\zeta' + \frac{1}{\lambda} \Nform \zeta''}_2^2
        + \norm{\zeta''}_2^2} && \text{by (\ref{eq:lambda-ineq}), thus
        }\frac{\nu^2}{\lambda} \leq \frac{M}{2} \; , \\
      = &~
      \frac{M}{2} \norm{\zeta + \frac{1}{\lambda} \Nform \zeta''}_2^2
      && \text{as }\hilbFB 1{q-1}_{c,\chi} \perp \hilbFB 0q_{c,\chi}
      \; .
  \end{align*}
  Furthermore, since
  \begin{align*}
      \norm{\zeta + \frac{1}{\lambda} \Nform \zeta''}_2 &\geq
      \norm\zeta_2 - \frac{1}{\lambda} \norm{\Nform\zeta''}_2 \\ 
      &\geq
      \norm\zeta_2 - \frac{\nu}{\lambda} \norm{\zeta''}_2 && \text{by
        (\ref{eq:nu-ineq})} \; , \\
      &\geq
      \paren{1 - \frac{\nu}{\lambda}} \norm\zeta_2 && \text{as
      }\norm{\zeta''}_2 \leq \norm\zeta_2 \; , \\
      &\geq \frac{3}{4} \norm\zeta_2 \geq 0 && \text{by
        (\ref{eq:lambda-ineq})}\; ,
  \end{align*}
  one has
  \begin{equation*}
    \frac{M}{2} \norm{\zeta + \frac{1}{\lambda} \Nform \zeta''}_2^2
    \geq \frac{M}{2} \cdot \paren{\frac{3}{4}}^2 \norm\zeta_2^2 \geq
    \frac{M}{4} \norm\zeta_2^2 \; .
  \end{equation*}
  This completes the proof.
\end{proof}

} 






\section{The linearizable case}
\label{sec:proof-linearizable}

%
%

\subsection{Proof of Theorem \ref{thm:main_thm} for linearizable $L$}

The proof of Theorem \ref{thm:main_thm} for linearizable $L$ is given
here so that one can see clearly how the proof works without having to
handle additional technicality required for the case of
non-linearizable line bundles.

\begin{thm} \label{thm:proof-linearizable}
  Suppose $L$ is linearizable and $q < s_F^-$ or $q > m -
  s_F^+$. Then, for any $\psi \in \kashf{0,q}(X;L)$ such that
  $\dbar\psi = 0$, there exists $\xi \in \kashf{0,q-1}(X;L)$ such that
  $\dbar\xi = \psi$ on $X$. (In case $q = 0 < s_F^-$, this means $\psi
  = 0$.) In other words, by virtue of Theorem \ref{rem:Kazama-Dolbeault},
  $H^q(X,L) = 0$ for any $q$ in the given range.
\end{thm}

\begin{proof}
  Fix any $\psi \in \kashf{0,q}(X;L) \cap \ker \dbar$. 

  An $L^2$-norm $\norm\cdot_{X,\chi}$ is chosen as follows.
  Since $L$ is linearizable, one can take $\hbar = 0$ (see \S
  \ref{sec:metric-on-L} for the definition of $\hbar$).
  Then, choose $\delta = 0$ and thus $\hbar_\delta = \hbar - \delta =
  0$.
  Choose the translational invariant hermitian metric $g$ of the form as described in
  the proof of Lemma \ref{lem:murakami-trick} for $q > m - s_F^+$ or Lemma
  \ref{lem:Murakami-trick2} for $q < s_F^-$, with $M = 1$.
  For the hermitian form $\Hform$ associated to $L$, choose $\Hform_E
  := \Hform|_{E\times E}$ as described in the proof of Lemma
  \ref{lem:bound-tame-part}.
  A hermitian metric $\eta$ on $L$ is then defined as in \S \ref{sec:metric-on-L}.
  Choose a convex increasing smooth function $\chiR$ (thus
  $\chiX := \chiR \circ \varphi$ is plurisubharmonic,
  i.e.~$\iddbar\chiX \geq 0$) such that $\norm\psi_{X,\chi} < \infty$.
  An $L^2$-norm $\norm\cdot_{X,\chi}$ is then fixed and $\psi \in
  \hilbbsp{0,q}_{\chi}(X;L)$.

  Note that every $\zeta \in \smFBCXtwo$ is contained in
  $\smFBtwo[0\,]$ for some sufficiently large but finite $c > 0$. 
  Consequently, the conclusion of Corollary
  \ref{cor:BK-ineq-with-Murakami} when $q > m - s_F^+$ or Corollary
  \ref{cor:BK-ineq2-with-Murakami} when $q < s_F^-$, as well as that
  of Lemma \ref{lem:bound-tame-part}, holds for all $\zeta = \zeta'
  + \zeta'' \in \smFBCXtwo$, where $\zeta' \in \smfb 1{q-1}_0(X;L)$
  and $\zeta'' \in \smfb 0q_0(X;L)$. 
  Since $\hbar_\delta = 0$, $\wCurv(\zeta,\zeta)$ (see
  (\ref{eq:wild-curv-integral})) and $\wCurv'_F(\zeta'',\zeta'')$ (see
  (\ref{eq:wild-trace-integral})) both vanish for all $\zeta = \zeta'
  + \zeta'' \in \smFBCXtwo$.

  Since $\chiX$ is plurisubharmonic on $X$ and $\dbarb\chiX = 0 =
  \diffb\chiX$, one can choose at every point $z\in X$ the coordinates
  such that both $g$ and $\cplxi \difff\dbarf \chiX$ are simultaneously
  diagonalized while keeping the decomposition
  (\ref{eq:cotangent-split}) orthogonal, and see that
  \begin{equation*}
    \Tr_{g,\eta} \pr_F \paren{\paren{\ddbar\chiX}^\dual (\zeta \otimes \conj\zeta)}
    = \Tr_{g,\eta} \paren{\difff\dbarf\chiX}^\dual (\zeta' \otimes \conj{\zeta'})
    \geq 0 \; .
  \end{equation*}
  Therefore, $\wtCurv(\zeta, \zeta) \geq 0$ (see
  (\ref{eq:wild-curv-integral})).

  As a result, combining Lemma \ref{lem:bound-tame-part} as well as
  the above facts about $\wCurv$, $\wCurv'_F$ and $\wtCurv$ with
  Corollary \ref{cor:BK-ineq-with-Murakami} or Corollary
  \ref{cor:BK-ineq2-with-Murakami}, one obtains
  \begin{equation*}
    \norm{S_q\zeta}_3^2 + \norm{T_{q-1}^*\zeta}_1^2
    \geq \tfrac{\pi}{4} \norm\zeta_2^2  
  \end{equation*}
  for all $\zeta \in \smFBCXtwo$. 
  This is the required $L^2$ estimate. 
  Proposition \ref{prop:estimate-imply-strong-sol} and Remark
  \ref{rem:regularity} then assert that
  there exists $\xi \in \kashf{0,q-1}(X;L)$ such that $\dbar\xi =
  \psi$ on $X$.
\end{proof}


\section{The non-linearizable case}
\label{chap:non-linearizable-case}

%
%

{
\newcommand{\Deltab}{\bm{\updelta}}
\newcommand{\wXi}{\widetilde\Xi}

\newcommand{\addD}[1]{#1}
\newcommand{\delD}[1]{}

For a non-linearizable line bundle $L$, the wild curvature terms
$\wCurv$ (see (\ref{eq:wild-curv-integral})) and $\wCurv'_F$ (see
(\ref{eq:wild-trace-integral})) are not identically zero.
In order to get the estimates for these terms, Takayama's Weak $\diff\dbar$-Lemma
(ref.~\cite{Takayama}*{Lemma 3.14}) is invoked. 
One is then forced to restrict attention to each of the $K_c$'s and
obtain the required $L^2$ estimates there.
What then remains is to show that the existence of a solution of the
$\dbar$-equation $\dbar \xi = \psi$ on every $K_c$ implies the
existence of a global solution.
The argument for this latter part is essentially the same as the one in
\cite{Grauert&Remmert}*{Ch.~IV, \S 1, Thm.~7}.

An apt coordinate system is fixed throughout this section.

\subsection{Bounds on the wild curvature terms}
\label{sec:bound-wCurv}

Takayama proves in \cite{Takayama} the following Weak $\ddbar$-Lemma.
{
  \theoremstyle{plain}
  \newtheorem{ddblem}[prop]{Weak $\diff\dbar$-Lemma}

  \begin{ddblem}[cf.~\cite{Takayama}*{Lemma 3.14}] \label{lem:weak-ddbarlemma}
    Let 
    $\omega$ be a positive real $(1,1)$-form on $X$, and let $\theta$ be a smooth
    real $1$-form on $X$ such that $\theta = \conj\beta + \beta$ for
    some smooth $(0,1)$-form $\beta$, and $d\theta$ is of
    type $(1,1)$. 
    Then for every positive number $\varepsilon$ and every
    relatively compact open subset $W$ of $X$, there exists a
    smooth function $\delta$ on $X$ such that
    \begin{equation*}
      -\varepsilon\omega < d\theta - 2\cplxi\diff\dbar\Re\delta <
      \varepsilon\omega \quad\text{on }W \; .
    \end{equation*}
    Moreover, if $\beta \in \kashf{0,1}(X)$, then
    $\delta$ can be chosen such that $\delta \in \kashf{}(X)$.
  \end{ddblem}
}

In the current situation, the role of $\beta$ in Lemma
\ref{lem:weak-ddbarlemma} is taken by $\cplxi~\dbar\hbar$ (therefore
$d \theta = 2\iddbar\Re\hbar$), and that of $W$ by $K_c$.

\begin{remark}
  In Takayama's formulation, the assertion of the Weak $\ddbar$-Lemma
  is that there exists a smooth real valued function $f_{\eps W} :=
  2(\Im f_0 + \Im \Psi_{M_0})$ on $X$ such that $-\varepsilon\omega < d\theta
  - \cplxi\ddbar f_{\eps W} < \varepsilon\omega \quad\text{on }W$,
  in which $f_0$ is a smooth function on $X$ such that $\beta = \phi +
  \dbar f_0$ for some real analytic $(0,1)$-form $\phi$ in $\kashf{0,1}(X)$,
  and $\Psi_{M_0}$ is some real analytic function in $\kashf{}(X)$.
  Therefore, the smooth function $\delta$ here is given by $\delta :=
  -\cplxi (f_0 + \Psi_{M_0})$ in Takayama's notation.
  If $\beta \in \kashf{0,1}(X)$, then one has $f_0 \in \kashf{}(X)$ as
  $\dbarf f_0 = 0$, so $\delta \in \kashf{}(X)$ also.
\end{remark}

\begin{remark} 
  As a side remark, following the construction of $\delta$ in
  \cite{Takayama}*{Lemma 3.14}, $\dbar \hbar_\delta = \dbar\hbar -
  \dbar\delta$ is real analytic on $X$, so $\hbar_\delta$ is real
  analytic on $\fieldC^n$. 
  It follows that the hermitian metric $\eta$ on $L$ is real analytic.
\end{remark}

Suitable estimates for
the wild curvature terms $\wCurv$ 
and $\wCurv'_F$ are obtained
by choosing a proper $\delta \in \kashf{}(X)$ according to the Weak
$\ddbar$-Lemma.


\begin{lemma} \label{lem:bounds-wild-curv}
  Suppose a hermitian metric $g$ on $X$ and a choice of $\Hform_E$ are
  fixed.
  Then, on every $K_c$ where $0 < c < \infty$, given any real number
  $\eps_\wildsub > 0$ and for any $q \geq 0$, one can choose $\delta_c
  \in \kashf{}(X)$ which yields a hermitian metric $\eta_c$ on $L$
  such that,
  for any given weight $\chiX$,
  \begin{gather} 
    \label{eq:bound-wCurv}
    \abs{\wCurv(\zeta,\zeta)} \leq
    \varepsilon_\wildsub q  \norm\zeta_{K_c,\eta_c,\chi}^2 \\
    \label{eq:bound-tr-wCurv}
    \abs{\wCurv'_F(\zeta'',\zeta'')} \leq
    \varepsilon_\wildsub m  \norm{\zeta''}_{K_c,\eta_c,\chi}^2 
    \leq \eps_\wildsub m \norm\zeta_{K_c,\eta_c,\chi}^2
  \end{gather}
  for all $\zeta = \zeta' + \zeta'' \in \smFBCtwo$ where $\zeta' \in \smfb
  1{q-1}(\cl K_c; L)$ and $\zeta'' \in \smbform{0,q}(\cl K_c; L)$.
\end{lemma}

\begin{proof}
  First the estimate for $\wCurv$ is considered.
  Recall that $\omega$ is the $(1,1)$-form associated to $g$.
  The Weak $\diff\dbar$-Lemma asserts that, for any
  $\varepsilon_\wildsub > 0$, there exists $\delta_c \in
  \kashf{}(X)$ such that
  \begin{equation} \label{eq:ineq-on-W} 
    - 2 \varepsilon_\wildsub \omega < 2 \iddbar \Re \hbar_{\delta_c} < 2
    \varepsilon_\wildsub \omega \quad\text{on }K_c \; .
  \end{equation}
  Such $\delta_c$ yields a hermitian metric $\eta_c$ on $L$ given the
  fixed choice of $\Hform_E$.
  Then, it follows from (\ref{eq:wild-curv-integral}) that, for any
  weight $\chiX$,
  \begin{equation*}
    \begin{aligned}
      - \eps_\wildsub \int_{K_c} e^{-\chiX} \Tr_{g,\eta_c} \pr_F
      \paren{g^\dual (\zeta \otimes \conj\zeta)} 
      \leq \wCurv (\zeta,\zeta) 
      \leq \eps_\wildsub \int_{K_c} e^{-\chiX}
      \Tr_{g,\eta_c} \pr_F \paren{g^\dual (\zeta \otimes \conj\zeta)}
    \end{aligned}
  \end{equation*}
  for any $\zeta = \zeta' + \zeta'' \in \smFBCtwo$
  ($\eps_\wildsub$ instead of $2\eps_\wildsub$ in the bounds because of
  the factor $\frac{1}{2}$ in $\omega = - \Im g = \frac{\cplxi}{2}
  \sum_{k,\ell} g_{k \conj \ell} dz^k \wedge d\conj{z^\ell}$).
  Note that
  \begin{equation*}
    \int_{K_c} e^{-\chiX} \Tr_{g,\eta_c} \pr_F \paren{g^\dual
    (\zeta \otimes \conj \zeta)}
    =  \norm{\zeta'}_{K_c,\eta_c,\chi}^2 + q
    \norm{\zeta''}_{K_c,\eta_c,\chi}^2 \leq q
    \norm\zeta_{K_c,\eta_c,\chi}^2
  \end{equation*}
  when $q \geq 1$. 
  When $q = 0$, the integral on the left hand side is
  zero, so the above inequality is still valid. 
  As a result, one obtains
  \begin{equation*}
    -\varepsilon_\wildsub q \norm\zeta_{K_c,\eta_c,\chi}^2 \leq 
    \wCurv(\zeta,\zeta) \leq
    \varepsilon_\wildsub q  \norm\zeta_{K_c,\eta_c,\chi}^2
  \end{equation*}
  and hence (\ref{eq:bound-wCurv}).

  For the estimate for $\wCurv'_F$, note that
  (\ref{eq:ineq-on-W}) implies
  \begin{equation*}
    -2\varepsilon_\wildsub \pr_F \omega < 2\iddbarb \Re
    \hbar_{\delta_c} < 2 \varepsilon_\wildsub \pr_F \omega
    \quad\text{on }K_c \; .
  \end{equation*}
  Then, one has $-\varepsilon_\wildsub m < 2\Tr_g\ddbarb\Re
  \hbar_{\delta_c} < \varepsilon_\wildsub m$ with the same
  $\varepsilon_\wildsub$ and $\delta_c$ as above.
  Therefore, it follows from (\ref{eq:wild-trace-integral}) that
  \begin{equation*}
    -\varepsilon_\wildsub m \norm{\zeta''}_{K_c,\eta_c,\chi}^2 \leq 
    \wCurv'_F(\zeta'',\zeta'') \leq
    \varepsilon_\wildsub m  \norm{\zeta''}_{K_c,\eta_c,\chi}^2
  \end{equation*}
  for any $\zeta'' \in \smfb 0q(\cl K_c;L)$, and hence
  (\ref{eq:bound-tr-wCurv}).
\end{proof}

\subsection{Existence of weak solutions on $K_c$}

With the bounds given in \S \ref{sec:bound-wCurv} for the wild
curvature terms, it is easy to follow the proof of Theorem
\ref{thm:proof-linearizable} and get the following
\begin{prop} \label{prop:proof-K_c}
  Suppose $L$ is a holomorphic line bundle on $X$ (which can possibly
  be non-linearizable), and suppose $q < s_F^-$ or $q > m -
  s_F^+$. 
  Then, there exists a suitable hermitian metric $g$ on $X$ such that
  the following holds:
  for any $0< c < \infty$, a hermitian metric $\eta_c$ on $L$ can be
  chosen such that, given any plurisubharmonic weight $\chiX$, the
  $L^2$ estimate
  \begin{equation*}
    \norm{S_q\zeta}_{K_c,\eta_c,\chi}^2 + \norm{T_{q-1}^*\zeta}_{K_c,\eta_c,\chi}^2
    \geq \tfrac{\pi}{4} \norm\zeta_{K_c,\eta_c,\chi}^2  
  \end{equation*}
  for all $\zeta \in \smFBCtwo \cap \Dom_{K_c,\eta_c,\chi} T_{q-1}^*$
  is satisfied.
%
\end{prop}

\begin{proof}

  Choose the translational invariant hermitian metric $g$ as described in
  the proof of Lemma \ref{lem:murakami-trick} for $q > m - s_F^+$ or Lemma
  \ref{lem:Murakami-trick2} for $q < s_F^-$, with $M = 2$.
  For the hermitian form $\Hform$ associated to $L$, choose $\Hform_E$
  as described in the proof of Lemma \ref{lem:bound-tame-part}.
  These choices are independent of $c$.
  
  Consider $K_c$ for some fixed $0 < c < \infty$.
  Take any $\varepsilon_\wildsub > 0$ such that
  \begin{equation} \label{eq:eps-choice} \tag{$*$}
    \varepsilon_\wildsub (q + m) \leq \frac{\pi}{4}
  \end{equation}
  and choose $\delta_c \in \kashf{}(X)$ according to Lemma
  \ref{lem:bounds-wild-curv} such that, for any given weight $\chiX$,
  the inequalities (\ref{eq:bound-wCurv}) and
  (\ref{eq:bound-tr-wCurv}) hold under the induced $L^2$-norm
  $\norm\cdot_{K_c,\eta_c,\chi}$.

  By the choices of the metrics, the conclusion of Corollary
  \ref{cor:BK-ineq-with-Murakami} when $q > m - s_F^+$ or Corollary
  \ref{cor:BK-ineq2-with-Murakami} when $q < s_F^-$, as well as that
  of Lemma \ref{lem:bound-tame-part}, holds for all $\zeta = \zeta'
  + \zeta'' \in \smFBCtwo \cap \Dom_{K_c,\eta_c,\chi} T_{q-1}^*$,
  where $\zeta' \in \smfb 1{q-1}(\cl K_c;L) \cap
  \Dom_{K_c,\eta_c,\chi}^{(1,q-1)} \dbarf^*$
  and $\zeta'' \in \smfb 0q(\cl K_c;L)$.

  Since $\chiX$ is plurisubharmonic, $\wtCurv(\zeta,\zeta) \geq 0$ for
  all $\zeta \in \smFBCtwo$ as in the proof of Theorem
  \ref{thm:proof-linearizable}.

  As a result, from Corollary \ref{cor:BK-ineq-with-Murakami} or
  \ref{cor:BK-ineq2-with-Murakami} as well as Lemma
  \ref{lem:bound-tame-part}, one obtains
  \begin{align*}
      &\phantom{\geq} \norm{S_q\zeta}_{K_c,\eta_c,\chi}^2 +
      \norm{T_{q-1}^*\zeta}_{K_c,\eta_c,\chi}^2 \\
      &\geq
      \begin{cases}
        \tfrac{\pi}{2} \norm\zeta_{K_c,\eta_c,\chi}^2 +
        \wCurv(\zeta,\zeta) & \text{for }q > m - s_F^+ \\
        \tfrac{\pi}{2} \norm\zeta_{K_c,\eta_c,\chi}^2 +
        \wCurv'_F(\zeta'',\zeta'') + \wCurv(\zeta,\zeta) & \text{for
        }q < s_F^-
      \end{cases}
      \\
      &
      \begin{aligned}[t]
        &\geq
        \tfrac{\pi}{2} \norm\zeta_{K_c,\eta_c,\chi}^2 - \eps_\wildsub
        \paren{m + q} \norm\zeta_{K_c,\eta_c,\chi} && 
        \text{\parbox{4cm}{\small by (\ref{eq:bound-wCurv}) and
            (\ref{eq:bound-tr-wCurv}), and $\eps_\wildsub q <
            \eps_\wildsub (m+q)$}}
        \\
        &\geq \tfrac{\pi}{4} \norm\zeta_{K_c,\eta_c,\chi}^2 &&
        \text{by (\ref{eq:eps-choice})} \; .
      \end{aligned}
  \end{align*}
  This gives the required $L^2$ estimate.
\end{proof}

Since, for any $\psi \in \kashf{0,q}(X;L)$, one has $\psi|_{K_c} \in
\hilbFB 0q(K_c;L)$ (unweighted) for any $0 < c < \infty$, it follows
the following corollary of Propositions
\ref{prop:estimate-imply-strong-sol} and \ref{prop:proof-K_c}.
\begin{cor} \label{cor:sol-on-K-nu}
  Consider the exhaustive sequence $\seq{K_\nu}_{\nu \in \Nnum_{>0}}$ of
  relatively compact open subsets of $X$. 
  Suppose $q < s_F^-$ or $q > m - s_F^+$.
  Then one can choose a suitable hermitian metric $g$ on $X$ and a
  sequence of hermitian metrics $\seq{\eta_\nu}_{\nu \in \Nnum_{>0}}$
  on $L$ as in Proposition \ref{prop:proof-K_c} such that, for any
  $\psi \in \kashf{0,q}(X;L) \cap \ker\dbar$, 
  there exists a sequence of solutions $\seq{\xi'_\nu}_{\nu \in \Nnum_{>0}}$
  such that $\xi'_\nu \in \hilbbsp{0,q-1}_{\eta_\nu}(K_\nu;L)$ (unweighted) and
  $\dbar\xi'_\nu = \psi|_{K_\nu}$ in $\hilbFB 0q_{\eta_\nu}(K_\nu;L)$.
\end{cor}

\begin{remark}
  Since $\chiX$ has to be smooth on a neighborhood of $\cl K_c$ (as
  required by \cite{Hoermander-L2}*{Prop.~2.1.1} so that $\smFBCtwo
  \cap \Dom T_{q-1}^*$ is dense in $\Dom T_{q-1}^* \cap \Dom S_q$
  under the suitable graph norm), if $\psi \in \kashf{0,q}(K_c;L)$,
  there may not exist such $\chiX$ such that $\norm\psi_{K_c,\chi} <
  \infty$. 
  To avoid technical difficulty, the author does not attempt to solve
  the $\dbar$-equation for any $\psi \in \kashf{0,q}(K_c;L)$ such that
  $\dbar\psi = 0$
  by means of $L^2$ estimates directly.
\end{remark}

\subsection{A Runge-type approximation}
\label{sec:runge-type-approx}

This section is devoted to proving a Runge-type approximation which is
required to construct a global solution to the equation $\dbar \xi =
\psi$ from the solutions on $K_\nu$'s given in Corollary
\ref{cor:sol-on-K-nu}.

In what follows, $q$ is assumed to be $0 < q < s_F^-$ or $q > m - s_F^+$,
and the hermitian metric $g$ as well as the family of hermitian
metrics $\seq{\eta_c}_{c > 0}$ as asserted by Proposition
\ref{prop:proof-K_c} is fixed.
Then, according to the choices of the $\eta_c$'s in the proof of
Proposition \ref{prop:proof-K_c}, for any $c', c > 0$, one has
\begin{equation*}
  \eta_c = \eta_{c'} e^{2\Re(\delta_{c'} - \delta_c)} =: \eta_{c'}
  e^{\Deltab_{c'c}} \; .
\end{equation*}
Note that $e^{\Deltab_{c'c}} > 0$ on $X$.
It is understood that the hermitian metric $\eta_c$ on $L$ is chosen
when the $L^2$-norm on $K_c$ is considered, so
write $\hilbFB 0q_{\eta_c,\chi}(K_c;L)$ as $\hilbFB 0q_{\chi}(K_c;L)$,
 $\inner\cdot\cdot_{K_c,\eta_c,\chi}$ as $\inner\cdot\cdot_{K_c,\chi}$ and so on
 to simplify notation.
When the weight $\chiX$ is absent from the notation,
e.g.~$\hilbFB 0q(K_c;L)$ or $\inner\cdot\cdot_{K_c}$, it is understood that
the corresponding object is unweighted, i.e.~$\chiX = 0$.

For any finite $c' > c > 0$ and for any $\Psi \in \hilbFB 0{q-1}(K_c;L)$, if
$\Psi$ is extended by zero to a section in $\hilbFB 0{q-1}(K_{c'};L)$,
then it follows that
\begin{equation}
  \label{eq:inner-prod-change}
  \inner{\zeta}{\Psi}_{K_c} =\inner{\zeta}{\Psi \addD{e^{\Deltab_{c'c}}}}_{K_{c'}}
\end{equation}
for any $\zeta \in \hilbFB 0{q-1}(K_{c'};L)$.

Define $\parres{\ker_{K_{c'}} T_{q-1}}_{K_c}$ to be
the image of $\ker_{K_{c'}} T_{q-1}$ under the restriction map
$\hilbFB 0{q-1}(K_{c'};L) \to \hilbFB 0{q-1}(K_c;L)$.
Note that $T_{q-1}$ commutes with the restriction map (as $c >
0$), so one has
\begin{equation*}
  \parres{\ker_{K_{c'}} T_{q-1}}_{K_c} \subset
  \ker_{K_c} T_{q-1} \; .
\end{equation*}


The following proof of the required Runge-type approximation is 
an analogue of the one for strongly pseudoconvex manifolds
given in \cite{Hoemander}*{Lemma 4.3.1}. 



\begin{prop} \label{prop:runge-type-approx}
  Suppose $0 < q < s_F^-$ or $q > m - s_F^+$, and $g$ and
  $\eta_c$'s are chosen according to Proposition \ref{prop:proof-K_c}. 
  Then, for any finite $c' > c > 0$,
  the closure of $\parres{\ker_{K_{c'}} T_{q-1}}_{K_c}$ in
  $\hilbbsp{0,q-1}(K_c;L)$ is $\ker_{K_c} T_{q-1}$.
  In other words, $\parres{\ker_{K_{c'}} T_{q-1}}_{K_c}$ is dense
  in $\ker_{K_c} T_{q-1}$. 
\end{prop}

\begin{proof}
  By virtue of the Hahn-Banach theorem, it suffices to show that for
  every $\Psi \in \hilbbsp{0,q-1}(K_c;L)$, if the induced bounded
  linear functional
  \begin{equation*}
    \hilbbsp{0,q-1}(K_c;L) \ni \zeta \mapsto \inner\zeta \Psi_{K_c}
  \end{equation*}
  vanishes on $\parres{
    \ker_{K_{c'}} T_{q-1}}_{K_c}$, then it also vanishes on
  $\ker_{K_c} T_{q-1}$. \footnote{If there
  exists  
  $\zeta \in \ker_{K_c} T_{q-1}$ which does not lie in the closure of
  $\parres{
    \ker_{K_{c'}}
    T_{q-1}}_{K_c}$ in $\hilbbsp{0,q-1}(K_c;L)$, then the Hahn-Banach theorem
  asserts that there is a bounded linear functional $\Lambda$ such
  that $\parres{
    \ker_{K_{c'}} T_{q-1}}_{K_c} \subset \ker \Lambda$ and
  $\Lambda\zeta = 1$.}

  Suppose that $\Psi \in \hilbFB 0{q-1}(K_c;L)$ satisfies the above assumption.
  Extend $\Psi$ by zero to $K_{c'}$ as a section in
  $\hilbbsp{0,q-1}(K_{c'};L)$.
  Now it suffices to show that there
  exists $\Xi \in \hilbFBtwoLong[']$ such that $\Xi \equiv 0$ on
  $K_{c'}\setminus \cl K_c$ and
  \begin{equation} \label{eq:solve-T-star-eq-aim} \tag{$\dagger$}
    \inner{\zeta}{\Psi
      \addD{e^{\Deltab_{c'c}}}}_{K_{c'}} = \inner{T_{q-1} \zeta}{\Xi}_{K_{c'}}
  \end{equation}
  for any $\zeta \in \Dom_{K_{c'}} T_{q-1}$, which then implies that
  \begin{equation} \label{eq:solve-T-star-eq} \tag{$\ddagger$}
    \inner{\zeta}{\Psi}_{K_c} = \inner{T_{q-1} \zeta}{\Xi e^{-\Deltab_{c'c}}}_{K_c}
  \end{equation}
  for any $\zeta \in \Dom_{K_{c'}} T_{q-1}$ due to (\ref{eq:inner-prod-change}).
  \addD{The equality (\ref{eq:solve-T-star-eq})
    holds true for $\zeta \in \smfb 0{q-1}_0(K_{c'};L)$ in particular,
    and $\smfb 0{q-1}(\cl K_c; L)$ is dense in $\Dom_{K_c} T_{q-1}$
    under the graph norm $\sqrt{\norm\zeta_{K_c}^2 + \norm{T_{q-1}\zeta}_{K_c}^2}$
    by \cite{Hoermander-L2}*{Prop.~2.1.1}, so} (\ref{eq:solve-T-star-eq})
  also holds true for $\zeta \in \Dom_{K_c} T_{q-1}$. 
%
%
  It follows that
  \begin{equation*}
    \inner{\zeta}{\Psi}_{K_c} = \inner{T_{q-1} \zeta}{\Xi
      e^{-\Deltab_{c'c}}}_{K_c} = 0
  \end{equation*}
  for all 
  $\zeta \in \ker_{K_c} T_{q-1} \subset \Dom_{K_c} T_{q-1}$ as
  required.
  It remains to show the existence of such $\Xi$.

  Take a sequence of smooth convex increasing functions $\chiR_\nu
  \colon \fieldR \to \fieldR$ such that $\chiR_\nu(x) = 0$ for all $x
  \leq c$, and $\chiR_\nu(x) \nearrow + \infty$ as $\nu
  \tendsto \infty$ for every $x > c$. Note that $\chiR_\nu \geq 0$ for any $\nu \geq 0$
  by such choice. Set $\chiX_\nu := \chiR_\nu \circ \varphi$ as
  before. 
  A sequence of weighted norms $\norm\cdot_{c',\nu} := \norm\cdot_{K_{c'},
    \chi_\nu}$ on $K_{c'}$ is then defined. 
  Let the
  corresponding inner products, Hilbert spaces and $\Dom$ also be
  distinguished by using the subscripts
  $c',\nu$, 
  and the corresponding adjoint of $T_{q-1}$ 
  by $T_{q-1}^{*,\nu}$.

  For any $q$ in the given range, the $L^2$ estimate
  in Proposition \ref{prop:proof-K_c} holds under each of the above weighted
  norms with $T_{q-1}^*$ replaced by $T_{q-1}^{*,\nu}$. 
  Since $\inner\zeta{\Psi \addD{e^{\Deltab_{c'c}}}e^{\chiX_\nu}}_{c',\nu} =
  \inner\zeta{\Psi\addD{e^{\Deltab_{c'c}}}}_{K_{c'}}$ and the right
  hand side vanishes for all $\zeta
  \in \addD{\ker_{K_{c'}} T_{q-1} = \ker_{c',\nu} T_{q-1}}$ by
  the assumption on $\Psi$, it follows that
  \begin{equation*}
    \Psi \addD{e^{\Deltab_{c'c}}} e^{\chiX_\nu}
    \in \paren{\ker_{c',\nu} T_{q-1}}^\bot = \cl{\im_{c',\nu}
      T_{q-1}^{*,\nu}} \; .
  \end{equation*}
  Given the $L^2$ estimate, Theorem
  \ref{thm:existence-weak-solution}~(\ref{item:solve-T-star-eq}) then
  asserts that there exists $\wXi^\nu \in \Dom_{c',\nu} T_{q-1}^{*,\nu}$
  such that $T_{q-1}^{*,\nu}\wXi^\nu = \Psi e^{\Deltab_{c'c}} e^{\chiX_\nu}$. 
%
  Therefore, one has
  \begin{equation*}
    \begin{aligned}
      \inner\zeta {\Psi e^{\Deltab_{c'c}} e^{\chiX_\nu}}_{c',\nu} &=
      \inner\zeta {T_{q-1}^{*,\nu} \wXi^\nu}_{c',\nu} \\ &=
      \inner{T_{q-1}\zeta} {\wXi^\nu}_{c',\nu} = \inner{T_{q-1}\zeta}
      {\wXi^\nu e^{-\chiX_\nu}}_{K_{c'}}
    \end{aligned}  
  \end{equation*}
  for all $\nu \in \Nnum$ and for all $\zeta \in \Dom_{c',\nu} T_{q-1}
  = \Dom_{K_{c'}} T_{q-1}$.
  By defining $\Xi^\nu := \wXi^\nu
  e^{-\chiX_\nu}$, one obtains
  \begin{equation} \tag{$*$} \label{eq:dbarf-star-eq}
    \inner\zeta {\Psi e^{\Deltab_{c'c}}}_{K_{c'}} = \inner{T_{q-1}\zeta}
    {\Xi^\nu}_{K_{c'}} \; .
  \end{equation}
  Moreover, notice that the constant in the $L^2$ estimate is
  independent of $\nu$ (which is chosen to be $\frac{\pi}{4}$ in
  Proposition \ref{prop:proof-K_c}).
  The estimate on the solution $\wXi^\nu$ from Theorem
  \ref{thm:existence-weak-solution}~(\ref{item:solve-T-star-eq}) then
  implies that
  \begin{equation} \tag{$**$} \label{eq:get-bound}
    \tfrac{\pi}{4} \int_{K_{c'}} \abs{\Xi^\nu}_{g,\eta_{c'}}^2 e^{\chiX_\nu}
    \leq \int_{K_{c'}} \abs{\Psi e^{\Deltab_{c'c}}}_{g,\eta_{c'}}^2
    e^{\chiX_\nu}
    = \int_{K_c} \abs{\Psi}_{g,\eta_c}^2 e^{\Deltab_{c'c}} e^{\chiX_\nu}
    \; ,
  \end{equation}
  where the last equality is due to the fact that $\Psi$ vanishes on
  $K_{c'} \setminus \cl K_c$.
  Since $\chiR_\nu(\varphi)$ is independent of $\nu$ when $\varphi
  \leq c$, the integral on the right hand side is independent of
  $\nu$, so the left hand side is a bounded sequence in $\nu$.
  This in turn implies that there exists a subsequence of
  $\seq{\Xi^\nu}_{\nu\in\Nnum}$ which converges to some $\Xi \in
  \hilbFBtwoLong[']$ (unweighted) in the weak topology.
  From (\ref{eq:get-bound}), since $\chiR_\nu(\varphi) \nearrow
  +\infty$ for $\varphi > c$, it follows that $\Xi \equiv 0$ when
  $\varphi > c$, i.e.~on $K_{c'} \setminus \cl K_c$.
  Moreover, from (\ref{eq:dbarf-star-eq}) it follows that
  (\ref{eq:solve-T-star-eq-aim}) holds
  for all $\zeta\in \Dom_{K_{c'}} T_{q-1}$. 
  This is what is desired.
\end{proof}

\subsection{Proof of Theorem $\text{\ref{thm:main_thm}}$ for general $L$}
\label{sec:proof-special-q}

First notice that,
if $q = 0 < s_F^-$, then the $L^2$ estimate in Proposition
\ref{prop:proof-K_c} holds when the metrics are chosen suitably, and
thus for any $\psi \in \kashf{}(X;L) \cap \ker\dbar$ one has
\begin{equation*}
  0 = \norm{\dbar\psi}_{K_c}^2 \geq \tfrac{\pi}{4} \norm\psi_{K_c}^2
\end{equation*}
(note that $T_{-1}^*\zeta = 0$ for all $\zeta \in \smform{}(\cl K_c;L)$).
This means that $\psi|_{K_c} = 0$ for any $c > 0$, and thus $\psi = 0$
on $X$.
Therefore, one has the following
\begin{thm}
  If $s_F^- > 0$, one has $H^0(X,L) = 0$.
\end{thm}

Assume $0 < q < s_F^-$ or $q > m - s_F^+$ in what follows.
The metrics $g$ and $\eta_\nu$'s from Corollary \ref{cor:sol-on-K-nu}
are fixed for this section.
Again, write $\hilbFB 0q_{\eta_\nu,\chi}(K_\nu;L)$ as $\hilbFB
0q_{\chi}(K_\nu;L)$ and so on, and notations like $\hilbFB 0q(K_c;L)$ or
$\norm\cdot_{K_c}$ are understood as unweighted objects, i.e.~$\chiX =
0$.

For every integer $\nu \geq 1$, as $\delta_{\nu + 1} - \delta_\nu$ is
smooth on $X$ and $\cl K_{\nu+1}$ is compact, there exists a constant
$M_{\nu+1}' \geq 1$ such that
\begin{equation}
  \label{eq:expand-constant}
  \norm\zeta_{K_\nu} \leq M_{\nu+1}' \norm\zeta_{K_{\nu+1}}
\end{equation}
for all $\zeta \in \hilbFB 0q(K_{\nu+1};L)$. 
Define also $M_1 := 1$ and $M_\nu := \prod_{k
= 2}^\nu M_k'$ for $\nu \geq 2$.

Proposition \ref{prop:runge-type-approx} is used to complete the proof
of Theorem \ref{thm:main_thm}. 
The following argument is adopted from \cite{Grauert&Remmert}*{Ch.~IV, \S 1,
  Thm.~7}.

\begin{thm}
  Suppose $0 < q < s_F^-$ or $q > m - s_F^+$. 
  Then one has $H^q(X,L) = 0$ for any $q$ in the given range.
\end{thm}
\begin{proof}
  Given any $\psi \in \kashf{0,q}(X;L) \cap \ker \dbar$, Corollary
  \ref{cor:sol-on-K-nu} provides a sequence of local solutions
  $\seq{\xi_\nu'}_{\nu \geq 1}$ such that $\xi_\nu' \in
  \hilbbsp{0,q-1}(K_\nu;L)$ and $\dbar\xi_\nu' = \psi|_{K_\nu}$
  for all integers $\nu \geq 1$. 
  First a sequence of local solutions $\seq{\xi_\nu}_{\nu \geq 1}$
  such that $\xi_\nu \in \hilbFB 0{q-1}(K_\nu;L)$,
  $\dbar \xi_\nu = \psi|_{K_\nu}$
  and 
  \begin{equation} \tag{$*$}
    \label{eq:new-sol-properties}
    \norm{\xi_{\nu+1} - \xi_\nu}_{K_\nu} < \frac{1}{M_\nu 2^\nu} 
  \end{equation}
  for all $\nu \geq 1$ is defined inductively as follows. 
  Set $\xi_1 := \xi_1'$. 
  Suppose $\xi_1, \dots, \xi_\nu$ are defined for some $\nu
  \geq 1$. 
  Let $\gamma_\nu' := \xi_{\nu + 1}'|_{K_\nu} - \xi_\nu$. 
  Notice that $\gamma_\nu' \in \ker_{K_\nu} T_{q-1} \subset \hilbFB
  0{q-1}(K_{\nu};L)$. 
  Proposition \ref{prop:runge-type-approx} then
  implies that there exists $\gamma_\nu \in \ker_{K_{\nu+1}} T_{q-1}
  \subset \hilbFB 0{q-1}(K_{\nu+1};L)$
  such that
  \begin{equation*}
    \norm{\gamma_\nu' - \gamma_\nu}_{K_\nu} < \frac{1}{M_\nu 2^\nu} \;
    .
  \end{equation*}
  Set $\xi_{\nu+1} := \xi_{\nu+1}' - \gamma_\nu$.
  Then one has $\dbar \xi_{\nu+1} = \dbar \xi_{\nu+1}' =
  \psi|_{K_{\nu+1}}$ and the inequality (\ref{eq:new-sol-properties})
  is satisfied. 
  The required sequence $\seq{\xi_\nu}_{\nu\geq 1}$ is therefore
  defined.

  Notice that, for every $\nu \geq 1$, the sequence
  $\seq{\xi_\mu|_{K_\nu}}_{\mu\geq \nu}$ converges in
  $\hilbbsp{0,q-1}(K_\nu;L)$. Indeed, for any $\mu \geq \nu \geq 1$ and for
  any integer $k > 0$,
  \begin{equation*}
    \begin{aligned}
      \norm{\xi_{\mu + k} - \xi_\mu}_{K_\nu} 
      &\leq \sum_{r=0}^{k-1} \norm{\xi_{\mu + r+1} - \xi_{\mu + r}}_{K_\nu} \\
      &\leq \sum_{r=0}^{k-1} \frac{M_{\mu+r}}{M_\nu} \norm{\xi_{\mu +
          r+1} - \xi_{\mu + r}}_{K_{\mu + r}} && \text{by
        (\ref{eq:expand-constant})} \; , \\
      &\leq \frac{1}{M_\nu} \sum_{r=0}^{k-1} \frac{1}{2^{\mu+r}}
      &&\text{by (\ref{eq:new-sol-properties})} \; , \\
      &\leq \frac{1}{M_\nu 2^{\mu-1}} \; , 
    \end{aligned}
  \end{equation*}
  which tends to $0$ as $\mu \tendsto \infty$, so
  $\seq{\xi_\mu|_{K_\nu}}_{\mu\geq \nu}$ is a Cauchy sequence in
  $\hilbbsp{0,q-1}(K_\nu;L)$. 
  Let $\xi^{(\nu)}$ be the limit of $\seq{\xi_\mu|_{K_\nu}}_{\mu \geq \nu}$ in
  $\hilbbsp{0,q-1}(K_\nu;L)$. 
  Since $\dbar \xi_\mu|_{K_\nu} = \psi|_{K_\nu}$ for all $\mu \geq
  \nu$, and $\dbar$ is a closed operator, one has $\dbar \xi^{(\nu)} =
  \psi|_{K_\nu}$ for all $\nu \geq 1$. 
  Now notice that restriction from $K_{\nu + 1}$ to $K_\nu$ is
  continuous by (\ref{eq:expand-constant}), so
  \begin{equation*}
    \xi^{(\nu+1)}|_{K_\nu} - \xi^{(\nu)} 
    = \lim_{\substack{\mu \geq \nu+1 \\ \mu \tendsto \infty}}
    \paren{\xi_\mu|_{K_\nu} - \xi_\mu|_{K_\nu}} = 0
  \end{equation*}
  in $\hilbbsp{0,q-1}(K_\nu;L)$. 
  On every $K_\nu$, different choices of
  $\delta_\nu \in \kashf{}(X)$ yield equivalent norms. 
  Therefore, by fixing one $\delta \in \kashf{}(X)$, one can consider
  $\Ltwo{0,q-1}(X;L;\text{loc})$, the space of locally $L^2$
  $L$-valued $(0,q-1)$-forms on $X$, and there exists $\xi' \in
  \Ltwo{0,q-1}(X;L;\text{loc})$ such that 
  \begin{equation*}
    \begin{aligned}
      &\xi'|_{K_\nu} = \xi^{(\nu)} &&\text{for all }\nu \geq 1
      \; ,  \text{ and} \\
      &\dbar \xi' = \psi &&\text{in }\Ltwo{0,q-1}(X;L;\text{loc})
      \; .
    \end{aligned}
  \end{equation*}
  Remark \ref{rem:regularity} then assures that there exists $\xi \in
  \kashf{0,q-1}(X;L)$ such that $\dbar \xi = \psi$ on $X$.

  Since $\psi \in \kashf{0,q}(X;L) \cap \ker \dbar$ is arbitrary, this
  shows that $H^q(X,L) = 0$. This completes the proof.
\end{proof}

}  



\newpage
\begin{multicols}{2}
  \printnomenclature
\end{multicols}

\end{document}